\documentclass[11pt]{amsart}
\usepackage{amsbsy,amssymb,amscd,amsfonts,latexsym,amstext,delarray, amsmath,color,caption}
\usepackage{tikz,lscape}
\usetikzlibrary{matrix,arrows}
\usepackage{graphicx}
\graphicspath{ {Home/Users/Fw18twelve/Desktop/Vanderbilt_Fall_2019/Research/Cubic/} }
\usepackage{multicol}
\usepackage{scrextend}
\usepackage{mathtools}
\usepackage{amssymb} 
\usepackage{calligra}
\usepackage{calrsfs}
\usepackage{eucal}
\usepackage{calrsfs}
\usepackage{epsfig}
\usepackage{xcolor,import}
\usepackage{transparent}
\usepackage{tikz-cd}
\usepackage{epstopdf}

\DeclareMathAlphabet{\pazocal}{OMS}{zplm}{m}{n}

\input xy

\xyoption{all}
    \advance \topmargin by -\headheight
    \advance \topmargin by -\headsep

    \evensidemargin \oddsidemargin
\newtheorem {theorem}    {Theorem}[section]

\theoremstyle{definition}
\newtheorem {lemma}      [theorem]    {Lemma}
\newtheorem {corollary}  [theorem]    {Corollary}

\theoremstyle{definition}

\theoremstyle{remark}

\newcommand{\defeq}{\vcentcolon=}
\newcommand{\gen}[1]{\langle #1 \rangle}

\def\a{\alpha}                
\def\b{\beta}

\def\eps{\varepsilon}

\def\a{\alpha}
\def\b{\beta}

\def\Z{{\mathbb Z}}

\def\Z{\mathbb{Z}}     
\def\N{\mathbb{N}}     
\def\lab{{\text{Lab}}}

\def\vertexradius{.1}
\def\vertex(#1){\fill (#1) circle (\vertexradius)}

\begin{document}

\title{\bf Torsion Subgroups of Groups with Cubic Dehn Function}
\maketitle
\begin{center}

Francis Wagner

\end{center}

\bigskip

\begin{center}

\textbf{Abstract}

\end{center}

We construct the first examples of finitely presented groups with cubic Dehn function containing a finitely generated infinite torsion subgroup. Moreover, we show that any infinite free Burnside group with sufficiently large odd exponent can be embedded as a subgroup of a finitely presented group with cubic Dehn function.

\bigskip


\section{Introduction}

Let $\pazocal{A}$ be an alphabet and $\pazocal{R}$ be a set of reduced words in the alphabet $\pazocal{A}\cup\pazocal{A}^{-1}$. Letting $F(\pazocal{A})$ be the free group with basis $\pazocal{A}$, define the \textit{normal closure} of $\pazocal{R}$ in $F(\pazocal{A})$, denoted $\gen{\gen{\pazocal{R}}}$, to be the smallest normal subgroup of $F(\pazocal{A})$ containing $\pazocal{R}$. It is easy to see that this subgroup exists and is generated by the set of reduced words of the form $fRf^{-1}$, where $f\in F(\pazocal{A})$ and $R\in\pazocal{R}$. The group $F(\pazocal{A})/\gen{\gen{\pazocal{R}}}$ is then denoted by $\gen{\pazocal{A}\mid\pazocal{R}}$.

Given a group $G$ isomorphic to $F(\pazocal{A})/\gen{\gen{\pazocal{R}}}$, it is convenient to view $G$ as being generated by $\pazocal{A}$, so that elements of $G$ can be represented by reduced words over $\pazocal{A}$. With this interpretation, $G$ is said to have \textit{presentation} $\gen{\pazocal{A}\mid\pazocal{R}}$. It follows immediately that a reduced word $W$ in the alphabet $\pazocal{A}\cup\pazocal{A}^{-1}$ represents the identity in $G$ if and only if there exist some $k\in\N$, $f_1,\dots,f_k\in F(\pazocal{A})$, $R_1,\dots,R_k\in\pazocal{R}$, and $\eps_1,\dots,\eps_k\in\{\pm1\}$ such that $W=\prod\limits_{i=1}^k f_iR_i^{\eps_i}f_i^{-1}$ in $F(\pazocal{A})$. If $W=1$ in $G$, then its \textit{area}, $\text{Area}(W)$, is the minimal value of $k$ so that there exists a representation of $W$ in $F(\pazocal{A})$ as above.

Alternatively, given a group $G$ with presentation $\pazocal{P}$, the area of a word $W$ representing the identity in $G$ can be defined as the minimal area of a van Kampen diagram $\Delta$ over $\pazocal{P}$ (see Section 2.1) such that $\lab(\partial\Delta)\equiv W$, where $\equiv$ represents (here and throughout the rest of this paper) letter-for-letter equality.

If $\pazocal{A}$ and $\pazocal{R}$ are both finite, then the group $G$ is called \textit{finitely presented}. First introduced in [9], the \textit{Dehn function} of the group $G$ with respect to its finite presentation $\pazocal{P}=\gen{\pazocal{A}\mid\pazocal{R}}$ is the function $\delta_{\pazocal{P}}:\N\to\N$ defined by $\delta_{\pazocal{P}}(n)=\max\{\text{Area}(W):|W|_{\pazocal{A}}\leq n\}$.

Dehn functions are defined up to an asymptotic equivalence $\sim$ taken on functions $\N\to\N$ defined by $f\sim g$ if and only if $f\preccurlyeq g$ and $g\preccurlyeq f$, where $f\preccurlyeq g$ if and only if there exists a $C>0$ such that $$f(n)\leq Cg(Cn)+Cn+C$$ for all $n\in\N$. Given a finitely presented group $G$ with finite presentations $\pazocal{P}$ and $\pazocal{S}$, it is a simple exercise to show that $\delta_{\pazocal{P}}\sim\delta_{\pazocal{S}}$. So, given a finitely presented group $G$, one can define the \textit{Dehn function} of $G$, $\delta_G$, as the Dehn function of any of its finite presentations.

The Dehn function is a useful invariant for studying finitely presented groups. Two of many examples of how this is so are as follows.

\begin{addmargin}[1em]{0em}

(1) The Dehn function is closely related to the solvability of the word problem in the group, as smaller Dehn functions correspond to groups with more tractable word problems [27]. 

(2) If $G$ is the fundamental group of a compact Riemannian manifold $M$, then $\delta_G$ is equivalent to the smallest isoperimetric function of the universal cover $\tilde{M}$.

\end{addmargin}

Note that under the equivalence relation $\sim$, all polynomial functions of degree $d$ are equivalent to one another. Because of this, it makes sense to consider groups of linear Dehn function, groups of quadratic Dehn function, etc. A finitely presented group is word hyperbolic in the sense of Gromov if and only if its Dehn function is linear [9]. Moreover, any finitely presented group $G$ satisfying $\delta_G\prec n^2$ is word hyperbolic [9], [3], [16]. This `gap' in possible Dehn functions leads naturally to the question of what properties satisfied by hyperbolic groups are satisfied by groups with quadratic Dehn function.

For example, hyperbolic groups are known to have solvable conjugacy problem [9], while in the early 1990s Rips posed the question of the solvability of the conjugacy problem in groups with quadratic Dehn function. In 2018 Olshanskii and Sapir [26] answered this problem in the negative, exhibiting groups with quadratic Dehn function and unsolvable conjugacy problem. A problem arising in a similar manner (using methods similar to those used in [26]) is what is addressed in this paper.

A class of groups $\pazocal{V}$ is called a \textit{variety} if it is closed under subgroups, direct products, and homomorphic images. Equivalently, a variety is a class of groups defined by a set of \textit{group laws}, i.e equations of the form $v(x_1,\dots,x_k)=1$ that hold in the group when any group elements $g_1,\dots,g_k$ are substituted for $x_1,\dots,x_k$. For example, the class of abelian groups is a variety defined by the group law $x_1x_2x_1^{-1}x_2^{-1}=1$.

Given an alphabet $\pazocal{A}$, define the \textit{verbal subgroup} of $F(\pazocal{A})$, $\pazocal{V}(\pazocal{A})$, as the subgroup generated by all possible values of the group laws defining $\pazocal{V}$. (For example, if $\pazocal{V}$ is the variety of abelian groups, $\pazocal{V}(\pazocal{A})$ is the derived subgroup of $F(\pazocal{A})$). It is easy to see that $\pazocal{V}(\pazocal{A})$ is a normal subgroup of $F(\pazocal{A})$. Then define $F_{\pazocal{V}}(\pazocal{A})=F(\pazocal{A})/\pazocal{V}(\pazocal{A})$ as the \textit{free group relative to the variety $\pazocal{V}$}. This name is justified by the universal property of relatively free groups: If $G\in\pazocal{V}$ and is generated by the set $\{g_i\}_{i\in I}$, then for $\pazocal{A}=\{a_i\}_{i\in I}$, there exists an epimorphism $\phi:F_{\pazocal{V}}(\pazocal{A})\to G$ such that $\phi(a_i\pazocal{V}(\pazocal{A}))=g_i$ for all $i\in I$.

The varieties of interest in this paper are the Burnside varieties $\pazocal{B}_n$ for $n\geq2$, where $\pazocal{B}_n$ is defined by the group law $x^n=1$. For simplicity of notation, the \textit{free Burnside group} (or the free group relative to the Burnside variety $\pazocal{B}_n$) is denoted $B(\pazocal{A},n)$, or simply $B(m,n)$ if $|\pazocal{A}|=m$.

The Burnside problem, dating back to 1902, asked whether or not there exists a finitely generated infinite torsion group. Although the problem was solved in the affirmative by Golod and Shaferevich in 1964 [8], the group constructed did not have finite exponent. This led to the Bounded Burnside problem, which essentially asked whether or not $B(m,n)$ is infinite for some $m,n\in\N$ (and, if so, for which choices of $m,n$). Novikov and Adian were the first to give examples of $m,n$ such that $B(m,n)$ is infinite and has solvable word problem, specifically for all $m>1$ and $n\geq4381$ odd [14]. Adian later improved the bound on $n$ to $n\geq665$ in 1978 [1]. In 1982, Olshanskii provided a simpler geometric proof that $B(m,n)$ is infinite for $m>1$ and $n\geq10^{10}$ odd, as well as proving the existence of the so-called Tarski monster groups [15].

It is established in each proof that these infinite torsion groups cannot be finitely presented, i.e they cannot be presented by a finite number of relations. As such, one cannot speak of the Dehn function of $B(m,n)$ for sufficiently large $n$. Since the word problem is solvable in this group, though, a question arises as to whether these free Burnside groups can be embedded into finitely presented groups with `small' Dehn functions.

Ghys and de la Harpe proved in 1991 that no hyperbolic group contains an infinite torsion subgroup [7]. In particular, this means that the group $B(m,n)$ for $m>1$ and sufficiently large odd $n$ cannot be embedded into a finitely presented group with linear Dehn function. On the other hand, Olshanskii and Sapir exhibited in 2000 such an embedding into a finitely presented group $G$ satisfying $\delta_G\preccurlyeq n^{10}$ [20].

Using a similar construction as was used in [19] and [26] and the geometric methods of Olshanskii's 1982 proof of the infiniteness of $B(m,n)$ for sufficiently large $n$, we prove the following in this paper:

\begin{theorem} \label{main theorem}

For all $n>10^{10}$ odd, there exists a finitely presented group $G_n$ with cubic Dehn function into which the free Burnside group $B(2,n)$ embeds. In particular, there exist finitely presented groups $G$ satisfying $\delta_G\sim n^3$ and containing finitely generated infinite torsion subgroups.

\end{theorem}

For $\pazocal{A}=\{a_i\}_{i=1}^\infty$, Shirvanyan proved that $B(\pazocal{A},n)$ embeds in the group $B(2,n)$ for $n$ as in the theorem [29]. Hence, Theorem \ref{main theorem} immediately implies the following corollary:

\begin{corollary}

For all $n>10^{10}$ odd, there exists a finitely presented group $G_n$ with cubic Dehn function into which the free Burnside groups $B(m,n)$ for $m\geq2$ embed.

\end{corollary}

It is the goal of the author to replace Theorem 1.1 in the near future with the `optimal' result, i.e an embedding of $B(2,n)$ in a group with quadratic Dehn function, solving the problem raised by Olshanskii.

\bigskip


\section{Maps and diagrams}

A main tool used through the remainder of this paper is the concept of van Kampen diagrams over group presentations that was introduced by its namesake in 1933 [30]. It is assumed that the reader is intimately acquainted with this concept. The following subsection functions to recall the most important definitions; for further reference, see [17], [13], and [28].

\subsection{van Kampen diagrams} \

Let $G$ be a group with presentation $\gen{\pazocal{A}\mid\pazocal{R}}$. Suppose $\Delta$ is an oriented 2-complex homeomorphic to a disk equipped with a \textit{labelling} function, i.e a function $\lab:E(\Delta)\to\pazocal{A}\cup\pazocal{A}^{-1}\cup\{1\}$ which satisfies $\lab(e^{-1})\equiv\lab(e)^{-1}$ for any edge $e\in E(\Delta)$ (with, of course, $1^{-1}\equiv1$). The label of a path in $\Delta$ is defined in the obvious way, that is $\lab(e_1\dots e_n)\equiv\lab(e_1)\dots\lab(e_n)$. For any edge $e$ in $\Delta$, $e$ is called a $0$-edge if $\lab(e)\equiv1$; otherwise, $e$ is called an $\pazocal{A}$-edge.

Finally, suppose that for each cell $\Pi$ of $\Delta$, one of the following is true:

\begin{addmargin}[1em]{0em}

(1) omitting the label of any zero edges, $\lab(\partial\Pi)$ is visually equal to a cyclic permutation of $R^{\pm1}$ for some $R\in\pazocal{R}$ 

(2) $\partial\Pi$ consists of $0$-edges and exactly two $\pazocal{A}$-edges $e$ and $f$, with $\lab(e)\equiv\lab(f^{-1})$ 

(3) $\partial\Pi$ consists only of $0$-edges. 

\end{addmargin}

%

Then $\Delta$ is called a \textit{(disk) van Kampen diagram} (or simply a \textit{disk diagram}) over the presentation $\gen{\pazocal{A}\mid\pazocal{R}}$. The cells satisfying condition (1) above are called $\pazocal{R}$-cells, while the others are called 0-cells.

It is easy to see that the contour, $\partial\Delta$, of a disk diagram $\Delta$ has label equal to the identity in $G$. Conversely, van Kampen's Lemma (Lemma 11.1 of [17]) states that a word $W$ over $\pazocal{A}$ represents the identity of $G$ if and only if there exists a disk diagram $\Delta$ over the presentation $\gen{\pazocal{A}\mid\pazocal{R}}$ with $\lab(\partial\Delta)\equiv W$.

The \textit{area}, $\text{Area}(\Delta)$, of a disk diagram $\Delta$ is the number of $\pazocal{R}$-cells it contains, while the area of a word $W$ satisfying $W=1$ in $G$ is the minimal area of a diagram $\Delta$ satisfying $\lab(\partial\Delta)\equiv W$.

A \textit{0-refinement} of a disk diagram $\Delta$ is a disk diagram $\Delta'$ obtained from $\Delta$ by the insertion of 0-edges or 0-cells. Note that a 0-refinement has the same area as the diagram from which it arises.

Let $\Delta$ be a disk diagram and $\Pi_1$ and $\Pi_2$ be two $\pazocal{R}$-cells in $\Delta$. Suppose $O_1,O_2$ are vertices of $\Pi_1,\Pi_2$, respectively, there exists a simple path $t$ from $O_1$ to $O_2$ in $\Delta$ such that $\lab(t)=1$ in $F(\pazocal{A})$ (that is, the free group with basis $\pazocal{A}$), and $\lab(\partial\Pi_1)$ read starting at $O_1$ is mutually inverse to $\lab(\partial\Pi_2)$ read starting at $O_2$ 
Then $\Pi_1$ and $\Pi_2$ are called \textit{cancellable} in $\Delta$.


This term is justified by the ability to `remove' the cells $\Pi_1$ and $\Pi_2$ from $\Delta$ without affecting the label of $\partial\Delta$, yielding a disk diagram $\Delta'$ satisfying $\lab(\partial\Delta')\equiv\lab(\partial\Delta)$ with $\text{Area}(\Delta')<\text{Area}(\Delta)$.

Naturally, a disk diagram is called \textit{reduced} if it has no pair of cancellable cells. By simply removing pairs of cancellable cells, any disk diagram over a presentation can be made reduced. This immediately leads to a strengtheened version of van Kampen's lemma: A word $W$ over $\pazocal{A}$ represents the identity in $G$ if and only if there exists a reduced disk diagram $\Delta$ over the presentation with $\lab(\partial\Delta)\equiv W$.

An annular (Schupp) diagram over the presentation $\gen{\pazocal{A}\mid\pazocal{R}}$ is defined similarly. It is then an immediate consequence of van Kampen's lemma that two words $W$ and $V$ are conjugate in $G$ if and only if there exists a reduced annular diagram $\Delta$ with contour components $p$ and $q$ satsifying $\lab(p)\equiv W$ and $\lab(q)\equiv V^{-1}$. 


\smallskip


\subsection{Graded maps on a disk or annulus} \

The definitions and lemmas presented over the next several subsections are those introduced by Olshanskii in his solution of the Burnside problem. These can be found in [17], with those relevant to the proof of Lemma \ref{a-cells are quadratic B(m,n)} restated here for easy reference.

A \textit{map} $\Delta$ is a finite oriented planar graph on a disk which subdivides the surface into polygonal cells. In particular, by `forgetting' the labelling, one can interpret a van Kampen diagram as a map.

A map $\Delta$ is called \textit{graded} if each cell $\Pi$ in $\Delta$ is assigned a nonnegative integer $r(\Pi)$ called its \textit{rank}. The map $\Delta$ is called a \textit{map of rank at most k} if all its cells have rank $\leq k$. The minimal $k$ for which $\Delta$ is a map of rank at most $k$ is called the \textit{rank} of $\Delta$ and denoted $r(\Delta)$.

For $r(\Delta)=k$, the \textit{type} of $\Delta$, $\tau(\Delta)$, is the $(k+2)$-vector $(r(\Delta),\tau_0,\dots,\tau_k)$, where $\tau_i$ is the number of cells of rank $k-i$ in $\Delta$. The types of maps are ordered lexicographically, i.e for two maps $\Delta$ and $\Gamma$ with $\tau(\Delta)=(r(\Delta),\tau_0,\dots,\tau_k)$ and $\tau(\Gamma)=(r(\Gamma),\sigma_0,\dots,\sigma_\ell)$, $\tau(\Delta)\leq\tau(\Gamma)$ if the following three conditions hold:

\begin{addmargin}[1em]{0em}

$\bullet$ $r(\Delta)\leq r(\Gamma)$;

$\bullet$ if $r(\Delta)=r(\Gamma)$, then $\tau_0\leq\sigma_0$;

$\bullet$ for $1\leq i\leq r(\Delta)$, if $r(\Delta)=r(\Gamma)$ and $\tau_j=\sigma_j$ for all $j<i$, then $\tau_i\leq\sigma_i$.

\end{addmargin}

For simplicity, the cells of rank 0 in a graded map are called \textit{0-cells}. All other cells are called \textit{$\pazocal{R}$-cells} (even though an alphabet $\pazocal{R}$ is not specified). 

The edges of the graph are divided into two disjoint sets, called the \textit{0-edges} and the \textit{$\pazocal{A}$-edges}. The \textit{length} of a path $p$ in a graded map $\Delta$, denoted $|p|$, is the number of $\pazocal{A}$-edges that comprise it. In particular, for $\partial\Pi$ the contour of a cell, $|\partial\Pi|$ is called the \textit{perimeter} of $\Pi$.

Three facts are assumed about graded maps, which are motivated by the definition of van Kampen diagrams:

\begin{addmargin}[1em]{0em}

(1) the inverse edge of a 0-edge is also a 0-edge

(2) the contour of a 0-cell either consists entirely of 0-edges or of exactly two $\pazocal{A}$-edges in addition to a number of 0-edges

(3) if $\Pi$ is an $\pazocal{R}$-cell, then $|\partial\Pi|>0$

\end{addmargin}

If $\Delta$ is a graded map and $\Gamma$ is a subspace homeomorphic to a disk bounded by some edgepath of $\Delta$, then $\Gamma$ is called a \textit{submap} of $\Delta$.

It is further assumed that the contour of a graded map has a fixed decomposition. In particular, if $\Delta$ is a graded map, then $\partial\Delta$ is factorized as $p_1\dots p_k$ with each $p_i$ called a \textit{section} of the contour.

\smallskip


\subsection{0-Bonds and 0-contiguity submaps} \

Let $\Delta$ be a graded map and $\Pi$ be a 0-cell whose contour contains exactly two $\pazocal{A}$-edges, $e_1$ and $e_2$. Then the pair of edges $e_1,e_2^{-1}$ are called \textit{immediately adjacent} (as is the pair $e_1^{-1},e_2$). Two edges $e$ and $f$ of $\Delta$ are then said to be \textit{adjacent} if there exists a sequence of edges $e=e_1,e_2,\dots,e_{k+1}=f$ such that $e_i$ and $e_{i+1}$ are immediately adjacent for $i=1,\dots,k$.

Let $\Delta$ be a graded map with adjacent edges $e$ and $f$. Suppose $e$ belongs to the contour of the $\pazocal{R}$-cell $\Pi_1$ and $f^{-1}$ to the contour of some $\pazocal{R}$-cell $\Pi_2$. Per the definition, set $e=e_1,\dots,e_{k+1}=f$ with 0-cells $\pi_1,\dots,\pi_k$ such that the only two $\pazocal{A}$-edges of $\partial\pi_i$ are $e_i^{-1}$ and $e_{i+1}$.

We can then write $\partial\pi_i=e_i^{-1}p_ie_{i+1}s_i$ for $i=1,\dots,k$ such that $|p_i|=|s_i|=0$. With the aid of 0-refinement, we can assume that $p=p_1\dots p_k$ and $s=s_k\dots s_1$ are simple paths such that each intersects $\Pi_1,\Pi_2$ only on its endpoints.

Then, the submap $\Gamma$ with contour $p^{-1}es^{-1}f^{-1}$ consisting of the cells $\pi_1,\dots,\pi_k$ is called a \textit{0-bond} between $\Pi_1$ and $\Pi_2$. Further, $e$ and $f^{-1}$ are called the \textit{contiguity arcs} of the 0-bond $\Gamma$ and $p$ and $s$ the \textit{side arcs}.

Similarly, if $e$ and $f$ are adjacent edges with $e$ belonging to the contour of some $\pazocal{R}$-cell $\Pi$ and $f^{-1}$ belonging to some section $q$ of the contour, then a 0-bond between $\Pi$ and $q$ is defined. What's more, a 0-bond between two sections of the contour can be defined.

Now suppose $e_1,f_1$ and $e_2,f_2$ are two pairs of adjacent edges such that $e_1$ and $e_2$ belong to the contour of some $\pazocal{R}$-cell $\Pi_1$ and $f_1^{-1},f_2^{-1}$ to some $\pazocal{R}$-cell $\Pi_2$. Then, construct two 0-bonds, $\Gamma_1$ and $\Gamma_2$, between the two pairs, with $\partial\Gamma_i=z_ie_iw_if_i^{-1}$. If $\Gamma_1=\Gamma_2$, set $\Gamma=\Gamma_1$. Otherwise, there exist subpaths $y_1$ and $y_2$ of $\partial\Pi_1$ and $\partial\Pi_2$, respectively, such that $y_1=e_1pe_2$ and $y_2=f_2^{-1}uf_1^{-1}$ (or $y_1=e_2pe_1$ and $y_2=f_1^{-1}uf_2^{-1}$). Then let $\Gamma$ be the submap with contour $z_1y_1w_2y_2$ (or $z_2y_1w_1y_2$). If $\Gamma$ does not contain $\Pi_1$ or $\Pi_2$, then $\Gamma$ is called a \textit{0-contiguity submap} of $\Pi_1$ to $\Pi_2$. In this case, $y_1$ and $y_2$ are called the \textit{contiguity arcs} of $\Gamma$; for clarity, these arcs are denoted $y_i=\Gamma_{  ^\wedge}\Pi_i$ (though if $\Pi_1=\Pi_2$, then $\Gamma_{ ^\wedge}\Pi_1$ is two distinct arcs). The paths $z_1$ and $w_2$ (or $z_2$ and $w_1$) are called the \textit{side arcs} of $\Gamma$. Note that both side arcs have length zero. The ratio $|y_1|/|\partial\Pi_1|$ is called the \textit{degree of contiguity} of $\Pi_1$ to $\Pi_2$, and is denoted $(\Pi_1,\Gamma,\Pi_2)$; similarly, $(\Pi_2,\Gamma,\Pi_1)=|y_2|/|\partial\Pi_2|$ is the degree of contiguity of $\Pi_2$ to $\Pi_1$.

As with 0-bonds, 0-contiguity submaps between a cell and a section of $\partial\Delta$ are similarly defined, as are between two sections of $\partial\Delta$. The contiguity arcs, side arcs, and degree of contiguity are defined similarly; for example, if $\Gamma$ is a 0-contiguity submap between an $\pazocal{R}$-cell $\Pi$ and a section $q$ of the contour of $\partial\Delta$, then the degree of contiguity of $q$ to $\Pi$ is $(q,\Gamma,\Pi)=|\Gamma_{ ^\wedge}q|/|q|$.

Note, however, that if $\Gamma$ is a contiguity submap between a cell $\Pi$ and itself, then $(\Pi,\Gamma,\Pi)$ is a pair of numbers.

Two 0-contiguity submaps $\Gamma_1$ and $\Gamma_2$ are \textit{disjoint} if they have no common cells, their contiguity arcs have no common points, and their side arcs have no common points.

\smallskip


\subsection{Bonds and contiguity submaps} \

In this subection, $\eps\in(0,1)$ is taken to be a fixed constant. For the moment, one can think of this number as `sufficiently small', with this interpretations made precise in the next section.

Set $k>0$ and suppose the terms $j$-bond and $j$-contiguity submap have been defined for all $0\leq j<k$. As with 0-contiguity submaps, the definitions of contiguity arcs, side arcs, and degree of contiguity for $j$-contiguity submaps follow. Two submaps $\Gamma_1,\Gamma_2$ such that $\Gamma_i$ is a $j_i$-contiguity submap for $j_i<k$ are called \textit{disjoint} if they have no common cells, their contiguity arcs have no common points, and their side arcs have no common points. Note that this definition agrees with the analogue for 0-contiguity submaps.

Let $\pi$, $\Pi_1$, and $\Pi_2$ be cells of a graded map $\Delta$, perhaps with $\Pi_1=\Pi_2$, satisfying the following:

\begin{addmargin}[1em]{0em}

(1) $r(\pi)=k$, $r(\Pi_i)>k$ for $i=1,2$,

(2) there are disjoint submaps $\Gamma_1,\Gamma_2$ such that $\Gamma_i$ is a $j_i$-contiguity submap of $\pi$ to $\Pi_i$ with $j_i<k$ and such that $\Pi_1$ is not contained in $\Gamma_2$ and $\Pi_2$ is not contained in $\Gamma_1$,

(3) $(\pi,\Gamma_i,\Pi_i)\geq\eps$ for $i=1,2$.

\end{addmargin}

For $i=1,2$, let $\partial\Gamma_i=v_is_i$ for $v_i={\Gamma_i}_{ ^\wedge}\pi$ and $\partial\pi=u_1v_1u_2v_2$. Letting $\Gamma$ be the submap with contour $s_1u_1^{-1}s_2u_2^{-1}$, $\Gamma$ is called the \textit{$k$-bond} between $\Pi_1$ and $\Pi_2$ defined by the contiguity submaps $\Gamma_1$ and $\Gamma_2$ with \textit{principal cell} $\pi$. The \textit{contiguity arc} of $\Gamma$ to $\Pi_i$ is defined as $\Gamma_{ ^\wedge}\Pi_i\defeq{\Gamma_i}_{ ^\wedge}\Pi_i$. The \textit{side arcs} of $\Gamma$ are defined in the obvious way.

A $k$-bond between an $\pazocal{R}$-cell and a section of the contour or two distinct sections of the contour is defined similarly.

Suppose $\Gamma_1$ is a $k$-bond between two cells $\Pi_1$ and $\Pi_2$ and $\Gamma_2$ is a $j$-bond between $\Pi_1$ and $\Pi_2$ for $j\leq k$. If $\Gamma_1=\Gamma_2$, then set $\Gamma=\Gamma_1$. Otherwise, if $\Gamma_1$ and $\Gamma_2$ are disjoint, then set $\partial\Gamma_i=z_iv_iw_is_i$ for $v_i={\Gamma_i}_{ ^\wedge}\Pi_1$ and $s_i={\Gamma_i}_{ ^\wedge}\Pi_2$. Then set $y_1$ as a subpath of $\partial\Pi_1$ of the form $v_1vv_2$ (or $v_2vv_1$) and $y_2$ as a subpath of $\partial\Pi_2$ of the form $s_2ss_1$ (or $s_1ss_2$). Setting $\Gamma$ as the submap with contour $z_1y_1w_2y_2$ (or $z_2y_1w_1y_2$), if $\Gamma$ does not contain $\Pi_1$ or $\Pi_2$, then it is called the \textit{$k$-contiguity submap} of $\Pi_1$ to $\Pi_2$ defined by the bonds $\Gamma_1$ and $\Gamma_2$. As with previous definitions, $y_i=\Gamma_{ ^\wedge}\Pi_i$ is called the \textit{contiguity arc} of $\Gamma$ to $\Pi_i$, $z_1$ and $w_2$ (or $w_1$ and $z_2$) are called the \textit{side arcs} of $\Gamma$, $(\Pi_1,\Gamma,\Pi_2)=|y_1|/|\Pi_1|$ is called the \textit{degree of contiguity} of $\Gamma$ to $\Pi_1$.

A $k$-contiguity submap between an $\pazocal{R}$-cell and a section of the contour is defined similarly, as is a $k$-contiguity submap between two sections of the contour.

The number $k$ is often omitted when referring to $k$-contiguity submaps, so that there will be reference merely to a contiguity submap. Further, if $\Gamma$ is a contiguity submap between $\Pi_1$ and $\Pi_2$ and $\partial\Gamma=p_1q_1p_2q_2$ with $q_i=\Gamma_{ ^\wedge}\Pi_i$, then $\partial(\Pi_1,\Gamma,\Pi_2)$ denotes the \textit{standard decomposition} $p_1q_1p_2q_2$.

\smallskip


\subsection{Auxiliary parameters} \

The arguments presented through the rest of this section rely on the \textit{lowest parameter principle} introduced in [17]. For this, we introduce the relation $>>$ on parameters defined as follows.

If $\a_1,\a_2,\dots,\a_k$ are parameters with $\a_1>>\a_2>>\dots>>\a_k$, then for $2\leq i\leq k$, it is understood that $\a_1,\dots,\a_{i-1}$ are assigned prior to the assignment of $\a_i$ and that the assignment of $\a_i$ is dependent on the assignment of its predecessors. The resulting inequalities are then understood as `$\a_i\leq$(any positive-valued expression involving $\a_1,\dots,\a_{i-1}$)'.

The principle makes the sequence of inequalities used throughout the rest of the section consistent without muddling the matter with the arithmetic of particular infinitesimals.

Specifically, the assignment of parameters used in this section is:
$$\a>>\b>>\gamma>>\delta>>\eps>>\zeta>>\iota$$
Note that $\eps$ is the parameter used to define contiguity submaps in the previous subsection. 

Further, one more restriction is imposed on the assignment of $\iota$, specifically that its inverse $n=1/\iota$ is a very large odd integer. It can be shown that the necessary constraints on the parameters can be relaxed to allow $n$ to be any odd integer greater than $10^{10}$.

\smallskip


\subsection{A-maps and A$^0$-maps} \

In a graded map $\Delta$, let $p=e_1\dots e_k$ be a path. The insertion or deletion in $p$ of a subpath of the form $ee^{-1}$ (or $e^{-1}e$) is called a \textit{combinatorial deformation of type I}. Meanwhile, if there exists a cell $\pi$ in $\Delta$ with $\partial\pi=f_1\dots f_t$, then the insertion or deletion in $p$ of the subpath $(f_1\dots f_t)^{\pm1}$ is called a \textit{combinatorial deformation of type II}. 

Two paths $p$ and $q$ are called \textit{combinatorially homotopic} if one can pass between them via a finite number of combinatorial deformations. A path $p$ in $\Delta$ is called \textit{geodesic} if $|p|\leq|q|$ for any path $q$ combinatorially homotopic to it.

A graded map $\Delta$ is called an A-map if it satisfies:

\begin{addmargin}[1em]{0em}

(A1) the contour of each cell $\Pi$ of rank $j$ is cyclically reduced with $|\partial\Pi|\geq nj$,

(A2) if $\Pi$ is a cell of rank $j$ in $\Delta$, then any subpath of length $\leq\max(j,2)$ in $\partial\Pi$ is geodesic in $\Delta$,

(A3) if $\Pi_1$ and $\Pi_2$ are $\pazocal{R}$-cells in $\Delta$ and $\Gamma$ is a contiguity submap of $\Pi_1$ to $\Pi_2$ with $(\Pi_1,\Gamma,\Pi_2)\geq\eps$, then $|\Gamma_{ ^\wedge}\Pi_2|<(1+\gamma)r(\Pi)$

\end{addmargin}

For some fixed positive integer $\ell$, an $A$-map on a disk whose contour is decomposed into at most $\ell$ sections is called an A$^0$-map. For the purposes of Lemma \ref{a-cells are quadratic B(m,n)}, we fix $\ell=10$ (as opposed to the convention $\ell=4$ taken in [17]).

Similar to the definition of A-maps, a section $q$ of a contour of a graded map $\Delta$ is called a \textit{smooth section of rank $k>0$} if the following two conditions are satisfied:

\begin{addmargin}[1em]{0em}

(1) every subpath of $q$ of length $\leq\max(k,2)$ is geodesic in $\Delta$,

(2) if $\Gamma$ is a contiguity submap of a cell $\pi$ to $q$ satisfying $(\pi,\Gamma,q)\geq\eps$, then $|\Gamma_{ ^\wedge}q|<(1+\gamma)k$.

\end{addmargin}

The following lemma is clear from the definitions above.

\begin{lemma} \label{Lemma 15.1} (1) A submap of an A-map is an A-map.

(2) If a subpath $p$ of a smooth section $q$ of rank $k$ in an A-map $\Delta$ is a subpath of the contour of a submap $\Gamma$, then $p$ can be regarded as a smooth section of rank $k$ in $\partial\Gamma$.

(3) Let $\Delta$ be an A-map containing a cell $\Pi$ of rank $k$ and $q$ be a subpath of $\partial\Pi$. Suppose $q$ is a section of the contour of a submap $\Gamma$ not containing $\Pi$. Then $q$ is a smooth section of rank $k$ in $\partial\Gamma$.

\end{lemma}

The following lemmas are proved in [17] and stated here for reference.

%
%

\begin{lemma} \label{Corollary 16.1}

\textit{(Corollary 16.1 of [17])} Let $\Delta$ be an A-map on a disk of nonzero rank whose contour is decomposed into the subsections $q_1,\dots,q_{10}$. Then, in $\Delta$, there exists an $\pazocal{R}$-cell $\pi$ and disjoint contiguity submaps $\Gamma_1,\dots,\Gamma_{10}$ of $\pi$ to $q_1,\dots,q_{10}$, respectively, (some of which may be absent) such that $$\sum_{i=1}^{10} (\pi,\Gamma_i,q_i)>1-\gamma$$

\end{lemma}

Note that Lemma \ref{Corollary 16.1} is proved in [17] for four distinguished sections $q_1,\dots,q_4$ of the contour. It is an easy exercise, though, to adapt this proof to the statement above.

The cell $\pi$ guaranteed by Lemma \ref{Corollary 16.1} is called a \textit{$\gamma$-cell}.

\begin{lemma} \label{Theorem 17.1}

\textit{(Theorem 17.1 of [17])} Let $\Delta$ be an A-map with contour $qt$. If $q$ is a smooth section, then $(1-\b)|q|\leq|t|$ (with equality if and only if $|q|=|t|=0$).

\end{lemma}

\begin{lemma} \label{Corollary 17.1}

\textit{(Corollary 17.1 of [17])} If $\Delta$ is an A-map, then $|\partial\Delta|>(1-\b)|\partial\Pi|$ for all $\pazocal{R}$-cells $\Pi$ in $\Delta$.

\end{lemma}

\begin{lemma} \label{Lemma 15.3}

\textit{(Lemma 15.3 of [17])} Let $\Delta$ be an A-map and $\Gamma$ be a contiguity submap of a cell $\Pi$ to a section $q$ and of the contour. If $p_1q_1p_2q_2=\partial(\Pi,\Gamma,q)$ and $P=\max\{|p_1|,|p_2|\}$, then $P<\zeta n\cdot r(\Pi)$.

\end{lemma}

\begin{lemma} \label{Lemma 15.4}

\textit{(Lemma 15.4 of [17])} If the degree of $\Gamma$-contiguity of a cell $\Pi$ to a section $q$ of a contour in an A-map $\Delta$ is equal to $\psi$ and $p_1q_1p_2q_2=\partial(\Pi,\Gamma,q)$, then $|q_2|>(\psi-2\b)|\partial\Pi|$.

\end{lemma}

\smallskip


\subsection{Graded Presentations} \

Given an alphabet $\pazocal{A}$, let $\{\pazocal{S}_i\}_{i=1}^\infty$ be a collection of subsets of $F(\pazocal{A})$ such that if $W\in\pazocal{S}_i$ and $V$ is a cyclic permutation of $W$ or $W^{-1}$, then $V\notin\pazocal{S}_j$ for any $j\neq i$. Set $\pazocal{R}_j=\cup_{i=1}^j\pazocal{S}_i$ for $j\geq1$, $\pazocal{R}_0=\emptyset$, and $\pazocal{R}=\cup_{i=1}^\infty\pazocal{S}_i$. Then $\gen{\pazocal{A}\mid\pazocal{R}}$ is called a \textit{graded presentation} for the group $G$. Further, define $G(j)=\gen{\pazocal{A}\mid\pazocal{R}_j}$ for all $j\geq0$. Note that $G(0)\cong F(\pazocal{A})$.

The words in $\pazocal{S}_i$ are called the \textit{relators of rank $i$}. For words $X,Y$ over $\pazocal{A}$, if $X=Y$ in $G(i)$, then $X$ and $Y$ are said to be \textit{equal in rank $i$}, denoted $X\stackrel{i}{=}Y$.

Given a disk diagram $\Delta$ over the presentation $\gen{\pazocal{A}\mid\pazocal{R}}$, let $\Pi$ be an $\pazocal{R}$-cell such that $\lab(\Pi)$ is a cyclic permutation of a relation of rank $i$ (or the inverse of such a relation). Then $\Pi$ is called a \textit{cell of rank $i$}, given by the representative notation $r(\Pi)=i$. Naturally, the 0-cells of $\Delta$ are called cells of rank 0.

Note that after forgetting the labels of the edges of a disk diagram $\Delta$ over $\gen{\pazocal{A}\mid\pazocal{R}}$, $\Delta$ is a graded map (with the ranks of cells assigned in the same way). A diagram satisfying this property is called a \textit{graded disk diagram}. It is then natural to define the \textit{rank} and \textit{type} of a graded disk diagram as the rank and type of the underlying map.

Let $\Delta$ be a graded disk diagram over $\gen{\pazocal{A}\mid\pazocal{R}}$ containing two $\pazocal{R}$-cells $\Pi_1$ and $\Pi_2$ with $r(\Pi_1)=r(\Pi_2)=j$. Suppose there exists a 0-refinement $\Delta'$ of $\Delta$ with copies $\Pi_1',\Pi_2'$ of $\Pi_1,\Pi_2$, respectively, such that there exist vertices $O_1,O_2$ of $\Pi_1',\Pi_2'$, respectively, and a simple path $t$ between them in $\Delta'$ satisfying $\lab(t)\stackrel{j-1}{=}1$. Suppose further that $\lab(\partial\Pi_1')$ read starting at $O_1$ is mutually inverse to $\lab(\partial\Pi_2')$ read starting at $O_2$. Then $\Pi_1$ and $\Pi_2$ are called a \textit{j-pair} in $\Delta$. 

This generalizes the concept of cancellable cells in a disk diagram over a presentation: If $\Delta$ is a graded disk diagram over a graded presentation with a $j$-pair $\Pi_1,\Pi_2$, then one can `remove' $\Pi_1$ and $\Pi_2$ from $\Delta$ at the cost of cells of rank $\leq j-1$, producing a graded disk diagram $\Delta'$ over the same presentation with $\lab(\Delta')\equiv\lab(\Delta)$ and $\tau(\Delta')<\tau(\Delta)$.

A graded disk diagram $\Delta$ over $\gen{\pazocal{A}\mid\pazocal{R}}$ is \textit{reduced} if, for any graded disk diagram $\Gamma$ over $\gen{\pazocal{A}\mid\pazocal{R}}$ satisfying $\lab(\Delta)\equiv\lab(\Gamma)$, the inequality $\tau(\Delta)\leq\tau(\Gamma)$ holds. Similar to reduced disk diagrams over general presentations, one can make any graded disk diagram reduced simply by the removal of $j$-pairs (for varying $j$). As a result, van Kampen's Lemma can again be strengthened: Given a graded presentation $G=\gen{\pazocal{A}\mid\pazocal{R}}$, a word $W$ over $\pazocal{A}$ represents the identity in $G$ if and only if there exists a reduced graded disk diagram $\Delta$ over the presentation with $\lab(\partial\Delta)\equiv W$.

Graded annular diagrams are defined similarly.

\smallskip


\subsection{The group $B(m,n)$} \

Fix $\pazocal{A}=\{a_1,\dots,a_m\}$ for some $m\geq2$. Let $G(0)=\gen{\pazocal{A}\mid\pazocal{R}_0=\emptyset}\cong F(\pazocal{A})$.

A graded presentation $G(\infty)=\gen{\pazocal{A}\mid\pazocal{R}}$ for $\pazocal{R}=\cup_{i=1}^\infty\pazocal{S}_i$, where every element of $\pazocal{R}$ has the form $A^n$ for $A\in F(\pazocal{A})$, is defined as follows. 

If $A^n\in\pazocal{S}_i$, then $A$ is called a \textit{period of rank $i$}. For $j\geq1$, suppose $\pazocal{R}_{j-1}=\cup_{i=1}^{j-1}\pazocal{S}_i$ and $G(j-1)=\gen{\pazocal{A}\mid\pazocal{R}_{j-1}}$ have already been defined.

Let $A$ be a word in $F(\pazocal{A})$ such that $A\neq1$. Then $A$ is said to be \textit{simple in rank $j-1$} if:

\begin{addmargin}[1em]{0em}

(1) there is no word $B\in F(\pazocal{A})$ such that $B$ is a period of rank $k\leq j-1$ and $A$ is conjugate in $G(j-1)$ to $B^m$ for some integer $m$,

(2) there is no word $C\in F(\pazocal{A})$ such that $|C|<|A|$ and $A$ is conjugate in $G(j-1)$ to $C^m$ for some integer $m$.

\end{addmargin}

Then let $\pazocal{X}_j$ be a maximal subset of words simple in rank $j-1$ of length $j$ subject to the constraint that if $A,B\in\pazocal{X}_j$ and $A\neq B$, then $A$ is not conjugate to $B$ or $B^{-1}$ in $G(j-1)$. 

As an example of this construction, $\pazocal{X}_1=\{a_i^{\eps_i}\}_{i=1}^m$ for some chosen $\eps_1,\dots,\eps_m\in\{\pm1\}$.

Finally, set $\pazocal{S}_j=\{A^n\mid A\in\pazocal{X}_j\}$.

Notice that every relation of rank $j$ has length $nj$, so that a reduced graded diagram over the presentation $\gen{\pazocal{A}\mid\pazocal{R}}$ satisfies (A1) for the definition of an A-map. The following lemma confirms that (A2) and (A3) also hold.

\begin{lemma} \label{Burnside A-maps}

(\textit{Lemma 19.4 of [17]}) If $\Delta$ is a reduced graded disk diagram over the presentation $\gen{\pazocal{A},\pazocal{R}}$ of $G(\infty)$, then the underlying map associated to $\Delta$ is an A-map.

\end{lemma}

What's more, the following explains why this construction is of particular importance here.

\begin{lemma} \label{Burnside group} 

\textit{(Theorem 19.1 and 19.7 of [17])} The group $G(\infty)$ is free in the Burnside variety $\pazocal{B}_n$, i.e $G(\infty)\cong B(m,n)$. Moreover, this group is infinite.

\end{lemma}

\smallskip


\subsection{Mass of a diagram} \

We now introduce a weighting on diagrams over the presentation $\gen{\pazocal{A}\mid\pazocal{R}}$ of $B(m,n)$, generalizing the area of such a diagram.

If $\Pi$ is an $\pazocal{R}$-cell in a reduced graded disk diagram $\Delta$ over the presentation $\gen{\pazocal{A}\mid\pazocal{R}}$ of $B(m,n)$, then we define the \textit{mass} of $\Pi$ as $\rho(\Pi)=|\partial\Pi|^2$. We extend this definition to the mass of the entire diagram, taking $\rho(\Delta)$ to be the sum of the massess of its $\pazocal{R}$-cells.

\begin{lemma} \label{a-cells are quadratic B(m,n)}

If $\Delta$ is a reduced graded disk diagram over the presentation $G(\infty)$ of $B(m,n)$, then $\rho(\Delta)\leq|\partial\Delta|^2$.

\end{lemma}

\begin{proof}

The proof inducts on $|\partial\Delta|$, with the base case $|\partial\Delta|\leq(1-\b)n$. In this case, Lemma \ref{Corollary 17.1} implies that $\Delta$ contains no $\pazocal{R}$-cells, so that $\rho(\Delta)=0$. So, in what proceeds, we assume that $|\partial\Delta|>(1-\b)n$ and $\Delta$ is a `minimal counterexample' to the lemma.

We partition $\partial\Delta$ into 10 section, $q=q_1\dots q_{10}$, any two of which differ in length by at most 1. By Lemma \ref{Burnside A-maps}, it then follows that the underlying map associated to $\Delta$ is an A$^0$-map. Applying Lemma \ref{Corollary 16.1}, there exists in $\Delta$ a $\gamma$-cell $\pi$ together with contiguity submaps $\Gamma_1,\dots,\Gamma_{10}$ of $\pi$ to $q_1,\dots,q_{10}$ (some of which may be absent).

As $|\partial\Delta|>(1-\b)n$, we have $(\frac{1}{10}-\frac{2}{n})|\partial\Delta|<|q_i|<(\frac{1}{10}+\frac{2}{n})|\partial\Delta|$ for all $1\leq i\leq 10$ (using $\b<\frac{1}{2}$).

\underline{Case 1.} Suppose one of the contiguity maps, say $\Gamma_1$, is absent.

Then, let $p=q_2\dots q_{10}$. Define $i,j\in\{2,\dots,10\}$ such that $i$ is the minimal index for which $\Gamma_i$ exists and $j$ the maximal such index. Using the defining bonds of $\Gamma_i$ and $\Gamma_j$, one can then form a contiguity submap $\Gamma$ of $\pi$ to $p$ containing all present contiguity submaps $\Gamma_2,\dots,\Gamma_{10}$ and such that $(\pi,\Gamma,p)>1-\gamma$. Letting $s_1t_1s_2t_2=\partial(\pi,\Gamma,p)$, Lemma \ref{Lemma 15.3} and \ref{Lemma 15.4} then imply that $|t_2|>(1-\gamma-2\b)|\partial\pi|$ and $|s_1|+|s_2|<2\zeta|\partial\pi|$.

Let $\bar{t}_1$ be the complement of $t_1$ in the contour of $\pi$ so that $\partial\pi=t_1^{-1}\bar{t}_1$. Further, let $\bar{t}_2$ be the complement of $t_2$ in the contour of $\Delta$, so that $\partial\Delta=t_2\bar{t}_2$. Letting $u=s_2^{-1}\bar{t}_1s_1^{-1}$, then $|\bar{t}_1|<\gamma|\partial\pi|$ and $|u|<2\gamma|\partial\pi|$ (taking $\zeta<2\gamma$). Now cut $\Delta$ along $u$ to get two graded subdiagrams $\Delta_1$ and $\Delta_2$ with contours $u^{-1}t_2$ and $u\bar{t}_2$.

Then, $|u|<2\gamma|\partial\pi|<2\gamma(1-\gamma-2\b)^{-1}|t_2|<3\gamma|t_2|<3\gamma|\partial\Delta|$ and $|t_2|<9(\frac{1}{10}+\frac{2}{n})|\partial\Delta|<\frac{19}{20}|\partial\Delta|$.

Hence, for sufficiently small $\gamma$,
$$|\partial\Delta_1|=|t_2|+|u|<(1+3\gamma)|t_2|<\frac{19}{20}(1+3\gamma)|\partial\Delta|<|\partial\Delta|$$
$$|\partial\Delta_2|=|u|+|\bar{t}_2|=|u|+|\partial\Delta|-|t_2|<|\partial\Delta|-(1-3\gamma)|t_2|<|\partial\Delta|$$

Applying the inductive hypothesis on both diagrams then gives 
$$\rho(\Delta_1)<(1+3\gamma)^2|t_2|^2$$
$$\rho(\Delta_2)<(|\partial\Delta|-(1-3\gamma)|t_2|)^2$$

For sufficiently small $\gamma$, one can take $\frac{40}{19}(1-3\gamma)\geq(1+3\gamma)^2+(1-3\gamma)^2$. So,
$$|t_2|((1+3\gamma)^2+(1-3\gamma)^2)\leq\frac{40}{19}|t_2|(1-3\gamma)\leq2|\partial\Delta|(1-3\gamma)$$

This means $|t_2|^2(1+3\gamma)^2+|t_2|^2(1-3\gamma)^2-2|\partial\Delta||t_2|(1-3\gamma)\leq0$, and so
$$(|\partial\Delta|-(1-3\gamma)|t_2|)^2+(1+3\gamma)^2|t_2|^2\leq|\partial\Delta|^2$$

This final inequality yields $$\rho(\Delta)=\rho(\Delta_1)+\rho(\Delta_2)<|\partial\Delta|^2$$

\underline{Case 2.} Suppose all $\Gamma_i$ are present.

Then set $s_1^it_1^is_2^it_2^i=\partial(\pi,\Gamma_i,q_i)$ for all $i$, $\partial\pi=t_1^{10}v_{10}t_1^9v_9\dots t_1^1v_1$, and $\partial\Delta=t_2^1w_1t_2^2w_2\dots t_2^{10}w_{10}$. Further, for $i=1,\dots,10$, let $\Delta_i$ be the subdiagram with contour $w_i(s_2^{i+1})^{-1}v_{i+1}(s_1^i)^{-1}$ (with indices counted mod 10).

As in the previous case, Lemma \ref{Lemma 15.3} implies that $|s_j^i|<\zeta|\partial\pi|$. Also, Lemma \ref{Corollary 17.1} implies that $|\partial\pi|<2|\partial\Delta|$.

Lemmas \ref{Lemma 15.1}(c) and \ref{Theorem 17.1} imply that $(1-\b)|t_1^i|<|s_1^i|+|s_2^i|+|t_2^i|$ for all $i$. So, since we also have $|t_2^i|\leq|q_i|<(\frac{1}{10}+\frac{2}{n})|\partial\Delta|$, it follows that $|s_1^i|+|s_2^i|+|t_2^i|<(\frac{1}{10}+\frac{2}{n}+4\zeta)|\partial\Delta|$ for all $i$.

This means that for all $i$, $$|\partial\Gamma_i|=|t_1^i|+|s_1^i|+|s_2^i|+|t_2^i|<\left(1+\frac{1}{1-\b}\right)\left(\frac{1}{10}+\frac{2}{n}+4\zeta\right)|\partial\Delta|<\left(\frac{1}{5}+\frac{4}{n}+8\zeta\right)|\partial\Delta|<\frac{2}{9}|\partial\Delta|$$
So, applying the inductive hypothesis, $\rho(\Gamma_i)<\frac{4}{81}|\partial\Delta|^2$ for all $i$.

Further, $|w_i|<|q_i|+|q_{i+1}|<(\frac{1}{5}+\frac{4}{n})|\partial\Delta|$ and $|v_i|<\gamma|\partial\pi|<2\gamma|\partial\Delta|$, so that
$$|\partial\Delta_i|<\bigg(\frac{1}{5}+\frac{4}{n}+4\zeta+2\gamma\bigg)|\partial\Delta|<\frac{2}{9}|\partial\Delta|$$ for all $i$.

So, applying the inductive hypothesis to all the $\Gamma_i$ and $\Delta_i$ yields $\rho(\Delta_i)<\frac{4}{81}|\partial\Delta|^2$.

Finally, note that since $\pi$ is an $\pazocal{R}$-cell, $\rho(\pi)=(\iota|\partial\pi|)^2<\frac{4}{n^2}|\partial\Delta|^2<\frac{1}{81}|\partial\Delta|^2$. Thus,
$$\rho(\Delta)=\sum_{i=1}^{10}\rho(\Gamma_i)+\sum_{i=1}^{10}\rho(\Delta_i)+\rho(\pi)\leq|\partial\Delta|^2$$

\end{proof} 

\medskip


\section{$S$-Machines}

\subsection{Definition of $S$-machine as a Rewriting System} \

There are many interpretations of $S$-machines. Following the conventions of [19], [24], [26], [27], and others, we approach them here as a rewriting system for words over group alphabets. As such, the definitions in this section are identical to those found in these references.

An $S$-machine $\textbf{S}$ is a rewriting system with hardware $(Y,Q)$, where $Q=\sqcup_{i=0}^N Q_i$ and $Y=\sqcup_{i=1}^N Y_i$ for some positive integer $N$. For convenience of notation, set $Y_0=Y_{N+1}=\emptyset$ in this setting.

The elements of $Q\cup Q^{-1}$ are known as \textit{state letters} or \textit{$q$-letters}, while those of $Y\cup Y^{-1}$ are \textit{tape letters} or \textit{$a$-letters}. The sets $Q_i$ and $Y_i$ are called the \textit{parts} of $Q$ and $Y$, respectively. Note that the parts of the state letters are typically represented by capital letters, while their elements are represented by lowercase.

The \textit{language of admissible words} is the collection of reduced words of the form $q_0^{\eps_0}u_1q_1^{\eps_1}\dots u_kq_k^{\eps_k}$ where $\eps_i\in\{\pm1\}$ and each subword $q_{i-1}^{\eps_{i-1}}u_iq_i^{\eps_i}$ either:

\begin{addmargin}[1em]{0em}

(1) belongs to $(Q_{j-1}F(Y_j)Q_j)^{\pm1}$;

(2) has the form $quq^{-1}$ for $q\in Q_j$ and $u\in F(Y_{j+1})$; or

(3) has the form $q^{-1}uq$ for $q\in Q_j$ and $u\in F(Y_j)$

\end{addmargin}

\medskip

For a reduced word $W\in F(Y\cup Q)$, define its \textit{$a$-length} $|W|_a$ as the number of $a$-letters that comprise it. The $q$-length and $\theta$-length of $W$ are defined similarly and denoted $|W|_q$ and $|W|_\theta$, respectively.

Let $W\equiv q_1u_1q_2u_2q_3\dots q_s$ be an admissible word with $q_i\in Q_{j(i)}^{\eps_i}$ for $\eps_i\in\{\pm1\}$  and $u_i\in F(Y)$. Then the \textit{base} of $W$ is $\text{base}(W)\equiv Q_{j(1)}^{\eps_1}Q_{j(2)}^{\eps_2}\dots Q_{j(s)}^{\eps_s}$, where these letters are merely representatives of their corresponding parts, and $u_i$ is called the $Q_{j(i)}^{\eps_i}Q_{j(i+1)}^{\eps_{i+1}}$-sector of $W$. Note that the base of an admissible word $W$ need not be a reduced word and that $W$ is permitted to have many sectors of the same name (for example, $W$ may contain many $Q_iQ_{i+1}$-sectors). 

The base $Q_0\dots Q_N$ is called the \textit{standard base} of \textbf{S}. An admissible word with the standard base is called a \textit{configuration}. 

Now, set $U_0,\dots,U_m$ and $V_0,\dots,V_m$ as a collection of reduced words over $Y\cup Q$ satisfying:

\begin{addmargin}[1em]{0em}

(1) $U_i$ and $V_i$ have base $Q_{\ell(i)}Q_{\ell(i)+1}\dots Q_{r(i)}$ with $\ell(i)\leq r(i)$ and such that both are subwords of admissible words

(2) $\ell(i+1)=r(i)+1$ for all $i$

(3) $U_0$ and $V_0$ start with letters from $Q_0$, while $U_m$ and $V_m$ end with letters from $Q_N$

\end{addmargin}

Define $Q(\theta)$ as the set of state letters appearing in some $U_i$. Note that $Q(\theta)$ contains exactly one state letter from each part.

Also, set $Y(\theta)=\cup Y_j(\theta)$ for $Y_j(\theta)\subseteq Y_j$. Each $Y_j(\theta)$ is called the \textit{domain} of $\theta$ in the corresponding sector of the standard base.

If $W$ is an admissible word with all its state letters contained in $Q(\theta)\cup Q(\theta)^{-1}$ and all its tape letters contained in $Y(\theta)\cup Y(\theta)^{-1}$, then define $W\cdot\theta$ as the result of simultaneously replacing every subword $U_i^{\pm1}\equiv(u_{\ell(i)}q_{\ell(i)}u_{\ell(i)+1}q_{\ell(i)+1}\dots q_{r(i)}u_{r(i)+1})^{\pm1}$ of $W$ by the subword $V_i^{\pm1}\equiv(v_{\ell(i)}q_{\ell(i)}'v_{\ell(i)+1}\dots q_{r(i)}'v_{r(i)+1})^{\pm1}$, followed by the necessary reduction to make the resulting word again admissible.

In this case, $\theta$ is called an \textit{$S$-rule} of $(Y,Q)$ and is denoted $\theta=[U_0\to V_0,\dots,U_m\to V_m]$. This notation fully describes the rule $\theta$ except for the corresponding sets $Y_j(\theta)$. Throughout the rest of this paper, $Y_j(\theta)$ is assumed to be either $Y_j$ or $\emptyset$ unless otherwise stated, with context making it clear which is chosen.

For any $S$-rule $\theta$, if $\theta$ is applicable to an admissible word $W$, then $W$ is called \textit{$\theta$-admissible}.

An important note to stress is that the application of an $S$-rule results in a reduced word, i.e reduction is not a separate step in the application of the $S$-rule.

If the $i$-th part of the $S$-rule $\theta$ is $U_i\to V_i$, $U_i$ and $V_i$ have base $Q_{\ell(i)}\dots Q_{r(i)}$, and $Y_{r(i)+1}(\theta)=\emptyset$, then this part of the rule is denoted $U_i\xrightarrow{\ell}V_i$ and $\theta$ is said to \textit{lock} the $Q_{r(i)}Q_{r(i)+1}$-sector.

Note that every $S$-rule $\theta$ has a natural inverse, namely $\theta^{-1}=[V_0\to U_0,\dots,V_m\to U_m]$ with $Y_j(\theta^{-1})=Y_j(\theta)$ for all $j$. 

The software of $\textbf{S}$ is then defined as a set of $S$-rules $\Theta(\textbf{S})=\Theta$ that is symmetric, i.e $\theta\in\Theta$ if and only if $\theta^{-1}\in\Theta$. It is convenient to partition $\Theta$ into two disjoint sets, $\Theta^+$ and $\Theta^-$, such that $\theta\in\Theta^+$ if and only if $\theta^{-1}\in\Theta^-$. The elements of $\Theta^+$ are called the \textit{positive rules} while those of $\Theta^-$ are the \textit{negative rules}.

For $t\geq0$, suppose $W_0,\dots,W_t$ are admissible words with the same base such that there exist $\theta_1,\dots,\theta_t\in\Theta$ satisfying $W_{i-1}\cdot\theta_i\equiv W_i$ for all $1\leq i\leq t$. Then the sequence of applications of rules $\pazocal{C}:W_0\to\dots\to W_t$ is called a \textit{computation} of \textit{length} or \textit{time} $t\geq0$ of $\textbf{S}$. Moreover, the word $H=\theta_1\dots\theta_t$ is called the \textit{history} of $\pazocal{C}$ and the notation $W_t\equiv W_0\cdot H$ is used to represent the computation.

A computation is called \textit{reduced} if its history is a reduced word in $F(\Theta^+)$. Every computation can be made reduced without changing the initial and final admissible words of the computation by simply removing consecutive mutually inverse rules.

Typically, it is assumed that each part of the state letters contains two (perhaps the same) fixed elements, called the \textit{start} and \textit{end} state letters. A configuration is called a \textit{start} (or \textit{end}) configuration if all of its state letters are start (or end) letters.

A \textit{recognizing} $S$-machine is one with specified sectors called the \textit{input sectors}. If a start configuration has all sectors empty except for the input sectors, then it is called an \textit{input configuration} and its projection onto $Y\cup Y^{-1}$ is called its \textit{input}. The end configuration with every sector empty is called the \textit{accept configuration}.

A configuration $W$ is \textit{accepted} by a recognizing $S$-machine if there is an \textit{accepting computation}, i.e a computation whose initial configuration is $W$ and whose final configuration is the accept configuration. If $W$ is an accepted input configuration with input $u$, then $u$ is also said to be \textit{accepted}.

If the configuration $W$ is accepted by the $S$-machine $\textbf{S}$, then $T(W)$ is the minimal time of its accepting computations. For a recognizing $S$-machine \textbf{S}, its \textit{time function} is $$T_{\textbf{S}}(n)=\max\{T(W): W\text{ is an accepted input configuration of } \textbf{S}, \ |W|_a\leq n\}$$

\medskip



If two recognizing $S$-machines have the same language of accepted words and $\Theta$-equivalent time functions, then they are said to be \textit{equivalent}.

The following simplifies how one approaches the rules of a recognizing $S$-machine.

\medskip

\begin{lemma} \label{simplify rules}

\textit{(Lemma 2.1 of [19])} Every recognizing $S$-machine $\textbf{S}$ is equivalent to a recognizing $S$-machine that satisfies:

(1) Every part of every rule has a 1-letter base (i.e if $U_i\to V_i$ is a part of a rule $\theta$, then $U_i\equiv u_iq_iu_{i+1}$ and $V_i\equiv v_iq_i'v_{i+1}$ for $q_i,q_i'$ state letters in $Q_i$)

(2) In every part $u_iq_iu_{i+1}\to v_iq_i'v_{i+1}$ of every rule, $\|u_i\|+\|v_i\|\leq1$ and $\|u_{i+1}\|+\|v_{i+1}\|\leq1$.

(3) Moreover, with the terminology of (2), $\|u_i\|+\|v_i\|+\|u_{i+1}\|+\|v_{i+1}\|\leq1$.

\end{lemma}

As a result of Lemma \ref{simplify rules}, we can assume that each part of every rule of an $S$-machine is of the form $q_i\to aq_i'b$ with $\|a\|+\|b\|\leq1$. However, it is often convenient to allow $\|a\|=\|b\|=1$ in a defining rule of an $S$-machine, with the implicit understanding that the machine is equivalent to one satisfying property (3).

\smallskip


\subsection{Some elementary properties of $S$-machines} \

The following is an immediate consequence of the definition of admissible words.

\begin{lemma} \label{locked sectors}

If the rule $\theta$ locks the $Q_iQ_{i+1}$-sector, i.e it has a part $q_i\xrightarrow{\ell}aq_i'b$ for some $q_i,q_i'\in Q_i$, then the base of any $\theta$-admissible word has no subword of the form $Q_iQ_i^{-1}$ or $Q_{i+1}^{-1}Q_{i+1}$.

\end{lemma}

Through the rest of this paper, we will often use copies of words over disjoint alphabets. To be precise, let $A$ and $B$ be disjoint alphabets, $W\equiv a_1^{\eps_1}\dots a_k^{\eps_k}$ with $a_i\in A$ and $\eps_i\in\{\pm1\}$, and $\varphi:\{a_1,\dots,a_k\}\to B$ be an injection. Then the \textit{copy} of $W$ over the alphabet $B$ formed by $\varphi$ is the word $W'\equiv\varphi(a_1)^{\eps_1}\dots\varphi(a_k)^{\eps_k}$. Typically, the injection defining the copy will be contextually clear.

Alternatively, a copy of an alphabet $A$ is a disjoint alphabet $A'$ which is in one-to-one correspondence with $A$. For a word over $A$, its copy over $A'$ is defined by the correspondence between the alphabets.

The following four lemmas are properties of some simple computations in $S$-machines that are fundamental to the proofs presented in the next two sections. They are stated here without proof, with a reference provided for their proofs in previous literature.

\begin{lemma} \label{multiply one letter}

\textit{(Lemma 2.7 of [19])} Let $\pazocal{C}:W_0\to\dots\to W_t$ be a reduced computation, where $W_0$ is an admissible word with the two-letter base $Q_iQ_{i+1}$. Denote the tape word of $W_j$ as $u_j$ for each $0\leq j\leq t$. Suppose that each rule of $\pazocal{C}$ multiplies the $Q_iQ_{i+1}$-sector by a letter on the left (respectively right). Suppose further that different rules multiply this sector by different letters. Then:

\begin{addmargin}[1em]{0em}

$(a)$ the history of the computation, $H$, is a copy of the reduced form of the word $u_tu_0^{-1}$ read from right to left (respectively the word $u_0^{-1}u_t$ read left to right). In particular, if $u_0\equiv u_t$, then the computation is empty

$(b)$ $\|H\|\leq\|u_0\|+\|u_t\|$

$(c)$ if $\|u_{j-1}\|<\|u_j\|$ for some $1\leq j\leq t-1$, then $\|u_j\|<\|u_{j+1}\|$

$(d)$ $\|u_j\|\leq\max(\|u_0\|,\|u_t\|)$

\end{addmargin}

\end{lemma}

%
%
%
%
%
%
%

\begin{lemma} \label{multiply two letters}

\textit{(Lemma 2.8 of [19])} Let $W$ be an admissible word with base $Q_iQ_{i+1}$ and $X_\ell,X_r$ be disjoint alphabets in $Y_i$. Let $\pazocal{C}:W\equiv W_0\to\dots\to W_t$ be a reduced computation and denote the tape word of $W_j$ by $u_j$ for each $0\leq j\leq t$. Suppose that each rule of $\pazocal{C}$ multiplies the $Q_iQ_{i+1}$-sector by a letter of $X_\ell$ on the left and a letter of $X_r$ on the right, with different rules multiplying by different letters. Then:

\begin{addmargin}[1em]{0em}

$(a)$ if $\|u_{j-1}\|<\|u_j\|$ for some $0\leq j\leq t-1$, then $\|u_j\|<\|u_{j+1}\|$

$(b)$ $\|u_j\|\leq\max(\|u\|,\|u'\|)$ for each $j$

$(c)$ $t\leq\frac{1}{2}(\|u_0\|+\|u_t\|)$.

\end{addmargin}

\end{lemma}

%
%
%
%
%
%

\begin{lemma} \label{unreduced base}

\textit{(Lemma 3.6 of [25])} Suppose the base of an admissible word $W$ of an $S$-machine is $Q_iQ_i^{-1}$ (respectively $Q_i^{-1}Q_i$). Let $\pazocal{C}$ be a reduced computation starting with an admissible word $W$ with tape word $u$ and ending with an admissible word $W'$ with tape word $u'$. Suppose each rule of $\pazocal{C}$ multiplies the $Q_iQ_{i+1}$-sector (respectively the $Q_{i-1}Q_i$-sector) by a letter from the left (respectively from the right), with different rules corresponding to different letters. Then the history of the computation has the form $H_1H_2^kH_3$, where $k\geq0$, $\|H_2\|\leq\min(\|u\|,\|u'\|)$, $\|H_1\|\leq\|u\|/2$, and $\|H_3\|\leq\|u'\|/2$.

\end{lemma}

\begin{lemma} \label{three part history}

\textit{(Lemma 2.8 of [26])} Suppose that a reduced computation $\pazocal{C}:W_0\to\dots\to W_t$ of an $S$-machine has a two-letter base and history of the form $H\equiv H_1H_2^kH_3$, $k\geq0$. Then $$|W_i|_a\leq|W_0|_a+|W_t|_a+2\|H_1\|+3\|H_2\|+2\|H_3\|$$

\end{lemma}

%
%
%
%
%
%

\smallskip


\subsection{Parameters} \

The arguments spanning the rest of this paper are reliant on the \textit{highest parameter principle}, the obvious dual to the lowest parameter principle described in Section 2.5.

In particular, we introduce the relation $<<$ on parameters defined as follows: 

If $\a_1,\a_2,\dots,\a_n$ are parameters with $\a_1<<\a_2<<\dots<<\a_n$, then for all $2\leq i\leq n$, it is understood that $\a_1,\dots,\a_{i-1}$ are assigned prior to the assignment of $\a_i$ and that the assignment of $\a_i$ is dependent on the assignment of its predecessors. The resulting inequalities are then understood as `$\a_i\geq$(any expression involving $\a_1,\dots,\a_{i-1}$)'

This principle makes the series of inequalities used throughout the rest of this proof consistent.

Specifically, the assignment of parameters we use here is:
\begin{align*}
\lambda^{-1}&<<k<<N<<c_0<<c_1<<c_2<<c_3<<c_4<<c_5<<c_6 \\
&<<L_0<<L<<K<<J<<\delta^{-1}<<c_7<<c_8<<N_1<<N_2<<N_3
\end{align*}

\bigskip


\section{Auxiliary Machines}

\subsection{The machine $\textbf{M}_1$} \

Let $n>10^{10}$ be a very large positive odd integer and set $\pazocal{A}=\{a_1,a_2\}$. Define the \textit{language of defining relations of $B(2,n)$} as the set $\pazocal{L}=\{u^n:u\in F(\pazocal{A})\}$. Note that $\pazocal{L}$ is a recursively enumerable set and $B(2,n)=\gen{\pazocal{A}\mid w=1, \ w\in\pazocal{L}}$.

The first recognizing $S$-machine in the construction used in this paper, $\textbf{M}_1$, has hardware $(Y,Q)$, where $Q=\sqcup_{i=0}^sQ_i$ and $Y=\sqcup_{i=1}^sY_i$, and software the set of rules $\Phi$. The machine has one input sector, assumed to be the $Q_0Q_1$-sector, whose tape alphabet $Y_1$ is a copy of $\pazocal{A}$; for simplicity, we identify $Y_1$ with $\pazocal{A}$. The main property of $\textbf{M}_1$ is the following:

\begin{lemma} \label{M_1 language}

There exists a recognizing $S$-machine as described above with language of accepted words $\pazocal{L}$. Moreover, for all $u^n\in\pazocal{L}$, there exists an accepting computation of length at most $c_0\|u\|$.

\end{lemma}

The existence of a such a machine with language of accepted words $\pazocal{L}$ is proved in [24] and [27] (relying only on the fact that $\pazocal{L}$ is recursively enumerable), while the main machine of [23] can be shown to satisfy the entire statement of Lemma \ref{M_1 language}.

\smallskip


\subsection{Historical sectors} \

Through the rest of this section, we alter the machine $\textbf{M}_1$ in specific ways to create the machines $\textbf{M}_2-\textbf{M}_5$ which satisfy properties that will be desirable for the construction of the main machine $\textbf{M}$. Many of the techniques used for these alterations are similar or identical to those used in [19], [22], and [26]. So, the proofs of many of the pertinent lemmas will be omitted, replaced instead with references to the proofs presented in those sources.

First, we introduce new sectors to the machine $\textbf{M}_1$, called \textit{historical sectors}, to obtain the recognizing $S$-machine $\textbf{M}_2$. The role of these sectors is the following: Given a computation of $\textbf{M}_1$ in the standard base, the same computation will execute identically in the \textit{working sectors}, i.e the analogues of the original sectors, while writing copies of its history in the newly added historical sectors.

The precise construction of $\textbf{M}_2$ is as follows.

Writing the hardware of $\textbf{M}_1$ as $(Y,Q)$ with $Q=\sqcup_{i=0}^sQ_i$ and $Y=\sqcup_{i=1}^sY_i$, the hardware of $\textbf{M}_2$ is then $(Y_h,Q_h)$ where $$Q_h=Q_{0,r}\sqcup Q_{1,\ell}\sqcup Q_{1,r}\sqcup Q_{2,r}\sqcup\dots\sqcup Q_{s,\ell}\sqcup Q_{s,r}, \ \ \ Y_h=Y_1\sqcup X_1\sqcup Y_2\sqcup\dots\sqcup X_{s-1}\sqcup Y_s\sqcup X_s$$

For simplicity of notation, set $Q_{0,\ell}=\emptyset$.

Here, for each $i$, $Q_{i,\ell}$ and $Q_{i,r}$ are copies of $Q_i$ (where $q_{i,\ell}\in Q_{i,\ell}$ and $q_{i,r}\in Q_{i,r}$ correspond to each $q_i\in Q_i$) and $X_i=X_{i,\ell}\sqcup X_{i,r}$ consists of two disjoint copies of $\Phi^+$, called the \textit{left} and \textit{right} history alphabets.

The positive rules of $\textbf{M}_2$ are in one-to-one correspondence with the positive rules of $\textbf{M}_1$: If $\theta=[q_0\to a_0q_0'b_1, \dots ,q_s\to a_sq_s'b_{s+1}]$ is a positive rule of $\textbf{M}_1$ (with $a_0$ and $b_{s+1}$ necessarily empty), then the corresponding positive rule $\theta_h$ of $\textbf{M}_2$ is made up of the pairs of parts $$q_{i,\ell}\to a_iq_{i,\ell}'(\theta_i^{(\ell)})^{-1}, \ \ \ q_{i,r}\to \theta_i^{(r)}q_{i,r}'b_{i+1}$$ where $\theta_i^{(\ell)}$ and $\theta_i^{(r)}$ are the copies of $\theta$ in the alphabet $X_{i,\ell}$ and $X_{i,r}$, respectively. For each $i$, the domain of $\theta_h$ in the $Q_{i,r}Q_{i+1,\ell}$-sector is $Y_i(\theta)$, while its domain in the $Q_{i,\ell}Q_{i,r}$-sector is $X_i$. In particular, no $Q_{i,\ell}Q_{i,r}$-sector is locked by a rule of $\textbf{M}_2$.

Note that as desired, for each $\theta\in\Phi$, the rule $\theta_h$ operates in the $Q_{i,r}Q_{i+1,\ell}$-sector in the same way that $\theta$ operates in the $Q_iQ_{i+1}$-sector in a computation of $\textbf{M}_1$.

In the standard base, the $Q_{i,r}Q_{i+1,\ell}$-sectors are called the \textit{working sectors} while the $Q_{i,\ell}Q_{i,r}$-sectors are called the \textit{historical sectors}. In a nonstandard base, sectors of the form $Q_{i,\ell}Q_{i,\ell}^{-1}$, $Q_{i,r}^{-1}Q_{i,r}$, and $Q_{i,r}^{-1}Q_{i,\ell}^{-1}$ are also called historical.

The start (or end) state letter for each part of the state letters of $\textbf{M}_2$ is the copy of the corresponding start (or end) letter in $\textbf{M}_1$, while the working sector corresponding to the input sector of $\textbf{M}_1$ (i.e the $Q_{0,r}Q_{1,\ell}$-sector) is the input sector of $\textbf{M}_2$.

For $w\in F(\pazocal{A})$ and $H_1\in F(\Phi^+)$, define $I_2(w,H_1)$ as the start configuration with $w$ written in the input sector, the copy of $H_1$ in the left alphabet $X_{i,\ell}$ written in the $Q_{i,\ell}Q_{i,r}$-sector for each $i$, and all other sectors empty. Conversely, define $A_2(H_1)$ as the end configuration with the copy of $H_1$ in the right alphabet $X_{i,r}$ written in the $Q_{i,\ell}Q_{i,r}$-sector for each $i$ and all other sectors empty.

\begin{lemma} \label{M_2 language}

(1) For any $u^n\in\pazocal{L}$, there exists $H_1\in F(\Phi^+)$ and a reduced computation $I_2(u^n,H_1)\equiv W_0\to\dots\to W_t\equiv A_2(H_1)$ of the machine $\textbf{M}_2$ satisfying $t=\|H_1\|\leq c_0\|u\|$.

(2) If there exists a computation $I_2(w,H_1)\to\dots\to A_2(H_1')$ of $\textbf{M}_2$ for $w\in F(\pazocal{A})$, then $w\in\pazocal{L}$ and $H_1'\equiv H_1$.

\end{lemma}

\begin{proof} 

(1) By Lemma \ref{M_1 language}, there exists a computation of $\textbf{M}_1$ with history $H_1\in F(\Phi^+)$ accepting $u^n$ and such that $\|H_1\|\leq c_0\|u\|$. Identifying each rule of $\textbf{M}_2$ with the corresponding rule of $\textbf{M}_1$, the computation of $\textbf{M}_2$ with history $H_1$ and initial configuration $I_2(u^n,H_1)$ satisfies the statement.

(2) Let $H$ be the history of the computation in question. Since $\textbf{M}_2$ operates as $\textbf{M}_1$ in the working sectors, it follows that $w$ is accepted by $\textbf{M}_1$, so that $w\in\pazocal{L}$ by Lemma \ref{M_1 language}. Note that in each historical sector, the computation multiplies on the left by the copy of $H^{-1}$ in the left alphabet and on the right by the copy of $H$ in the right alphabet. Since the historical sectors of $I_2(w,H_1)$ (of $A_2(H_1')$) have no letters from the right alphabets (from the left alphabets), we then see that $H_1$ and $H_1'$ must be the copy of $H$ in $F(\Phi^+)$.

\end{proof}

%
%
%
%
%
%

A benefit of adding historical sectors is in providing a linear estimate for the lengths of computations $W_0\to\dots\to W_t$ in terms of $\|W_0\|$ and $\|W_t\|$, as evidenced from the following statements from previous literature.

\begin{lemma} \label{one alphabet historical words}

\textit{(Lemma 3.9 of [19])} Let $W_0\to\dots\to W_t$ be a reduced computation of $\textbf{M}_2$ with base $Q_{i,\ell}Q_{i,r}$ and history $H$. Assume that all tape letters of $W_0$ belong to either $X_{i,\ell}$ or $X_{i,r}$. Then $\|H\|\leq|W_t|_a$ and $|W_0|_a\leq |W_t|_a$.

\end{lemma}

%
%
%

\begin{lemma} \label{M_2 restriction}

\textit{(Lemma 3.13 of [22])} Suppose a reduced computation $W_0\to\dots\to W_t$ of $\textbf{M}_2$ has base of length at most some constant $N_0$ and containing a subword $Q_{i,\ell}Q_{i,r}$. If $W_0$ has no letters from $X_{i,\ell}$ or no letters from $X_{i,r}$, then there is a constant $c=c(N_0)$ such that $|W_0|_a\leq c|W_t|_a$.

\end{lemma}

%
%
%

\begin{lemma} \label{M_2 bound}

\textit{(Lemma 3.12 of [19])} For any reduced computation $W_0\to\dots\to W_t$ of $\textbf{M}_2$ with base of length at least three, $|W_i|_a\leq9(|W_0|_a+|W_t|_a)$ for all $0\leq i\leq t$.

\end{lemma}

\smallskip


\subsection{Primitive Machines} \

As in the constructions in [19] and [26], we next introduce two machines, denoted $\textbf{LR}(Y)$ and $\textbf{RL}(Y)$ for an alphabet $Y$ and called \textit{primitive machines}, that will be used to alter $\textbf{M}_2$.

The standard base of $\textbf{LR}(Y)$ is $Q^{(1)}PQ^{(2)}$ with $Q^{(1)}=\{q^{(1)}\}$, $P=\{p^{(1)},p^{(2)}\}$, and $Q^{(2)}=\{q^{(2)}\}$. The letter $p^{(1)}$ is the start letter of $P$, while $p^{(2)}$ is the end letter. 

The tape alphabets are two disjoint copies of $Y$, denoted $Y^{(1)}$ and $Y^{(2)}$ and assigned in the natural way.

The positive rules of $\textbf{LR}(Y)$ come in the following three forms:

\begin{addmargin}[1em]{0em}

$\bullet$ $\zeta^{(1)}(a)=[q^{(1)}\to q^{(1)}, \ p^{(1)}\to a_1^{-1}p^{(1)}a_2, \ q^{(2)}\to q^{(2)}]$ for all $a\in Y$, where $a_i$ is its copy in $Y^{(i)}$.

\textit{Comment.} In practice, this moves $p^{(1)}$ left, replacing letters from the $Q^{(1)}P$-sector with their copies in the $PQ^{(2)}$-sector.

\medskip

$\bullet$ $\zeta^{(12)}=[q^{(1)}\xrightarrow{\ell} q^{(1)}, \ p^{(1)}\to p^{(2)}, \ q^{(2)}\to q^{(2)}]$

\textit{Comment.} When $p^{(1)}$ meets $q^{(1)}$, this rule switches it to $p^{(2)}$. This is called the \textit{connecting rule}.

\medskip

$\bullet$ $\zeta^{(2)}(a)=[q^{(1)}\to q^{(1)}, \ p^{(2)}\to a_1p^{(2)}a_2^{-1}, \ q^{(2)}\to q^{(2)}]$ for all $a\in Y$, where $a_i$ is its copy in $Y^{(i)}$.

\textit{Comment.} When $p^{(2)}$ is present, this rule moves it to the right towards $q^{(2)}$ and replaces letters in the $PQ^{(2)}$-sector with their copies in the $Q^{(1)}P$-sector.

\end{addmargin}

The state letters of $P$ are called \textit{running state letters}. In practice, they ``run" left to the adjacent state letter and then right to the other, as is indicated by the name of the machine.

It is useful to note the following two points:

\begin{addmargin}[1em]{0em}

$\bullet$ In a computation of the standard base, each of the rules $(\zeta^{(j)}(a))^{\pm1}$ changes the $a$-length of a configuration by at most two. In particular, it changes the length by two or leaves it the same.

$\bullet$ Consider the projection of configurations onto $F(Y)$ given by sending state letters to the identity and letters from $Y^{(j)}$ to their copies in $Y$. No rule of $\textbf{LR}(Y)$ changes the value of the configuration under this projection. An application of this useful fact is referred to as a \textit{projection argument}.

\end{addmargin}

\begin{lemma} \label{primitive computations}

\textit{(Lemma 3.1 of [19])} Let $\pazocal{C}:W_0\to\cdots\to W_t$ be a reduced computation of $\textbf{LR}(Y)$ in the standard base. Then:

\begin{addmargin}[1em]{0em}

(1) if $|W_{i-1}|_a<|W_i|_a$ for some $1\leq i\leq t-1$, then $|W_i|_a<|W_{i+1}|_a$

(2) $|W_i|_a\leq\max(|W_0|_a,|W_t|_a)$ for each $i$

(3) if $W_0\equiv q^{(1)}up^{(1)}q^{(2)}$ and $W_t\equiv q^{(1)}vp^{(2)}q^{(2)}$ for some $u,v\in F(Y^{(1)})$, then $u\equiv v$, $|W_i|_a=|W_0|_a\defeq l$ for each $i$, $t=2l+1$, and the $Q^{(1)}P$-sector is locked in the rule $W_l\to W_{l+1}$. Moreover, letting $\bar{u}$ be the mirror image of $u$ and $\bar{u}_2$ its canonical copy over $Y^{(2)}$, the history $H$ is a copy of $\bar{u}\zeta^{(12)}\bar{u}_2^{-1}$

(4) if $W_0\equiv q^{(1)}up^{(j)}q^{(2)}$ and $W_t\equiv q^{(1)}vp^{(j)}q^{(2)}$ for some $u,v$ and $j\in\{1,2\}$, then $u\equiv v$ and the computation is empty (i.e $t=0$)

(5) if $W_0$ is of the form $q^{(1)}up^{(1)}q^{(2)}$, $q^{(1)}p^{(1)}uq^{(2)}$, $q^{(1)}up^{(2)}q^{(2)}$, or $q^{(1)}p^{(2)}uq^{(2)}$ for some word $u$, then $|W_i|_a\geq|W_0|_a$ for every $i$.

\end{addmargin}

\end{lemma}

%
%
%
%
%

The machine $\textbf{RL}(Y)$ is the right analogue of $\textbf{LR}(Y)$. To be precise, the standard base of $\textbf{RL}(Y)$ is $Q^{(1)}RQ^{(2)}$ with $R=\{r^{(1)},r^{(2)}\}$, the tape alphabets are again two copies of $Y$ denoted $Y^{(1)}$ and $Y^{(2)}$, and the positive rules are:

\begin{addmargin}[1em]{0em}

$\bullet$ $\xi^{(1)}(a)=[q^{(1)}\to q^{(1)}, \ r^{(1)}\to a_1r^{(1)}a_2^{-1}, \ q^{(2)}\to q^{(2)}]$ for all $a\in Y$, where $a_i$ is its copy in $Y^{(i)}$

$\bullet$ $\xi^{(12)}=[q^{(1)}\to q^{(1)}, \ r^{(1)}\xrightarrow{\ell}r^{(2)}, \ q^{(2)}\to q^{(2)}]$

$\bullet$ $\xi^{(2)}(a)=[q^{(1)}\to q^{(1)}, \ r^{(2)}\to a_1^{-1}r^{(2)}a_2, \ q^{(2)}\to q^{(2)}]$ for all $a\in Y$, where $a_i$ is its copy in $Y^{(i)}$.

\end{addmargin}

There is an obvious analogue of Lemma \ref{primitive computations} for $\textbf{RL}(Y)$, which can be verified in exactly the same way.

Next, we generalize these machines as in [26] to force the running state letters to run back and forth a number of times rather than just once. Specifically, for the parameter $k$ (see Section 3.3), define the machine $\textbf{LR}_k(Y)$ (or $\textbf{RL}_k(Y)$) as the composition of $\textbf{LR}(Y)$ (respectively $\textbf{RL}(Y)$) with itself $k$ times.

To be precise, the standard base of $\textbf{LR}_k(Y)$ is the same as that of $\textbf{LR}(Y)$, but with $P=\{p^{(i)}\}_{i=1}^{2k}$ consisting of $2k$ letters rather than 2. The end letter of $P$ is now $p^{(2k)}$, while the start letter is again $p^{(1)}$.

The tape alphabets of $\textbf{LR}_k(Y)$ are the same two copies of $Y$ that serve as the tape alphabets of $\textbf{LR}(Y)$.

The positive rules are:

\begin{addmargin}[1em]{0em}

$\bullet$ $\zeta^{(2j-1)}(a)=[q^{(1)}\to q^{(1)}, \ p^{(2j-1)}\to a_1^{-1}p^{(2j-1)}a_2, \ q^{(2)}\to q^{(2)}]$ for $1\leq j\leq k$ and for all $a\in Y$, where $a_i$ is its copy in $Y^{(i)}$.
\medskip

$\bullet$ $\zeta^{(2j-1,2j)}=[q^{(1)}\xrightarrow{\ell} q^{(1)}, \ p^{(2j-1)}\to p^{(2j)}, \ q^{(2)}\to q^{(2)}]$ for all $1\leq j\leq k$

\medskip

$\bullet$ $\zeta^{(2j)}(a)=[q^{(1)}\to q^{(1)}, \ p^{(2j)}\to a_1p^{(2j)}a_2^{-1}, \ q^{(2)}\to q^{(2)}]$ for all $1\leq j\leq k$ and for all $a\in Y$, where $a_i$ is its copy in $Y^{(i)}$.

\medskip

$\bullet$ $\zeta^{(2j,2j+1)}=[q^{(1)}\to q^{(1)}, \ p^{(2j)}\xrightarrow{\ell} p^{(2j+1)}, \ q^{(2)}\to q^{(2)}]$ for all $1\leq j\leq k-1$

\end{addmargin}

Note that for the case $k=1$, one can see that $\textbf{LR}_1(Y)=\textbf{LR}(Y)$.

The definition of $\textbf{RL}_k(Y)$ is formulated similarly.

The analogues of Lemma \ref{primitive computations} hold for the machines $\textbf{LR}_k(Y)$ and $\textbf{RL}_k(Y)$. For example, the following is the analogue of part (3):

\begin{lemma} \label{LR_k analogue}

Let $\pazocal{C}:W_0\to\dots\to W_t$ be a reduced computation of $\textbf{LR}_k(Y)$ in the standard base. If $W_0\equiv q^{(1)}up^{(1)}q^{(2)}$ and $W_t\equiv q^{(1)}vp^{(2k)}q^{(2)}$ for some $u,v\in F(Y^{(1)})$, then $u\equiv v$, $|W_i|_a=|W_0|_a\defeq l$ for each $i$, and $t=2lk+2k-1$.

\end{lemma}

When the alphabet $Y$ is contextually clear, it is convenient to omit it from the names of the machines. So, there will be reference in subsequent constructions to the machines $\textbf{LR}$, $\textbf{RL}$, $\textbf{LR}_k$, and $\textbf{RL}_k$.

\smallskip


\subsection{The machine $\overline{\textbf{M}}_2$} \

The next machine in this construction, $\textbf{M}_3$, is the composition of the machine $\textbf{M}_2$ with the primitive machines $\textbf{LR}$ and $\textbf{RL}$. To aid in this construction, we first introduce (as in [26]) the intermediate recognizing $S$-machine $\overline{\textbf{M}}_2$.

For convenience of notation, rewrite the standard base of $\textbf{M}_2$ as $\sqcup_{i=0}^{s} Q_i$ (note that this changes the value of $s$ from that used in Sections 4.1 and 4.2). Then the standard base of $\overline{\textbf{M}}_2$ is obtained by replacing each $Q_i$ with the three-letter subword $P_iQ_iR_i$, so that the standard base of $\overline{\textbf{M}}_2$ is $$P_0Q_0R_0P_1Q_1R_1\dots P_{s}Q_{s}R_{s}$$ 
The parts $P_i$ contain the running state letters of $\textbf{LR}$ for each $i$, i.e $P_i=\{p_i^{(1)},p_i^{(2)}\}$. Similarly, the parts $R_i$ contain the running state letters of $\textbf{RL}$, so that $R_i=\{r_i^{(1)},r_i^{(2)}\}$.

For $0\leq i\leq s-1$, the tape alphabets of the $Q_iR_i$-, $R_iP_{i+1}$-, and $P_{i+1}Q_{i+1}$-sectors are copies of the tape alphabet of the $Q_iQ_{i+1}$-sector of $\textbf{M}_2$. The $P_0Q_0$- and $Q_sR_s$-sectors have empty tape alphabet and are present only for notational purposes.

For every positive rule $\theta$ of $\textbf{M}_2$, there are two corresponding positive rules of $\overline{\textbf{M}}_2$, denoted $\theta(1)$ and $\theta(2)$. If the $i$-th part of $\theta$ is $q_i\to a_iq_i'b_{i+1}$, then there are three parts of $\theta(j)$ given by
$$p_i^{(j)}\xrightarrow{\ell} a_ip_i^{(j)}, \ q_i\xrightarrow{\ell}q_i', \ r_i^{(j)}\to r_i^{(j)}b_{i+1}$$ 
The rules constructed this way fully comprise the positive rules of $\overline{\textbf{M}}_2$. So, the $P_iQ_i$- and $Q_iR_i$-sectors are always locked and the only work being done is in the $R_iP_{i+1}$-sectors. Naturally, the $R_0P_1$-sector is the input sector of the machine.

If the $Q_iQ_{i+1}$-sector is a working sector of $\textbf{M}_2$, then the $R_iP_{i+1}$-sector is called a working sector of $\overline{\textbf{M}}_2$. On the other hand, if the $Q_iQ_{i+1}$-sector is historical, then the $Q_iR_i$-, $R_iP_{i+1}$-, and $P_{i+1}Q_{i+1}$-sectors are all called historical.

For non-standard bases, appropriate $Q_iQ_i^{-1}$-, $R_iR_i^{-1}$-, etc sectors are also historical sectors of $\overline{\textbf{M}}_2$. Essentially, a sector is historical for $\overline{\textbf{M}}_2$ if its tape alphabet is the disjoint union of left and right parts.

If $B$ is the base of some computation of $\overline{\textbf{M}}_2$ and $UV$ is a two-letter subword of $B$ such that $UV$-sectors are historical (respectively working, input), then $UV$ is a historical (respectively working, input) subword of $B$.

\smallskip


\subsection{The machine $\textbf{M}_3$} \

The standard base of the recognizing $S$-machine $\textbf{M}_3$ is the same as that of $\overline{\textbf{M}}_2$, with the names `input', `working', and `historical' assigned to sectors in the same way. However, each part of the state letters has more letters than the corresponding part of the state letters of $\overline{\textbf{M}}_2$.

The idea of the function of $\textbf{M}_3$ is the following. Consider the start configuration of $\overline{\textbf{M}}_2$ with a word $w$ in the input $R_0P_1$-sector, copies of a history word $H_1\in\Phi^+$ in the left alphabets of the historical $R_iP_{i+1}$-sectors, and all other sectors empty. The machine first executes $\textbf{RL}$ in the appropriate historical sectors, running the letters of $R_i$ right and then left, and then executes the computation of $\overline{\textbf{M}}_2$ whose history is the natural copy of $H_1$. When the letters of the historical sectors have been converted to the right alphabet, the machine then executes copies of $\textbf{LR}$ in the appropriate historical sectors, moving the letters of $P_{i+1}$ left and then right, then executes the computation of $\overline{\textbf{M}}_2$ whose history is a copy of $H_1^{-1}$. These four steps result in a copy of the same start configuration of $\overline{\textbf{M}}_2$ (but with different state letters). The machine then repeats these four steps $k$ times and finishes by running the first three one more time. 

As a result, the machine $\textbf{M}_3$ is a concatenation of $4k+3$ different machines, which are denoted $\textbf{M}_3(1),\dots,\textbf{M}_3(4k+3)$, each of which has a distinct copy of the standard base of $\overline{\textbf{M}}_2$. So, each part of the state letters of $\textbf{M}_3$ is the disjoint union of $4k+3$ sets corresponding to the parts of the state letters of $\textbf{M}_3(1),\dots,\textbf{M}_3(4k+3)$. 

To force the correct order of this concatenation, there are \textit{transition rules} $\chi(i,i+1)$, which switch the state letters from the end letters of one machine to the start letters of the next. Further, $\chi(i,i+1)$ locks any sector that is locked by either the rules of $\textbf{M}_3(i)$ or the rules of $\textbf{M}_3(i+1)$.

The following is a detailed description of the concatenated machines and transition rules:

\begin{addmargin}[1em]{0em}

$\bullet$ The machine $\textbf{M}_3(1)$ consists of the \textit{parallel work} of copies of the machine $\textbf{RL}=\textbf{RL}(\Phi^+)$ in the appropriate historical sectors. As such, if the $Q_iQ_{i+1}$-sector of $\textbf{M}_2$ is historical, then the subword $Q_iR_iP_{i+1}$ of the standard base of $\textbf{M}_3(1)$ is treated as a copy of the standard base of $\textbf{RL}$. The term ``parallel work" means that each rule of $\textbf{M}_3(1)$ executes the corresponding rule of $\textbf{RL}$ simultanelously in each of these copies, identifying the right alphabet of the $Q_iR_i$-sector and the left alphabet of the $R_iP_{i+1}$-sector with the corresponding tape alphabets of $\textbf{RL}$. So, for any rule of $\textbf{M}_3(1)$, its domain in the historical $Q_iR_i$-sectors (respectively $R_iP_{i+1}$-sectors) is the corresponding copy of $X_{i,r}$ (respectively $X_{i,\ell}$). Additionally, each rule locks every other sector except for the input sector. Each part of the state letters not present in a copy of $\textbf{RL}$ consists of one letter. Start and end letters are assigned in the obvious way.

$\bullet$ The transition rule $\chi(1,2)$ locks every sector except for the input sector and the historical $R_iP_{i+1}$-sectors and switches the state letters from the end state of $\textbf{M}_3(1)$ to the start state of $\textbf{M}_3(2)$. Additionally, the domain of this rule in each historical $R_iP_{i+1}$-sector is the corresponding left alphabet.

$\bullet$ $\textbf{M}_3(2)$ is a copy of the machine $\overline{\textbf{M}}_2$.

$\bullet$ The transition rule $\chi(2,3)$ locks all sectors except for the historical $R_iP_{i+1}$-sectors and changes the state letters from the end letters of $\textbf{M}_3(2)$ to the start letters of $\textbf{M}_3(3)$. Its domain in each historical $R_iP_{i+1}$-sector is the corresponding right alphabet.

$\bullet$ $\textbf{M}_3(3)$, similar to $\textbf{M}_3(1)$, works in parallel as $\textbf{LR}=\textbf{LR}(\Phi^+)$ in the historical sectors, but now with the subwords $R_iP_{i+1}Q_{i+1}$ functioning as the copy of the standard base of $\textbf{LR}$. The copy of $X_{i,r}$ (respectively $X_{i,\ell}$) in the $R_iP_{i+1}$-sector (respectively $P_{i+1}Q_{i+1}$-sector) is identified with the corresponding tape alphabet of $\textbf{LR}$. Otherwise, $\textbf{M}_3(3)$ is defined in a similar way as $\textbf{M}_3(1)$.

$\bullet$ The transition rule $\chi(3,4)$ locks all sectors except for the historical $R_iP_{i+1}$-sectors and changes the state letters from the end letters of $\textbf{M}_3(3)$ to the start letters of $\textbf{M}_3(4)$. Its domain in each historical $R_iP_{i+1}$-sector is the corresponding right alphabet.

$\bullet$ $\textbf{M}_3(4)$ is then a copy of the `inverse' machine $\overline{\textbf{M}}_2^{-1}$. This means the hardware is a copy of $\overline{\textbf{M}}_2$, while its positive rules are copies of the negative rules of $\overline{\textbf{M}}_2$.

$\bullet$ The transition rule $\chi(4,5)$ locks every sector other than the input and the historical $R_iP_{i+1}$-sectors and changes the state letters from the end state of $\textbf{M}_3(4)$ to the start state of $\textbf{M}_3(5)$. Its domain in each historical $R_iP_{i+1}$-sector is the corresponding left alphabet.

$\bullet$ For $2\leq i\leq k$, the machines $\textbf{M}_3(4i-3),\dots,\textbf{M}_3(4i)$ are copies of $\textbf{M}_3(1),\dots,\textbf{M}_3(4)$, respectively, with transition rules defined similarly.

$\bullet$ Finally, the machines $\textbf{M}_3(4k+1)$, $\textbf{M}_3(4k+2)$, and $\textbf{M}_3(4k+3)$ are copies of $\textbf{M}_3(1)$, $\textbf{M}_3(2)$, and $\textbf{M}_3(3)$, respectively. 

\end{addmargin}

Note that the $Q_0R_1$-sector is locked by every rule of $\textbf{M}_3$.

Similar to the definition of $I_2(w,H_1)$, the configuration $I_3(w,H_1)$ of $\textbf{M}_3$ is the start configuration with $w\in F(\pazocal{A})$ written in the input sector and copies of $H_1\in F(\Phi^+)$ written in the left alphabet of each of the historical $R_iP_{i+1}$-sectors. 

Also, the configuration $A_3(H_1)$ is the end configuration of $\textbf{M}_3(4k+3)$ with $H_1\in F(\Phi^+)$ written in the right alphabets of the historical $R_iP_{i+1}$-sector. Note that $A_3(H_1)$ is then the configuration obtained from $A_2(H_1)$ by inserting appropriate $P$- and $R$-letters next to the $Q$-letters.

A configuration of $\textbf{M}_3$ is \textit{tame} if each $P$- and $R$-letter is adjacent to a $Q$-letter (i.e all $P_iQ_i$- and $Q_iR_i$-sectors are empty) and every historical sector contains no letters from either the left or the right alphabets. Note that for all $w\in F(\pazocal{A})$ and $H_1\in F(\Phi^+)$, the configurations $I_3(w,H_1)$ and $A_3(H_1)$ are tame.

\begin{lemma} \label{M_3 tame}

\textit{(Lemma 3.14 of [26])} Let $\pazocal{C}:W_0\to\dots\to W_t$ be a reduced computation of $\textbf{M}_3(2i+1)$ in the standard base for some $0\leq i\leq 2k+1$. Then:

\begin{addmargin}[1em]{0em}

$(a)$ $|W_j|_a\leq\max(|W_0|_a,|W_t|_a)$ for every $j$; 

moreover, if $W_0$ is tame, then $|W_0|_a\leq|W_1|_a\leq\dots\leq|W_t|_a$

$(b)$ $t\leq\|W_0\|+\|W_t\|-5$;

moreover, if $W_0$ is tame, then $t\leq2\|W_t\|-5$.

\end{addmargin}

\end{lemma}

\begin{lemma} \label{M_3 historical sector}

\textit{(Lemma 3.15 of [26])} Let $W_0\to\dots\to W_t$ be a reduced computation of $\textbf{M}_3$ with base containing a historical sector of the form $(R_iP_{i+1})^{\pm1}$. Then for each $j$, there is at most one occurrence of the transition rules $\chi(j,j+1)^{\pm1}$ in the history $H$ of the computation.

\end{lemma}

%
%
%

\begin{lemma} \label{M_3 standard base}

\textit{(Lemma 3.16 of [26])} Let $\pazocal{C}:W_0\to\dots\to W_t$ be a reduced computation of $\textbf{M}_3$ in the standard base with history $H$.

\begin{addmargin}[1em]{0em}

$(a)$ If $H\equiv\chi(i,i+1)H'\chi(i+4,i+5)$, then the word $W_0$ is a copy of $W_t$

$(b)$ If $\pazocal{C}$ has two subcomputations $\pazocal{C}_1$ and $\pazocal{C}_2$ with histories $\chi(i,i+1)H'\chi(i+4,i+5)$ and $\chi(j,j+1)H''\chi(j+4,j+5)$, then they have equal lengths and $\pazocal{C}_2$ is a cyclic permutation of a copy of $\pazocal{C}_1$

$(c)$ for the parameter $c_1$, $|W_j|_a\leq c_1\max(|W_0|_a,|W_t|_a)$ for $j=0,1,\dots,t$; \\ moreover, if $W_0$ is a tame, then $|W_j|_a\leq c_1|W_t|_a$ for all $j$.

\end{addmargin}

\end{lemma}

\begin{lemma} \label{M_3 accepting computations} 

Let $H_1\in F(\Phi^+)$ and $W_0\equiv A_3(H_1)$ or $W_0\equiv I_3(w,H_1)$ for some $w\in F(\pazocal{A})$. If $\pazocal{C}:W_0\to\dots\to W_t\equiv W$ is a reduced computation, then $t\leq c_1|V|_a$ where $V$ is any admissible subword of $W$ with base $Q_iR_iP_{i+1}Q_{i+1}$ where $R_iP_{i+1}$ is a historical sector.

\end{lemma}

\begin{proof}

Suppose $W_0\equiv A_3(H_1)$. Lemma \ref{M_3 standard base} implies that the history $H$ of $\pazocal{C}$ has the form $$H=h_{4k+3}\chi(4k+3,4k+2)h_{4k+2}\dots\chi(j+1,j)h_j$$ where $1\leq j\leq 4k+3$, $\chi(i+1,i)\equiv\chi(i,i+1)^{-1}$, and $h_i$ is the history of a reduced computation of $\textbf{M}_3(i)$ for $j\leq i\leq 4k+3$. Lemmas \ref{M_3 standard base}, \ref{primitive computations}, \ref{multiply two letters}, and \ref{one alphabet historical words} then imply that for $j+1\leq i\leq 4k+3$, $h_i$ has length bounded by $\|H_1\|$. So, $\|H\|-\|h_j\|\leq (4k+2)(\|H_1\|+1)$, so that taking $c_1>>k$ proves the statement if $\|h_j\|=0$.

Supposing $h_j$ is nonempty, let $W_s\to\dots\to W_t\equiv W$ be the subcomputation with this history. Let $V_s\to\dots\to V_t$ be the restriction of this subcomputation to:

\begin{addmargin}[1em]{0em}

$\bullet$ a three-letter subword $Q_iR_iP_{i+1}$ for $R_iP_{i+1}$ historical if $j=4r+1$ for some integer $r$

$\bullet$ a three-letter subword $R_iP_{i+1}Q_{i+1}$ for $R_iP_{i+1}$ historical if $j=4r+3$ for some integer $r$

$\bullet$ a historical $R_iP_{i+1}$-sector if $j$ is even

\end{addmargin}

Note that $|V_s|_a=\|H_1\|$ in each case. Lemmas \ref{primitive computations} and \ref{multiply two letters} then imply that in each case, there exists $s\leq l\leq t$ such that $\|H_1\|=|V_s|_a=\dots=|V_l|_a<|V_{l+1}|_a<\dots<|V_t|_a$. The same lemmas then imply that $l-s\leq2\|H_1\|+1$, while $2(t-l)+\|H_1\|=|V_t|_a$. So, since $s\leq (4k+2)(\|H_1\|+1)$ as above, $t\leq(4k+2)(\|H_1\|+1)+2\|H_1\|+1+(|V_t|_a-\|H_1\|)/2$.

Since $|W|_a\geq|V_t|_a\geq\|H_1\|$, the statement holds for sufficiently large $c_1$.

A similar argument implies the statement if $W_0\equiv I_3(w,H_1)$.

\end{proof}

\begin{lemma} \label{M_3 language}

$(1)$ For any $u^n\in\pazocal{L}$, there exists $H_1\in F(\Phi^+)$ with $\|H_1\|\leq c_0\|u\|$ and a reduced computation  $I_3(u^n,H_1)\to\dots\to A_3(H_1)$ of $\textbf{M}_3$ of length at most $c_1\|u\|$.

$(2)$ If there exists a reduced computation $I_3(w,H_1)\to\dots\to A_3(H_1')$ of $\textbf{M}_3$ for $w\in F(\pazocal{A})$, then $w\in\pazocal{L}$ and $H_1'\equiv H_1$.

\end{lemma}

\begin{proof}

(1) The existence of a computation of the form $I_3(u^n,H_1)\to\dots\to A_3(H_1)$  follows from the definition of the machine and Lemmas \ref{M_2 language}(1) and \ref{primitive computations}(3), as one can concatenate appropriate computations of $\overline{\textbf{M}}_2^{\pm1}$, $\textbf{LR}$, and $\textbf{RL}$ (connected by transition rules). As the appropriate computations of $\overline{\textbf{M}}_2^{\pm1}$ are of length at most $c_0\|u\|$ and those of $\textbf{LR}$ and $\textbf{RL}$ are of length $2\|H_1\|+1\leq2\|u\|+1$, the bound on the length of the computation is given by the choice of paramaters.

(2) The initial configuration has state letters of the machine $\textbf{M}_3(1)$ while the final configuration has those of $\textbf{M}_3(4k+3)$. Note that passing from state letters of one machine to those of the next requires the presence of a transition rule. Lemma \ref{M_3 historical sector} then implies that the history of the computation must be of the form $H_1''\chi(1,2)H_2''\chi(2,3)\dots\chi(4k+2,4k+3)H_{4k+3}''$, where $H_i''$ is the history of a reduced computation in $\textbf{M}_3(i)$. Applying Lemmas \ref{M_2 language}(2) and \ref{primitive computations} to the subcomputations of history $H_i''$ implies the statement.

\end{proof}

\smallskip


\subsection{The machine $\textbf{M}_4$} \

As in [26], the machine $\textbf{M}_4$ is a composition of $\textbf{M}_3$ with a `mirror copy' of itself, in the sense of what is defined below.

Let $B_3$ be the standard base of $\textbf{M}_3$ and $B_3'$ a copy of $B_3$. Then the standard base of $\textbf{M}_4$ is $B_3(B_3')^{-1}$. Given the start (or end) letter of a part of the standard base of $\textbf{M}_3$, the start (or end) letter of the corresponding part of $B_3$ is the same while that of the corresponding part of $(B_3')^{-1}$ is the inverse of its copy.

The tape alphabet of a sector formed by a two-letter subword of $B_3$ is the same as the corresponding sector in $\textbf{M}_3$, while that of a sector formed by a subword of $(B_3')^{-1}$ is the same as that of the sector corresponding to its inverse. The tape alphabet formed by the last letter of $B_3$ and the first letter of $(B_3')^{-1}$ is empty, so that every rule locks it.

The positive rules of $\textbf{M}_4$ are in one-to-one correspondence with the positive rules of $\textbf{M}_3$, with each executing in parallel on $B_3$ and $(B_3')^{-1}$ in the same way as its corresponding rule. For example, if a positive rule of $\textbf{M}_3$ has part $q\to aq$, then the corresponding positive rule of $\textbf{M}_4$ has parts $q\to aq$ and $(q')^{-1}\to(q')^{-1}(a')^{-1}$, where $q'$ and $a'$ are the corresponding copies of $q$ and $a$.

Note the mirror symmetry of this construction. As such, given a sector formed by a two-letter subword of $B_3$, we refer to the corresponding sector formed by a subword of $(B_3')^{-1}$ as its mirror copy (and vice versa) and call $(B_3')^{-1}$ the mirror copy of $B_3$.

A sector is historical (or working) in the machine $\textbf{M}_4$ if it corresponds to a historical (or working) sector of $\textbf{M}_3$. However, the only input sector is the the $R_0P_1$-sector; in particular, its mirror copy in $(B_3')^{-1}$ is not considered an input sector.

Note that a configuration of $\textbf{M}_4$ is merely the concatenation of a configuration of $\textbf{M}_3$ with a copy of the inverse of one. So, to each configuration of $\textbf{M}_4$, one can associate two configurations of $\textbf{M}_3$ which completely define it. As such, a configuration of $\textbf{M}_4$ is called \textit{tame} if both of its associated configurations are tame in $\textbf{M}_3$.

There are obvious analogues of Lemmas \ref{M_3 tame} through \ref{M_3 standard base} in the setting of $\textbf{M}_4$. In particular, the following analogue of Lemma \ref{M_3 standard base}$(c)$ is stated here for easy reference.

\begin{lemma} \label{M_4 bounds}

Let $\pazocal{C}:W_0\to\dots\to W_t$ be a reduced computation of $\textbf{M}_4$ in the standard base. Then $|W_j|_a\leq c_1\max(|W_0|_a,|W_t|_a)$ for $j=0,1,\dots,t$; moreover, $|W_j|_a\leq c_1|W_t|_a$ if $W_0$ is a tame configuration.

\end{lemma}

Similarly, there is an obvious analogue of Lemma \ref{M_3 language} for the machine $\textbf{M}_4$. To this end, for $w\in F(\pazocal{A})$ and $H_1\in F(\Phi^+)$, define $I_4(w,H_1)$ (or $A_4(H_1)$) as the configuration of $\textbf{M}_4$ whose associated configurations in $\textbf{M}_3$ are both $I_3(w,H_1)$ (or $A_3(H_1)$).

\begin{lemma} \label{M_4 language}

(1) For any $u^n\in\pazocal{L}$, there exists an $H_1\in F(\Phi^+)$ with $\|H_1\|\leq c_0\|u\|$ and a reduced computation $I_4(u^n,H_1)\to\dots\to A_4(H_1)$ of $\textbf{M}_4$ of length at most $c_1\|u\|$.

(2) If there exists a reduced computation $I_4(w,H_1)\to\dots\to A_4(H_1')$ of $\textbf{M}_4$ for $w\in F(\pazocal{A})$, then $w\in\pazocal{L}$ and $H_1'\equiv H_1$.

\end{lemma}

\smallskip


\subsection{The machine $\overline{\textbf{M}}_4$} \

The machine $\overline{\textbf{M}}_4$ is a simple tweak to the machine $\textbf{M}_4$. Its role in the scope of our construction is similar to that of the machine $\overline{\textbf{M}}_2$.

The standard base of $\overline{\textbf{M}}_4$ adds just one part to that of $\textbf{M}_4$. In particular, setting $B_4$ as the standard base of $\textbf{M}_4$, the standard base of $\overline{\textbf{M}}_4$ is $\{t\}B_4$, where $\{t\}$ consists of a single letter (which, clearly, acts as both the start and end letter of its part). The length of this standard base is henceforth the parameter $N$.

The tape alphabet of the new sector in the standard base, i.e the $\{t\}P_0$-sector, is empty. All other tape alphabets are carried over from $\textbf{M}_4$. 

The positive rules of $\overline{\textbf{M}}_4$ then correspond to those of $\textbf{M}_4$, operating on the copy the hardware of $\textbf{M}_4$ in the same way and locking the new sector.

The input sector is the same as that of $\textbf{M}_4$. The corresponding definitions (for example, historical sector, working sector, etc) then extend in the clear way, as do all lemmas of Section 4.6.

\smallskip


\subsection{The machine $\textbf{M}_5$} \

Similar to the main machine of [26], the recognizing $S$-machine $\textbf{M}_5$ is created from $\overline{\textbf{M}}_4$, $\textbf{LR}_k$ (where $k$ is the parameter specified in Section 3.3), and three more simple machines. The standard base of the machine is the same as that of $\overline{\textbf{M}}_4$, but with each part other than the one-letter part $\{t\}$ consisting of more state letters.

The rules partition $\textbf{M}_5$ into five different machines, denoted $\textbf{M}_5(1),\dots,\textbf{M}_5(5)$, with the transition rules $\theta(i,i+1)$ functioning in a similar way as the rules $\chi(i,i+1)$ did in the construction of $\textbf{M}_3$ in Section 4.5. Accordingly, each part of the state letters (other than the one-letter part $\{t\}$) is the disjoint union of five sets corresponding to these five machines; for this construction, though, we also include one additional letter in each part of the state letters (other than the one-letter part $\{t\}$), which functions as the end state letter of that part.

One major difference between $\textbf{M}_5$ and its predecessors is that the $Q_0R_0$-sector functions as the input sector (as opposed to the $R_0P_1$-sector).

Each of the rules locks the $\{t\}P_0$-sector and operates on the mirror copy of $B_3$ in the symmetric way. As such, when defining the rules below, we detail only the parts of the rules with state letters from $B_3$, implicitly defining the rest of the rule.

The positive rules of $\textbf{M}_5(1)$ are in correspondence with $\pazocal{A}$. The rule corresponding to the letter $a\in\pazocal{A}$ writes a copy of $a$ on the left of the $R_0P_1$-sector and a copy of its inverse on the right of the $Q_0R_0$-sector. Each part of the state letters has just one letter corresponding to this submachine.

The transition rule $\theta(12)$ locks all sectors except for the $R_0P_1$-sector and changes the state letters from the end (only) letters of $\textbf{M}_5(1)$ to the start letters of $\textbf{M}_5(2)$.

The machine $\textbf{M}_5(2)$ operates as the machine $\textbf{LR}_k$ on the subword $R_0P_1Q_1$ of the standard base. Each other part of the state letters consists of just one letter corresponding to this submachine and all other sectors are locked.

The transition rule $\theta(23)$ locks all sectors except for the $R_0P_1$-sector and changes the state letters from the end letters of $\textbf{M}_5(2)$ to the start letters of $\textbf{M}_5(3)$.

The machine $\textbf{M}_5(3)$ inserts history words in the left alphabets of the historical $R_iP_{i+1}$-sectors. As a result, the positive rules of $\textbf{M}_5(3)$ are in one-to-one correspondence with $\Phi^+$, with such a rule inserting the left copy of the corresponding rule to the left of each appropriate $P_{i+1}$-letter. Each rule locks all other sectors except for the $R_0P_1$-sector. The state letters are fixed by every rule, so that each part contains just one state letter corresponding to $\textbf{M}_5(3)$.

The transition rule $\theta(34)$ locks all sectors except for the $R_0P_1$- and historical $R_iP_{i+1}$-sectors and changes the state letters to the start letters of $\textbf{M}_5(4)$. Its domain in each unlocked historical sector is the corresponding left alphabet.

The machine $\textbf{M}_5(4)$ is a copy of the machine $\overline{\textbf{M}}_4$.

The transition rule $\theta(45)$ locks all sectors except for historical $R_iP_{i+1}$-sectors and changes the state letters to the start (only) state letters of $\textbf{M}_5(5)$. The domain of the rule in each historical $R_iP_{i+1}$-sector is the corresponding left alphabet.

The machine $\textbf{M}_5(5)$ acts similart to the inverse of $\textbf{M}_5(3)$, erasing the history words in the historical $R_iP_{i+1}$-sectors. However, these rules write in the right alphabets and lock every other sector, particularly the $R_0P_1$-sector (in contrast to the rules of $\textbf{M}_5(3)$). Similar to $\textbf{M}_5(3)$, each part of the state letters contains just one letter corresponding to $\textbf{M}_5(5)$.

Finally, the accept rule $\theta_0$ locks every sector and switches from the end (only) state letters of $\textbf{M}_5(5)$ to the end state letters of the machine. These end state letters form the accept configuration $A_5$.

Note that the machine $\textbf{M}_5$ is very similar to the main machine of [26], with only the first machine $\textbf{M}_5(1)$ altered. As such, many of the lemmas in the next few subsections will be presented with partial or no proof.

\medskip


\subsection{Standard computations of $\textbf{M}_5$} \

For simplicity of notation, denote the inverse of each of the transition rules by switching the indices, i.e so that $\theta(i,i+1)^{-1}=\theta(i+1,i)$. These inverses are also referred to as transition rules.

The history $H$ of a reduced computation of $\textbf{M}_5$ can be factorized in such a way that each factor is either a transition rule or a maximal nonempty product of rules of one of the five defining machines $\textbf{M}_5(1),\dots,\textbf{M}_5(5)$. The \textit{step history} of a reduced computation is then defined so as to capture the order of the types of these factors. To do this, denote the transition rule $\theta(i,j)$ by the pair $(i,j)$ and a factor that is a product of rules in $\textbf{M}_5(i)$ simply by $(i)$. 

For example, if $H\equiv H'H''H'''$ where $H'$ is a product of rules from $\textbf{M}_5(2)$, $H''\equiv\theta(23)$, and $H'''$ is a product of rules from $\textbf{M}_5(3)$, then the step history of a computation with history $H$ is $(2)(23)(3)$. So, the step history of a computation is some concatenation of the letters
$$\{(1),(2),(3),(4),(5),(12),(23),(34),(45),(21),(32),(43),(54)\}$$ 
Note that there is no reference to the accept rule $\theta_0$ or its inverse in the step history. This is because $\theta_0^{\pm1}$ can only appear as the first or last letter of the history of a reduced computation, in which case any subsequent rule is either the transition rule $\theta(54)$ or from $\textbf{M}_5(5)$. As such, one can view $\theta_0^{\pm1}$ as a rule of $\textbf{M}_5(5)$ when creating the step history, so that it contributes to maximal subcomputations of step history $(5)$.

What's more, one can omit reference to a transition rule when its existence is clear from its necessity. For example, given a reduced computation with step history $(2)(23)(3)$, one can instead write the step history as $(2)(3)$, as the rule $\theta(23)$ must occur in order for the subcomputation of $\textbf{M}_5(3)$ to be possible.

If the step history of a computation is $(i-1,i)(i,i+1)$, it is also permitted for the step history to be written as $(i-1,i)(i)(i,i+1)$ even though the maximal `subcomputation' with step history $(i)$ is empty.

A \textit{one-step computation} is a computation with step history of one of the following forms:

\begin{addmargin}[1em]{0em}

$\bullet$ $(i)$

$\bullet$ $(i)(i,i\pm1)$

$\bullet$ $(i\pm1,i)(i)$

$\bullet$ $(i\pm1,i)(i)(i,i\pm1)$

\end{addmargin}

Certain subwords cannot appear in the step history of a reduced computation. For example, it is clear that it is impossible for the step history of a reduced computation to contain the subword $(1)(3)$. The next few lemmas display the impossibility of some less obvious potential subwords.

\begin{lemma} \label{first step history}

\textit{(Lemma 4.2 of [26])} If the base of a reduced computation $\pazocal{C}$ of $\textbf{M}_5$ has at least one historical subword $UV$ corresponding to a historical $(R_iP_{i+1})^{\pm1}$-sector (or its mirror copy), then its step history is not:

\begin{addmargin}[1em]{0em}

$(a)$ $(34)(4)(43)$ or $(54)(4)(45)$

$(b)$ $(23)(3)(32)$

\end{addmargin}

\end{lemma}

%
%
%
%
%

\begin{lemma} \label{second step history}

\textit{(Lemma 4.3 of [26])} There are no reduced computations of $\textbf{M}_5$ in the standard base whose step history is $(32)(2)(23)$, $(12)(2)(21)$, or $(21)(1)(12)$.

\end{lemma}

\begin{proof}


Suppose the step history is $(21)(1)(12)$. Let $\pazocal{C}'$ be the maximal subcomputation with step history $(1)$. Then the restriction of $\pazocal{C}'$ to the two-letter base $Q_0R_0$ (i.e to the input sector) satisfies the hypotheses of Lemma \ref{multiply one letter}. Since $\theta(12)$ locks the input sector, applying Lemma \ref{multiply one letter}$(a)$ to this restricted subcomputation implies that it must be empty, yielding a contradiction.

The proofs of the other two cases can be found in [26].

\end{proof}

The previous two lemmas quickly imply the following:

\begin{lemma} \label{(A) and (B)}

\textit{(Lemma 4.5 of [26])} The step history of every reduced computation of $\textbf{M}_5$ with standard base either:

\begin{addmargin}[1em]{0em}

(A) contains a subword of the form $(34)(4)(45)$, $(54)(4)(43)$, $(12)(2)(23)$, or $(32)(2)(21)$

(B) is a subword of one of the words
$$(4)(45)(5)(54)(4), \ (4)(43)(3)(34)(4), \ (2)(23)(3)(34)(4),$$ $$(4)(43)(3)(32)(2), \ (2)(21)(1), \text{ or } (1)(12)(2)$$

\end{addmargin}

\end{lemma}

%
%
%
%
%

The history of a reduced computation of $\textbf{M}_5$ is called \textit{controlled} if it is of one of the following two forms:

\begin{addmargin}[1em]{0em}

$(a)$ $\chi(4i+2,4i+3)H'\chi(4i+3,4i+4)$ for some $0\leq i\leq k$, i.e the computation has step history (4) and works as $\overline{\textbf{M}}_4$, with the subcomputation of history $H'$ operating (in parallel) as $\textbf{LR}$


$(b)$ $\zeta^{(2i,2i+1)}H'\zeta^{(2i+1,2i+2)}$, i.e the computation has step history (2) and works (in parallel) as $\textbf{LR}_k$

\end{addmargin}

\begin{lemma} \label{M_5 controlled}

\textit{(Lemma 4.4 of [26])} Let $\pazocal{C}:W_0\to\dots\to W_t$ be a reduced compuation of $\textbf{M}_5$ with controlled history $H$. Then the base of the computation is a reduced word and all configurations are uniquely defined by the history $H$ and the base of $\pazocal{C}$.\newline Moreover, if $\pazocal{C}$ is a computation in the standard base, then $|W_0|_a=\dots=|W_t|_a$ and $\|H\|=2l+3$ (respectively $\|H\|=l+2$), where $l$ is the $a$-length of each historical $R_iP_{i+1}$-sector (respectively of the input sector) of $W_0$ if $H$ is of the form $(a)$ (respectively $(b)$).

\end{lemma}

%
%
%

Define $I_5(w_1,w_2)$ as the start configuration with $w_1\in F(\pazocal{A})$ inserted in the input sector and a copy of $w_2^{-1}\in F(\pazocal{A})$ in its mirror copy. Similarly, define $I_5'(w_1,w_2)$ as the $\theta(21)$-admissible configuration with $w_1$ inserted in the $R_0P_1$-sector and a copy of its inverse inserted in the mirror copy.

\begin{lemma} \label{M_5 language}

$(1)$ For every $u^n\in\pazocal{L}$, $I_5(u^n,u^n)$ is accepted by $\textbf{M}_5$.

$(2)$ If $I_5(w_1,w_2)$ is accepted by $\textbf{M}_5$, then $w_1\equiv w_2\in\pazocal{L}$.

\end{lemma}

\begin{proof}

(1) First, consider the computation $I_5(u^n,u^n)\to\dots\to W_1$ of $\textbf{M}_5(1)$ whose history is a copy of $u^n$ read right to left. By Lemma \ref{multiply one letter}, $W_1\cdot \theta(12)\equiv I_5'(u^n,u^n)$.

Next, consider the analogue of the computation of Lemma \ref{primitive computations}(3) to the machine $\textbf{LR}_k$. Letting $H_2$ be the history of this computation, it follows that $I_5'(u^n,u^n)\cdot H_2$ is a copy of $I_5'(u^n,u^n)$. Set $W_2\equiv I_5'(u^n,u^n)\cdot(H_2\theta(23))$.

Let $H_1\in F(\Phi^+)$ be the word corresponding to $u^n$ in Lemma \ref{M_4 language}(1) (or Lemma \ref{M_1 language}). Identifying $H_1$ with the rules of $\textbf{M}_5(3)$ and setting $W_3=W_2\cdot H_1\theta(34)$, it then follows that $W_3$ is the configuration obtained from $I_4(u^n,H_1)$ by adjoining $t$ to the front of it.

By Lemma \ref{M_4 language}(1) and the definition of the rules of $\textbf{M}_5(4)$, there exists a computation of $\textbf{M}_5(4)$ with history $H_4$ and initial configuration $W_3$ and final configuration the analogue of $A_4(H_1)$. Letting $W_4\equiv W_3\cdot(H_4\theta(45))$, there exists a reduced computation of $\textbf{M}_5(5)$ with initial configuration $W_4$ and whose final configuration $W_5$ is $\theta_0$-admissible.

Concatenating these computations yields a reduced computation $I_5(u^n,u^n)\to\dots\to W_5\to A_5$.

(2) As every rule of $\textbf{M}_5$ operates identically on the input sector and its mirror image, it follows immediately that $w_1\equiv w_2$. Let $w\equiv w_1$.

By Lemmas \ref{first step history} through \ref{(A) and (B)}, the step history of an accepting computation for $I_5(w,w)$ must have prefix $(1)(12)(2)(23)(3)(34)(4)(45)$, i.e the history has prefix $H_1\theta(12)H_2\theta(23)H_3\theta(34)H_4\theta(45)$. Lemma \ref{multiply one letter} implies that $I_5(w,w)\cdot H_1\theta(12)\equiv I_5'(w,w)$; the analogue of Lemma \ref{primitive computations}(3) for the machine $\textbf{LR}_k$ then implies that $W_2\equiv I_5'(w,w)\cdot H_2(23)$ is a copy of $I_5'(w,w)$; it then follows that $W_3\equiv W_2\cdot H_3\theta(34)$ is the analogue of $I_4(w,H)$ for $H\in F(\Phi^+)$ the natural copy of $H_3$. But that the subcomputation with initial configuraiton $W_3$ and history $H_4$ results in a $\theta(45)$-admissible word  implies that $w\in\pazocal{L}$ by Lemma \ref{M_4 language}(2).

\end{proof}

Motivated by Lemma \ref{M_5 language}, the simpler notation $I_5(w)$ will be used in place of $I_5(w,w)$.

\begin{lemma} \label{M_5 tame}

\textit{(Lemma 4.7 of [26])} Let $\pazocal{C}:W_0\to\dots\to W_t$ be a reduced computation in the standard base of $\textbf{M}_5$ that either satisfies (B) of Lemma \ref{(A) and (B)} or has step history of length at most 2. Then for the parameter $c_2$:

\begin{addmargin}[1em]{0em}

$(a)$ $|W_j|_a\leq c_2\max(|W_0|_a,|W_t|_a)$ for $0\leq j\leq t$

$(b)$ $t\leq c_2^2(\|W_0\|+\|W_t\|)$

\end{addmargin}

\end{lemma}

\begin{proof}

The only alteration needed to the proof presented in [26] is the observation that a reduced computation of $\textbf{M}_5(1)$ satisfies the hypotheses of Lemma \ref{multiply one letter}.

\end{proof}

\begin{lemma} \label{M_5 accepted configurations}

For every accepted configuration $W$ of $\textbf{M}_5$, there exists an accepting computation of length at most $c_3\|W\|$ and such that the $a$-length of every configuration does not exceed $c_3|W|_a$.

\end{lemma}

\begin{proof}

Let $\pazocal{C}$ be an accepting computation of $W$ with step history of minimal length and $\bar{\pazocal{C}}$ be its inverse computation. Set $H$ as the history of $\bar{\pazocal{C}}$.

Then, there exists a maximal subcomputation $A_5\equiv W_0\to\dots\to W_r$ of $\bar{\pazocal{C}}$ with step history $(5)$. Applying Lemma \ref{multiply one letter} to $W_0\to\dots\to W_r$ implies that $r\leq\|W_r\|$ and $|W_j|_a\leq |W_r|_a$ for all $0\leq j\leq r$, so that the statement follows if $t=r$ (i.e $\pazocal{C}$ is a one-step computation of step history $(5)$). 

If $t>r$, then there exists a maximal subcomputation $A_5\equiv W_0\to\dots\to W_l$ of $\bar{\pazocal{C}}$ whose step history is $(5)(4)$. Then $W_{r+1}\to\dots\to W_l$ is a one-step computation of step history $(4)$ with $W_{r+1}\equiv A_4(H)$ for some $H\in F(\Phi^+)$. Lemmas \ref{M_3 standard base} and \ref{M_3 accepting computations} then apply to this computation, yielding $|W_j|_a\leq c_1|W_l|_a$ for all $r+1\leq j\leq l$ and $l-r-1\leq c_1\|W_l\|$. So, $l\leq(c_1+1)\|W_l\|$ and $|W_j|_a\leq|W_l|_a$ for all $0\leq j\leq l$ since $\|W_r\|=\|W_{r+1}\|$, implying the statement if $t=l$.

Suppose the length of $\bar{\pazocal{C}}$, $t$, is greater than $l$.  Suppose further that the step history of the subcomputation $W_l\to\dots\to W_t$ has prefix $(43)(3)(34)(4)(45)$. Then, let $W_l\to\dots\to W_x$ be the maximal subcomputation with this step history. Then we can replace the subcomputation $W_r\to\dots\to W_x$ with a computation of step history $(5)$ described in Lemma \ref{multiply one letter}. This reduces the length of the step history, contradicting the assumption that $\pazocal{C}$ is minimal in regards to step history length.

Lemmas \ref{first step history} and \ref{second step history} then imply that the step history of the subcomputation $W_l\to\dots\to W_t$ must be a prefix of either: 
\begin{addmargin}[1em]{0em}

$\bullet$ $(43)(3)(34)(4)$ or

$\bullet$ $(43)(3)(32)(2)(21)(1)$

\end{addmargin}

As $W_l$ is $\theta(43)$-admissible and the operation of each rule of $W_0\to\dots\to W_l$ acts in the same way on each historical $R_iP_{i+1}$-sector, we must have $W_l\equiv I_4(w,H_1'')$ for some $w\in F(\pazocal{A})$ and $H_1''\in F(\Phi^+)$. Similarly, we have $W_{r+1}\equiv A_4(H_1')$ for some $H_1'\in F(\Phi^+)$. But then $W_{r+1}\to\dots\to W_l$ is a computation of $\overline{\textbf{M}}_4$ of the form $A_4(H_1')\to\dots\to I_4(w,H_1'')$, so that Lemma \ref{M_4 language}(2) implies that $w\in\pazocal{L}$ and $H_1'\equiv H_1''$. Setting $w=u^n$, suppose $\|H_1'\|>c_0\|u\|$. Lemma \ref{M_4 language}(1) then provides $H_1\in F(\Phi^+)$ with $\|H_1\|\leq c_0\|u\|$ and a computation $A_4(H_1)\to\dots\to I_4(w,H_1)$ of length at most $c_1\|u\|$. Identifying $H_1$ with the rules of $\textbf{M}_5(1)$ then gives a computation $A_5\to\dots\to I_4(w,H_1)$ of length at most $\|H_1\|+c_1\|u\|+3$ and so that the $a$-length of every configuration is bounded by $|I_4(w,H_1)|$. Using Lemma \ref{multiply one letter}, we can then find a computation $I_4(w,H_1)\cdot\theta(43)\to\dots\to W_{l+1}$ of step history $(3)$ so that the computation's length and the configurations' $a$-lengths are bounded by $|W_{l+1}|_a$. Concatenating these computations then gives a computation $A_5\to\dots\to W_{l+1}$ with the desirable bounds, so that we can assume that $W_l\equiv I_4(w,H_1)$.

Let $W_{l+1}\to\dots\to W_x$ be the maximal subcomputation with step history $(3)$. The restriction of this subcomputation to any historical $R_iP_{i+1}$-sector, $V_{l+1}\to\dots\to V_x$, then satisfies the hypotheses of Lemma \ref{multiply one letter}. But then the history of this computation is at most $\|V_{l+1}\|+\|V_x\|=\|H_1\|+\|V_x\|$ and $\|V_j\|\leq\max(\|H_1\|,\|V_x\|)$ for all $l+1\leq j\leq x$. So, since the projection of each configuration to the input sector is $u^n$, the computation's length is at most $|W_x|_a$ and each $|W_j|_a$ is at most $(s+1)|W_x|_a$. Combining the bounds with the first part of the computation implies the statement if $x=t$.

Otherwise, in the first case, $W_{x+1}$ is tame, so that Lemma \ref{M_3 standard base} gives the appropriate bounds. For the second case, applications of Lemmas \ref{multiply one letter} and \ref{primitive computations} give the appropriate bounds.

%
%

\end{proof}


\subsection{Computations of $\textbf{M}_5$ with long histories} \

For any accepted configuration $W$ of $\textbf{M}_5$, we fix an accepting computation $\pazocal{C}(W)$ according to Lemma \ref{M_5 accepted configurations}.

\begin{lemma} \label{M_5 long history}

\textit{(Lemma 4.10 of [26])} Let $W_0$ be an accepted configuration, $\pazocal{C}:W_0\to\dots\to W_t$ be a reduced computation of $\textbf{M}_5$, and $H_0,H_t$ be the histories of $\pazocal{C}(W_0)$, $\pazocal{C}(W_t)$, respectively. Then for the parameters $c_4$ and $c_5$, either:

\begin{addmargin}[1em]{0em} 

$(a)$ $t\leq c_4\max(\|W_0\|,\|W_t\|)$ and $\|W_j\|\leq c_5\max(\|W_0\|,\|W_t\|)$ for every $j=0,\dots,t$ or 

$(b)$ $\|H_0\|+\|H_t\|\leq t/500$ and the sum of lengths of all maximal subcomputations of $\pazocal{C}$ with step histories $(12)(2)(23)$, $(32)(2)(21)$, $(34)(4)(45)$, and $(54)(4)(43)$ is at least $0.99t$.

\end{addmargin}

\end{lemma}

\begin{lemma} \label{M_5 long history has controlled}

\textit{(Lemma 4.13 of [26])} Let $W_0$ be an accepted configuration and $\pazocal{C}:W_0\to\dots\to W_t$ be a reduced computation of $\textbf{M}_5$ in the standard base with $t>c_4\max(\|W_0\|,\|W_t\|)$. Then the history of any subcomputation $\pazocal{D}:W_r\to\dots\to W_s$ of $\pazocal{C}$ (or the inverse of $\pazocal{D}$) of length at least $0.4t$ contains a subcomputation with controlled history.

\end{lemma}

%
%
%
%
%

\begin{lemma} \label{M_5 special computations}

Let $\pazocal{C}:W_0\to\dots\to W_t$ be a nonempty reduced computation of $\textbf{M}_5$ such that:

\begin{addmargin}[1em]{0em}

(1) $W_0\equiv A_5$ or $W_0\equiv I_5(u^n)$ for some $u^n\in\pazocal{L}$ and 

(2) $W_t\equiv A_5$ or $W_t\equiv I_5(v^n)$ for some $v^n\in\pazocal{L}$. 

\end{addmargin}

Then the sum of the lengths of all subcomputations of $\pazocal{C}$ with step histories $(12)(2)(23)$, $(32)(2)(21)$, $(34)(4)(45)$, and $(54)(4)(43)$ is at least $0.99t$.

\end{lemma}

\begin{proof}

Lemmas \ref{first step history} and \ref{second step history} imply that $\pazocal{C}$ can be factorized as $\pazocal{D}_1\pazocal{C}_2\dots\pazocal{D}_{l-1}\pazocal{C}_{l-1}\pazocal{D}_l$ where $l\geq1$, every $\pazocal{C}_i$ has one of the four step histories from Lemma \ref{(A) and (B)}(A), and every $\pazocal{D}_i$ is a one-step subcomputation with step history (1), (3), or (5). Further, $\pazocal{D}_1$ (respectively $\pazocal{D}_l$) have step history (1) or (5), which is determined by the configuration $W_0$ (respectively $W_t$). 

Suppose $l=1$. Then $\pazocal{C}$ is a one-step computation with step history $(1)$ or $(5)$. In the former case, both $W_0$ and $W_t$ have empty $R_0P_1$-sectors, so that Lemma \ref{multiply one letter}$(a)$ implies that the computation is empty. In the latter case, both have empty historical $R_iP_{i+1}$-sectors, so that again the computation must be empty.

So, we take $l\geq2$. Let $H(i)$ and $K(i)$ be the histories of $\pazocal{C}_i$ and $\pazocal{D}_i$, respectively.

Set $\pazocal{D}_i:W_x\to\dots\to W_y$ and let $V_x\to\dots\to V_y$ be the restriction of $\pazocal{D}_i$ to a sector in which it inserts letters. As in the proof of Lemma \ref{M_5 long history}, $\|K(i)\|\leq|V_x|_a+|V_y|_a$, while $\|H(i)\|\geq2k|V_x|_a$ and $\|H(i+1)\|\geq2k|V_y|_a$. As for $\pazocal{D}_1$ (respectively $\pazocal{D}_l$), we can choose the sector so that the initial (respectively final) configuration satisfies $|V_x|_a=0$ (respectively $|V_y|_a=0$). It follows that $\sum_i\|H(i)\|\geq500\sum_i\|K(i)\|$. This implies that $$\sum_i\|H(i)\|\geq0.99\bigg(\sum_i\|H(i)\|+\sum_i\|K(i)\|\bigg)=0.99t$$

\end{proof}

\begin{lemma} \label{M_5 factor}

Let $\pazocal{C}:W_0\to\dots\to W_t$ be a reduced computation of $\textbf{M}_5$ such that $W_0\equiv A_5$ or $W_0\equiv I_5(u^n)$ for some $u^n\in\pazocal{L}$. Let $H$ be the history of $\pazocal{C}$. Then there exists a factorization $H\equiv H_1H_2$ such that:

\begin{addmargin}[1em]{0em}

(1) for $\pazocal{C}_1$ the subcomputation of history $H_1$, the sum of the lengths of all subcomputations of $\pazocal{C}_1$ with step histories $(12)(2)(23)$, $(32)(2)(21)$, $(34)(4)(45)$, and $(54)(4)(43)$ is at least $0.99\|H_1\|$

(2) $\|H_2\|\leq c_4\|W_t\|$ or $\|H_2\|\leq\|H_1\|/200$

\end{addmargin}

\end{lemma}

\begin{proof}

If $t>c_4\max(\|W_0\|,\|W_t\|)$, then Lemma \ref{M_5 long history} implies the statement for $H_1=H$ and $H_2$ empty. So, we assume $t\leq c_4\max(\|W_0\|,\|W_t\|)$.

If $W_0\equiv A_5$, then $\|W_t\|\geq\|W_0\|$, so that $t\leq c_4\|W_t\|$, implying the statement for $H_2=H$ and $H_1$ empty. So, we assume that $W_0\equiv I_5(u^n)$ for some $u^n\in\pazocal{L}$.

Factor $\pazocal{C}$ as $\pazocal{D}_1\pazocal{C}_2\dots\pazocal{D}_{l-1}\pazocal{C}_{l-1}\pazocal{E}$, where $\pazocal{D}_i$ is a one-step computation with step history $(1)$, $(3)$, or $(5)$, $\pazocal{C}_i$ is a maximal subcomputation with step history from Lemma \ref{(A) and (B)}(A), and the step history of $\pazocal{E}$ is of the form (B) in this lemma.

Set $\pazocal{C}_1'=\pazocal{D}_1\pazocal{C}_2\dots\pazocal{D}_{l-1}\pazocal{C}_{l-1}$ and $H_1$ as its history. As in previous lemmas, the sum of the lengths of all subcomputations of $\pazocal{C}_1'$ with step histories of the four forms of (A) is at least $0.99\|H_1\|$. So, it suffices to prove that for $H_2$ the history of $\pazocal{E}$, statement (2) holds. 

Assume $\|H_2\|>c_4\|W_t\|$. By Lemmas \ref{first step history} and \ref{second step history}, there is a maximal prefix $H'$ of $H$ for which the step history of the corresponding computation is a prefix of $(1)(12)(2)(23)$. But Lemma \ref{primitive computations} implies that any such computation $I_5(w)\to\dots\to V$ has length bounded by $c_1\|V\|$. This implies that $H$ has a prefix so that the computation corresponding to this has step history $(1)(12)(2)(23)$; in particular, $H_1$ is nonempty.

Set $s=\|H_1\|$, so that $\pazocal{E}:W_s\to\dots\to W_t$. By Lemma \ref{M_5 tame}, $\|H_2\|\leq c_2^2(\|W_s\|+\|W_t\|)$, so that $\|W_s\|>\frac{c_4-c_2^2}{c_2^2}\|W_t\|>c_2\|W_t\|$ by choice of parameters.

\underline{Case 1.} The step history of $\pazocal{C}_1'$ ends with $(32)(2)(21)$.

By Lemma \ref{second step history}, $\pazocal{E}$ has step history $(1)$. But $W_s$ has empty $Q_0R_0$-sector, so that Lemma \ref{multiply one letter} implies that $\|W_t\|\geq\|W_s\|$, contradicting the assumption that $\|W_s\|>c_2\|W_t\|$.

\underline{Case 2.} The step history of $\pazocal{C}_1'$ ends with $(12)(2)(23)$.

By Lemmas \ref{first step history} and \ref{(A) and (B)}, the step history of $\pazocal{E}$ is then a prefix of $(3)(34)(4)$. Let $W_s\to\dots\to W_x$ be the maximal subcomputation with step history $(3)$.

If $x<t$, then $W_{x+1}\to\dots\to W_t$ is a computation with step history $(4)$ and $W_{x+1}$ is tame. So, Lemma \ref{M_4 bounds} implies that $\|W_{x+1}\|\leq c_1\|W_t\|$, meaning $\|W_s\|>\frac{c_2}{c_1}\|W_x\|>c_1\|W_x\|$ by taking $c_2>c_1^2$.

But $W_s\to\dots\to W_x$ is a computation of step history $(3)$ and $W_s$ has empty historical $R_iP_{i+1}$-sectors, so that $\|W_s\|\leq\|W_x\|$ giving a contradiction.

\underline{Case 3.} The step history of $\pazocal{C}_1'$ ends with $(34)(4)(45)$.

By Lemmas \ref{first step history} and \ref{(A) and (B)}, the step history of $\pazocal{E}$ is then a prefix of $(5)(54)(4)$. Let $W_s\to\dots\to W_x$ be the maximal subcomputation with step history $(5)$. As in the previous case, we then have $\|W_s\|>c_1\|W_x\|$. In particular, letting $V_s\to\dots\to V_x$ be the restriction of the subcomputation to a historical $R_iP_{i+1}$-sector, $\|V_s\|>c_1\|V_x\|$.

By Lemma \ref{multiply one letter}, $x-s\leq\|V_s\|+\|V_x\|<2\|V_s\|$. Meanwhile, if $t>x$, then the subcomputation $W_{x+1}\to\dots\to W_t$ is a computation of $\overline{\textbf{M}}_4$ with $W_{x+1}\equiv A_4(H_1)$ for some $H_1\in F(\Phi^+)$. Then, Lemma \ref{M_3 accepting computations} implies that $t-x\leq c_1|U_t|_a$ where $U_t$ is any admissible subword of $W_t$ with base $Q_iR_iP_{i+1}Q_{i+1}$ where $R_iP_{i+1}$ is historical. Note that $|W_s|_a=\kappa\|V_s\|$ where $\kappa$ is the number of historical $R_iP_{i+1}$-sectors and their mirror images in the standard base, while $|W_t|_a\geq\kappa|U_t|_a$. So, $\|V_s\|>c_2|U_t|_a$, implying $t-x<\|V_s\|$. Combining the inequalities above, we get $\|H_2\|<3\|V_s\|$. 

Now consider the subcomputation $W_l\to\dots\to W_s$ with step history $(34)(4)(45)$. Then each of the $4k+1$ subcomputations of $W_{l+1}\to\dots\to W_{s-1}$ between transition rules have length at least $\|V_s\|$, so that taking $k$ sufficiently large yields $\|H_1\|\geq(4k+1)\|V_s\|>200\|H_2\|$.

\underline{Case 4.} The step history of $\pazocal{C}_1'$ ends with $(54)(4)(43)$.

Lemmas \ref{first step history} and \ref{(A) and (B)} imply that the step history of $\pazocal{E}$ is a prefix of either $(3)(34)(4)$ or $(3)(32)(2)$. If it is a prefix of $(3)(34)(4)$, then an alagous argument to Case 3 applies. So, assume the step history is $(3)(32)(2)$.

Let $W_s\to\dots\to W_x$ be the maximal subcomputation with step history $(3)$. Then, we have $\|W_s\|>c_1\|W_x\|$, so that letting $V_s\to\dots\to V_x$ be the restriction of this subcomputation to a historical $R_iP_{i+1}$-sector, $\|V_s\|>c_1\|V_x\|$. So, $x-s\leq 2\|V_s\|$. 

If $t>x$, then $W_{x+1}\to\dots\to W_t$ is a computation of (two parallel copies of) $\textbf{LR}_k$ on the subword $R_0P_1Q_1$ (and its mirror image). Let $U_{x+1}\to\dots\to U_t$ be the restriction of this subcomputation to the base $R_0P_1Q_1$.

Let $w$ be the projection of $W_s$ to the input sector. Then $|U_{x+1}|_a=\|w\|$. By Lemma \ref{primitive computations}, there exists $y$ with $x+1\leq y\leq t$ such that $|U_{x+1}|_a=\dots=|U_y|_a$ and $|U_y|_a<|U_{y+1}|_a<\dots<|U_t|_a$ (if $t>y$). Further, the same lemma implies that $y-x-1\leq k(2\|w\|+1)$ and $t-y\leq|U_t|_a$.

Combining, this then implies that $$\|H_2\|\leq2\|V_s\|+k(2\|w\|+1)+|U_t|_a+1\leq2\|V_s\|+(2k+1)|U_t|_a+(k+1)$$
Note that $V_s$ is the copy of the history of a computation of $\textbf{M}_1$ accepting $w$. So, since each rule changes the size of any sector by at most two, $2\|V_s\|\geq\|w\|$, i.e $|W_s|_a\leq2(N+1)\|V_s\|$. So, since $\|W_s\|>c_2\|W_t\|$, we have $\|V_s\|\geq c_1|U_t|_a$ by taking $c_2>>c_1>>N$. So, taking $c_1>>k$, it then follows that $\|H_2\|\leq3\|V_s\|+(k+1)$.

But the subcomputation $W_l\to\dots\to W_s$ with step history $(54)(4)(43)$ must have length at least $(2k+1)\|V_s\|$. But since $\|W_s\|>c_2\|W_t\|$, we have $N(\|V_s\|+1)+2\|w\|>c_2(2|U_t|_a+N)$. Since $|U_t|_a\geq\|w\|$, it then follows that $\|V_s\|>c_2-1>k(k+1)$ since $c_2>>k$. Hence, $\|H_1\|\geq2k\|V_s\|+k(k+1)>200\|H_2\|$ by taking $k$ sufficiently large.

\end{proof}

%
%

\begin{lemma} \label{M_5 one-step}

\textit{(Lemma 4.14 of [26])} Let $W_0$ be an accepted configuration and $\pazocal{C}:W_0\to\dots\to W_t$ be a reduced computation with step history of length 1. Suppose $|W_j|_a>3|W_0|_a$ for some $1\leq j\leq t$. Then there is a sector $QQ'$ such that a state letter from $Q$ or from $Q'$ inserts an $a$-letter increasing the length of the sector for each rule of the subcomputation $W_j\to\dots\to W_t$.

\end{lemma}

\medskip


\subsection{The machines $\textbf{M}_{6,1}$ and $\textbf{M}_{6,2}$} \

Similar to the construction in [19], the next auxiliary machine, $\textbf{M}_{6,1}$, is the \textit{cyclic} machine that functions as the composition of the machine $\textbf{M}_5$ with itself a large number of times, specifically the parameter $L$.

Letting $\{t(i)\}B_4(i)$ be a copy of the standard base of $\textbf{M}_5$ for each $1\leq i\leq L$, the standard base of $\textbf{M}_{6,1}$ is
$$\{t(1)\}B_4(1)\{t(2)\}B_4(2)\{t(3)\}\dots\{t(L)\}B_4(L)$$ 
However, we also assign a tape alphabet to the space after the final letter $R_s(L)$ of $B_4(L)$, which corresponds to the $R_s(L)\{t(1)\}$-sector. As such, it is possible for an admissible word of $\textbf{M}_{6,1}$ to have base $$\{t(1)\}B_4(1)\dots\{t(L)\}B_4(L)\{t(1)\}B_4(1)\{t(2)\}$$ i.e it essentially `wraps around' the standard base. This is the defining property of a `cyclic' machine.

The tape alphabet of any sector formed by a one-letter part of the standard base (including the $R_s(L)\{t(1)\}$-sector) is defined as empty in this construction. The tape alphabets of all other sectors arise from $\textbf{M}_5$ in the natural way. 

The rules of $\textbf{M}_{6,1}$ are in correspondence with those of $\textbf{M}_5$, with each rule operating in parallel on each of the copies of the standard base of $\textbf{M}_5$ in the same way as its corresponding rule.

The copies of the input sector (respectively the historical sectors) are taken as the input (respectively historical) sectors of $\textbf{M}_{6,1}$.

Naturally, there arise submachines $\textbf{M}_{6,1}(1),\dots,\textbf{M}_{6,1}(5)$ corresponding to the submachines of $\textbf{M}_5$. As such, the definition of step history and controlled history extend to computations of $\textbf{M}_{6,1}$.

The lemmas of Section 4.9 have natural analogues to computations in $\textbf{M}_{6,1}$. For example, letting $I_6(w)$ be the start configuration obtained by concatenating $L$ copies of $I_5(w)$, the following is the obvous analogue of Lemma \ref{M_5 language}:

\begin{lemma} \label{M_{6,1} language}

For $w\in F(\pazocal{A})$, $I_6(w)$ is accepted by $\textbf{M}_{6,1}$ if and only if $w\in\pazocal{L}$.

\end{lemma}

\smallskip

The machine $\textbf{M}_{6,2}$, unique to this construction, has hardware that is a copy of that of $\textbf{M}_{6,1}$. The rules of the machine are in correspondence with those of $\textbf{M}_{6,1}$ and operate analogously except for one modification: Each rule locks the first input sector, i.e the $Q_0(1)R_0(1)$-sector.

The definitions of $\textbf{M}_{6,1}$ extend in the obvious way to $\textbf{M}_{6,2}$, and the lemmas of Section 4.9 again have natural analogues in this machine. For example, letting $J_6(w)$ be the configuration obtained from $I_6(w)$ by erasing the copy of $w$ in the first input sector, the following is the natural analogue of Lemma \ref{M_5 language} (and Lemma \ref{M_{6,1} language}):

\begin{lemma} \label{M_{6,2} language}

For $w\in F(\pazocal{A})$, $J_6(w)$ is accepted by $\textbf{M}_{6,2}$ if and only if $w\in\pazocal{L}$.

\end{lemma}

\medskip


\section{The machine $\textbf{M}$}

\subsection{Definition of the machine} \

The final step of this construction is to combinine the machines $\textbf{M}_{6,1}$ and $\textbf{M}_{6,2}$ to create the cyclic machine $\textbf{M}$ that is sufficient for the proof of Theorem \ref{main theorem}.

Similar to $\textbf{M}_{6,1}$ and $\textbf{M}_{6,2}$, the standard base of $\textbf{M}$ is of the form $\{t(1)\}B_4(1)\dots\{t(L)\}B_4(L)$, with the same sectors considered input, historical, etc.  However, each of the parts making up $B_4(i)$ consists of more state letters than its counterparts in $\textbf{M}_{6,1}$ and $\textbf{M}_{6,2}$. 

To be precise, any part of the standard base that is not a one-letter part $\{t(j)\}$ consists of a copy of the corresponding part of the standard base of $\textbf{M}_{6,1}$, a (disjoint) copy of the corresponding part of the standard base of $\textbf{M}_{6,2}$, and two new letters, $q(s)$ and $q(a)$, which function as the part's start and end letters, respectively. The accept configuration of $\textbf{M}$, denoted $W_{ac}$, is taken to be the concatenation of the end letters of the parts of the standard base. So, taking $B_4(j,a)$ as the concatenation of the end letters of the parts of $B_4(j)$, we have:
$$W_{ac}\equiv t(1)B_4(1,a)t(2)\dots t(L)B_4(L,a)$$

The rules $\Theta$ of $\textbf{M}$ are partitioned into two disjoint sets, $\Theta_1$ and $\Theta_2$. The positive rules of each consist of two connecting rules and a set of `working' rules. Unlike in previous constructions, though, these two sets are not connected in order to force them to run sequentially, rather in order to force them to operate `one or the other'.

The rules of $\Theta_1^+$ are defined as follows:

\begin{addmargin}[1em]{0em}

$\bullet$ The transition rule $\theta(s)_1$ locks all sectors other than the input sectors and their mirror copies. It switches the state letters from the start state of $\textbf{M}$ to the copy of the start state of $\textbf{M}_{6,1}$.

$\bullet$ The `working' rules of $\Theta_1^+$ operate exactly as the positive rules of the machine $\textbf{M}_{6,1}$.

$\bullet$ The transition rule $\theta(a)_1$ locks all sectors and switches the state letters from the copies of the end state letters of $\textbf{M}_{6,1}$ to the end state letters of $\textbf{M}$.

\end{addmargin}

The rules of $\Theta_2^+$ are defined as follows:

\begin{addmargin}[1em]{0em}

$\bullet$ The transition rule $\theta(s)_2$ locks each of the sectors locked by $\theta(s)_1$, but also locks the $Q_0(1)R_0(1)$-sector. It switches the state letters from the start state of $\textbf{M}$ to the copy of the start state of $\textbf{M}_{6,2}$.

$\bullet$ The `working' rules of $\Theta_2^+$ operate exactly as the positive rules of the machine $\textbf{M}_{6,2}$.

$\bullet$ The transition rule $\theta(a)_2$ locks all sectors and switches the state letters from the copies of the end state letters of $\textbf{M}_{6,2}$ to the end state letters of $\textbf{M}$.

\end{addmargin}

By the definition of the rules, one might infer that the first input sector $Q_0(1)R_0(1)$ is of particular significance. Hence, we refer to it as the \textit{`special' input sector}.

For $w\in F(\pazocal{A})$, the natural copy of $I_6(w)$ (respectively $J_6(w)$) in the hardware of this machine is $\theta(s)_1^{-1}$-admissible (respectively $\theta(s)_2^{-1}$-admissible). We denote $I(w)$ (respectively $J(w)$) as the start input configuration satisfying $I(w)\equiv I_6(w)\cdot\theta(s)_1^{-1}$ (respectively $J(w)\equiv J_6(w)\cdot\theta(s)_2^{-1}$). Note that both $I(w)$ and $J(w)$ are $\theta(s)_1$-admissible, while $I(w)$ is not $\theta(s)_2$-admissible for nonempty $w$.

\medskip


\subsection{Standard computations of $\textbf{M}$} \

Next, we adapt the definition of the step history to computations of $\textbf{M}$. To this end, let the letters $(s)_i^{\pm1}$ and $(a)_i^{\pm1}$ represent the transition rules of $\Theta_i$ and add a subscript to each letter of the step history of a subcomputation acting as $\textbf{M}_{6,i}$.

So, an example of a step history of a reduced computation of $\textbf{M}$ is $(s)_1(1)_1(12)_1(2)_1$, while a general step history is some concatenation of the letters 
$$\{(s)_i^{\pm1},(a)_i^{\pm1},(1)_i,(2)_i,(3)_i,(4)_i,(5)_i,(12)_i,(23)_i,(34)_i,(45)_i,(21)_i,(32)_i,(43)_i,(54)_i; \ i=1,2\}$$
A reduced computation is called a \textit{one-machine computation} if every letter of its step history contains the same index; if this index is $i$, then it is called a \textit{one-machine computation in the $i$-th machine}. For example, a computation with step history $(s)_1(1)_1(12)_1(2)_1$ is a one-machine computation in the first machine, while a computation with step history $(1)_1(s)_1^{-1}(s)_2(1)_2$ is not a one-machine computation. A computation that is not one-machine is called \textit{multi-machine}.

As in Section 4.9, some subwords clearly cannot appear in the step history of a reduced computation, while other impossibilities are less obvious. As $\textbf{M}$ acts as parallel copies of $\textbf{M}_5$ throughout a one-machine computation, though, Lemmas \ref{first step history} and \ref{second step history} (with the same subscript added to each letter) apply. 

\begin{lemma} \label{turn}

Let $\pazocal{C}:W_0\to W_1\to W_2$ be the reduced computation with step history $((s)_1^{-1}(s)_2)^{\pm1}$ and base $(Q_0(1)R_0(1))^{\pm1}$, that is, the `special' input sector (or its inverse). Then $|W_i|_a=0$ for $0\leq i\leq 2$. 

\end{lemma}

\begin{lemma} \label{return to start}

If a one-machine computation $\pazocal{C}:W_0\to\dots\to W_t$ of $\textbf{M}$ in the standard base has step history of the from $(s)_ih_i(s)_i^{-1}$, then $W_0$ is either $I(u^n)$ or $J(u^n)$ for some $u^n\in\pazocal{L}$.

\end{lemma}

\begin{proof}

Note that $W_0$ is $\theta(s)_i$-admissible, so that it is a start configuration with all sectors empty except perhaps for the input sectors and their mirror copies. Assume at least one of these sectors is not empty.

As $h_i$ cannot be empty, it must have prefix $(1)_i(12)_i$ by Lemma \ref{multiply one letter}$(a)$. 

If $i=1$, then Lemma \ref{multiply one letter}$(a)$ further implies that $W_0$ must have a copy of some word $w\in F(\pazocal{A})$ written in each input sector and a copy of its inverse written in each mirror copy of the input sector. In particular, this means $W_0\equiv I(w)$.

If $i=2$, then the analogous argument implies the same except for an empty `special' input sector, i.e $W_0\equiv J(w)$.

Lemmas \ref{first step history} and \ref{second step history} then imply that $h_i$ has prefix $(1)_i(12)_i(2)_i(23)_i(3)_i(34)_i(4)_i(45)_i$. As in the proof Lemma \ref{M_5 language}(2), this implies that $w\in\pazocal{L}$.

\end{proof}

\begin{lemma} \label{M language}

A start configuration $W$ is accepted by the machine $\textbf{M}$ if and only if $W\equiv I(u^n)$ or $W\equiv J(u^n)$ for some $u^n\in\pazocal{L}$.

\end{lemma}

\begin{proof}

For any $u^n\in\pazocal{L}$, the definition of the rules and Lemmas \ref{M_{6,1} language} and \ref{M_{6,2} language} imply that $I(u^n)$ and $J(u^n)$ are accepted by one-machine computations of the first and second machine, respectively.

Let $\pazocal{C}:W\equiv W_0\to\dots\to W_t\equiv W_{ac}$ be an accepting computation. Consider the maximal one-machine computation, $\pazocal{C}_1:W_0\to\dots\to W_r$, that serves as a prefix of $\pazocal{C}$. 

Then the step history of $\pazocal{C}_1$ is of the form $(s)_ih_i(s)_i^{-1}$ or $(s)_ih_i(a)_i$, where the maximal subcomputation with step history $h_i$ is a computation of $\textbf{M}_{6,i}$. 

In the first case, the statement follows from Lemma \ref{return to start}. 

In the second, note that $W_0$ is $\theta(s)_i$-admissible, so that all its sectors are empty except perhaps for the input sectors and their mirror images. As in the proof of Lemma \ref{return to start}, applications of Lemma \ref{multiply one letter}$(a)$ imply that there exists $w\in F(\pazocal{A})$ such that $W_0\equiv I(w)$ if $i=1$ and $W_0\equiv J(w)$ if $i=2$. 

That $w\in\pazocal{L}$ follows from Lemmas \ref{M_{6,1} language} and \ref{M_{6,2} language} applied to the maximal subcomputation of step history $h_i$.

\end{proof}

The history $H$ of a reduced computation $\pazocal{C}$ of $\textbf{M}$ is called \textit{controlled} if $\pazocal{C}$ is a one-machine computation and $H$ corresponds to a controlled computation of $\textbf{M}_5$. As such, the next lemma follows immediately from Lemma \ref{M_5 controlled}.

\begin{lemma} \label{M controlled}

Let $\pazocal{C}:W_0\to\dots\to W_t$ be a reduced computation of $\textbf{M}$ with controlled history $H$. Then the base of the computation is a reduced word and all configurations are uniquely defined by the history $H$ and the base of $\pazocal{C}$.\newline Moreover, if $\pazocal{C}$ is a computation in the standard base, then $|W_0|_a=\dots=|W_t|_a$ and $\|H\|=2s+3$ (respectively $\|H\|=s+2$), where $s$ is the $a$-length of each historical $R_i(j)P_{i+1}(j)$-sector (respectively of each $R_0(j)P_1(j)$-sector) of $W_0$ if $H$ is of the form $(a)$ (respectively $(b)$) in the definition of controlled history (see Section 4.9).

\end{lemma}

\smallskip


\subsection{Projected subwords} \

For a configuration $W$ and $1\leq i\leq L$, define $W(i)$ as the admissible subword of $W$ with base $\{t(i)\}B_4(i)$. So, for any configuration $W$, $W\equiv W(1)\dots W(L)$. It is worth noting, then, that if a rule $\theta$ is applicable to some configuration $W$, then $\theta$ operates on each $W(j)$ identically for each $j\geq2$.

Particularly, define the admissible words $A(i)\equiv W_{ac}(i)$, $I(w,i)\equiv (I(w))(i)$, and $J(w,i)\equiv (J(w))(i)$ for all $w\in F(\pazocal{A})$.

For an admissible word $V$ with base a subword of $\{t(i)\}B_4(i)$ for some $1\leq i\leq L$, a \textit{coordinate shift} of $V$ is an admissible word $V'$ which is the copy of $V$ with base a subword of $\{t(j)\}B_4(j)$ for some $1\leq j\leq L$ obtained simply changing all $i$'s to $j$'s. For example, if $W$ is an accepted configuration, then for $i,j\geq2$, $W(i)$ and $W(j)$ are coordinate shifts of one another.

Further, for a configuration $W$ and $1\leq i\leq L$, define $W(i,m)$ as the admissible subword of $W(i)$ whose base is the copy of $(B_3')^{-1}$. As above, similarly define $A(i,m)$, $I(w,i,m)$, and $J(w,i,m)$ for notational convenience.

\begin{lemma} \label{extending one-machine}

Let $H$ be the history of a one-machine computation of the $j$-th machine $\pazocal{C}:V_0\to\dots\to V_t$ with base $\{t(i)\}B_4(i)$ for some $i\in\{1,\dots,L\}$. Then there exists a one-machine computation $W_0\to\dots\to W_t$ of the standard base with history $H$ such that $W_l(i)\equiv V_l$ for all $0\leq l\leq t$.

\end{lemma}

\begin{proof}

For each $l\in\{0,\dots,t\}$, construct from $V_l$ the admissible words $V_l(x)$ with base $\{t(x)\}B_4(x)$ for $1\leq x\leq L$ as follows:

\begin{addmargin}[1em]{0em}

$\bullet$ \underline{Case 1:} Suppose $j=1$. 

Then for all $1\leq x\leq L$, define $V_l(x)$ as the coordinate shift of $V_l$ with base $\{t(x)\}B_4(x)$.

$\bullet$ \underline{Case 2:} Suppose $j=2$ and $i\geq2$. 

Then define $V_l(x)$ for $x\geq2$ in the same way as in Case 1. 

However, in this case, set $V_l(1)$ as the admissible word resulting from emptying the `special' input sector of the coordinate shift of $V_l$ with base $\{t(1)\}B_4(1)$.

$\bullet$ \underline{Case 3:} Suppose $j=2$ and $i=1$. 

Then, define $V_l(m)$ as the admissible subword of $V_l$ whose base is the copy of $(B_3')^{-1}$ in $B_4(1)$. 

Next, define $V_l(x)(m)$ as the appropriate coordinate shift of $V_l$ for each $x\geq2$. 

Letting $V_l(x)(m')$ be the mirror copy of $V_l(x)(m)$, define $V_l(x)\equiv t(x)V_l(x)(m')V_l(x)(m)$.

\end{addmargin}

Now define $W_l\equiv V_l(1)\dots V_l(L)$ for each $0\leq l\leq t$. Clearly, $W_l(i)\equiv V_l$ for all $l$.

Letting $\theta_l$ be the letter of $H$ corresponding to the transition $V_{l-1}\to V_l$, it is easy to check that in each case, $W_{l-1}\cdot\theta_l\equiv W_l$.

\end{proof}

\begin{lemma} \label{extending one-machine m}

Let $H$ be the history of a one-machine computation of the $j$-th machine $\pazocal{C}:V_0\to\dots\to V_t$ with base the copy of $(B_3')^{-1}$ in $B_4(i)$ for some $i\in\{1,\dots,L\}$. Then there exists a computation $W_0\to\dots\to W_t$ of the standard base with history $H$ such that $W_l(i,m)\equiv V_l$ for all $0\leq l\leq t$.

\end{lemma}

\begin{proof}

For each $l\in\{0,\dots,t\}$, construct from $V_l$ the admissible words $V_l(x)$ with base $B_4(x)$ for all $1\leq x\leq L$ as follows:

\begin{addmargin}[1em]{0em}

$\bullet$ \underline{Case 1:} Suppose $j=1$.

Let $V_l(m')$ be the mirror copy of $V_l$, so that it is an admissible word whose base is the copy of $B_3$ in $B_4(i)$.

Then define $V_l(i)$ as the admissible word with base $\{t(i)\}B_4(i)$ formed by adjoining $t(i)V_l(m')$ to $V_l$. 

Finally, define $V_l(x)$ as the appropriate coordinate shift of $V_l(i)$ for each $1\leq x\leq L$.

$\bullet$ \underline{Case 2:} Suppose $j=2$ and $i\geq2$.

Define $V_l(x)$ in the same way as in Case 1 for $2\leq x\leq L$. 

However, in this case, define $V_l(1)$ as the result of emptying the `special' input sector of the coordinate shift of $V_l(i)$.

$\bullet$ \underline{Case 3:} Suppose $j=2$ and $i=1$. 

For each $x$, set $V_l(x)(m)$ as the corresponding coordinate shift of $V_l$. 

Then, for $x\geq2$, let $V_l(x)(m')$ be the mirror copy of $V_l(x)(m)$, while setting $V_l(1)(m')$ as the result of emptying the `special' input sector of the mirror copy of $V_l$. 

Then, set $V_l(x)=t(x)V_l(x)(m')V_l(x)(m)$ for each $1\leq x\leq L$.

\end{addmargin}

Now define $W_l\equiv t(1)V_l(1)t(2)\dots t(L)V_l(L)$ for each $0\leq l\leq t$. Clearly, $W_l(i,m)\equiv V_l$ for each $l$.

Letting $\theta_l$ be the letter of $H$ corresponding to the transition $V_{l-1}\to V_l$, it is easy to check that in each case, $W_{l-1}\cdot\theta_l\equiv W_l$.

\end{proof}

\begin{lemma} \label{input sectors locked}

Let $\pazocal{C}:A(i)\to\dots\to A(i)$ be a one-machine computation of the $j$-th machine. Then the step history of $\pazocal{C}$ has no occurrence of $(1)_j$. In particular, for $H$ the history of $\pazocal{C}$, every rule of $H$ locks the `special' input sector, $W_{ac}$ is $H$-admissible, and $W_{ac}\cdot H\equiv W_{ac}$.

\end{lemma}

\begin{proof}

Lemma \ref{second step history} implies that the step history of $\pazocal{C}$ has no occurrence of $(21)_j(1)_j(12)_j$. So, the definition of a one-machine computation implies that there is no occurrence of the letter $(1)_j$. As a consquence, the `special' input sector (and in fact all input sectors) are locked by every rule of $\pazocal{C}$.

Lemma \ref{extending one-machine} implies the rest of the statement, as one can check that the configuration arising from $A(i)$ in all three cases is $W_{ac}$.

\end{proof}

\begin{lemma} \label{input sectors locked m}

Let $\pazocal{C}:A(i,m)\to\dots\to A(i,m)$ be a one-machine computation of the $j$-th machine. Then the step history of $\pazocal{C}$ has no occurrence of $(1)_j$. In particular, for $H$ the history of $\pazocal{C}$, every rule of $H$ locks the `special' input sector, $W_{ac}$ is $H$-admissible, and $W_{ac}\cdot H\equiv W_{ac}$.

\end{lemma}

\begin{proof}

Using Lemma \ref{extending one-machine m}, this follows from a very similar proof as the one presented for Lemma \ref{input sectors locked}.

\end{proof}

\begin{lemma} \label{subword return to start}

Let $H\equiv(s)_jh_j(s)_j^{-1}$ be the history of a one-machine computation $\pazocal{C}$ in the base $\{t(i)\}B_4(i)$ for some $i\in\{1,\dots,L\}$. Then the initial configuration of $\pazocal{C}$ is $I(u^n,i)$ or $J(u^n,i)$ for some $u^n\in\pazocal{L}$. 

\end{lemma}

\begin{proof}

Denote the computation $\pazocal{C}$ as $V_0\to\dots\to V_t$. 

Then, applying Lemma \ref{extending one-machine}, there exist configurations $W_0,\dots,W_t$ such that $W_l(i)\equiv V_l$ and $W_0\cdot H\equiv W_t$.

Applying Lemma \ref{return to start} then implies that $W_0\equiv I(u^n)$ or $W_0\equiv J(u^n)$ for some $u^n\in\pazocal{L}$, whence $V_0\equiv W_0(i)\equiv I(u^n,i)$ or $J(u^n,i)$.

\end{proof}

Again, we form the `mirror' of Lemma \ref{subword return to start}:

\begin{lemma} \label{subword return to start m}

Let $H\equiv (s)_jh_j(s)_j^{-1}$ be the history of a one-machine computation $\pazocal{C}$ whose base is the copy of $(B_3')^{-1}$ in $B_4(i)$ for some $i\in\{1,\dots,L\}$. Then the initial configuration of $\pazocal{C}$ is $I(u^n,i,m)$ or $J(u^n,i,m)$ for some $u^n\in\pazocal{L}$.

\end{lemma}

\begin{proof}

Using Lemma \ref{extending one-machine m}, this follows by a similar proof as the one presented to Lemma \ref{subword return to start}.

\end{proof}

Using Lemma \ref{M language}, a similar proof as the one presented for Lemma \ref{subword return to start} then implies the following:

\begin{lemma} \label{subword M language}

If $W_0$ is an admissible subword of a start configuration with base $\{t(i)\}B_4(i)$ for some $i\in\{1,\dots,L\}$, then there exists a one-machine computation $W_0\to\dots\to A(i)$ if and only if $W_0\equiv I(u^n,i)$ or $W_0\equiv J(u^n,i)$ for some $u^n\in\pazocal{L}$.

\end{lemma}

Similarly, the following `mirror' lemma follows from Lemma \ref{M language} via a proof similar to that presented as the proof of Lemma \ref{subword return to start m}:

\begin{lemma} \label{subword M language m}

If $W_0$ is a subword of a start configuration whose base is the copy of $(B_3')^{-1}$ in $B_4(i)$ for some $i\in\{1,\dots,L\}$, then there exists a one-machine computation $W_0\to\dots\to A(i,m)$ if and only if $W_0\equiv I(u^n,i,m)$ or $W_0\equiv J(u^n,i,m)$ for some $u^n\in\pazocal{L}$. 

\end{lemma}

\begin{lemma} \label{starts with I}

For $i\in\{1,\dots,L\}$, suppose $\pazocal{C}:I(w,i)\equiv W_0\to\dots\to W_t$ is a maximal one-machine subcomputation of a multi-machine computation $\pazocal{D}$ such that $\pazocal{C}$ serves as a prefix of $\pazocal{D}$. Let $H$ be the history of $\pazocal{C}$. Then $w\in\pazocal{L}$ and either:

\begin{addmargin}[1em]{0em}

(1) $I(w)$ is $H$-admissible with $I(w)\cdot H\equiv W_{ac}$ or $I(w)\cdot H\equiv I(w')$ for some $w'\in\pazocal{L}$, or

(2) $J(w)$ is $H$-admissible with $J(w)\cdot H\equiv W_{ac}$ or $J(w)\cdot H\equiv J(w')$ for some $w'\in\pazocal{L}$.

\end{addmargin}

\end{lemma}

\begin{proof}

By the definition of the machine, the step history of $\pazocal{C}$ is either of the form $(s)_jh_j(a)_j$ or $(s)_jh_j(s)_j^{-1}$ for $j\in\{1,2\}$. By Lemmas \ref{subword return to start} and \ref{subword M language}, it follows that $w\in\pazocal{L}$.

If $\pazocal{C}$ is a one-machine computation of the first machine, then the construction described in the proof of Lemma \ref{extending one-machine} implies that $I(w)$ is $H$-admissible; if it is a one-machine computation of the second machine, then the same shows that $J(w)$ is $H$-admissible.

Clearly, if the step history is of the form $(s)_1h_1(a)_1$ (respectively $(s)_2h_2(a)_2$), then $I(w)\cdot H\equiv W_{ac}$ (respectively $J(w)\cdot H\equiv W_{ac}$). Otherwise, $W_t$ must be $\theta(s)_j$-admissible for some $j$, and so must be $I(w',i)$ or $J(w',i)$ for some $w'\in\pazocal{L}$. Applying the proof of Lemma \ref{subword return to start} to the inverse computation shows that $w'\in\pazocal{L}$, while the construction described in the proof of Lemma \ref{extending one-machine} shows that $I(w)\cdot H\equiv I(w')$ (respectively $J(w)\cdot H\equiv J(w')$).

\end{proof}

\begin{lemma} \label{starts with J}

For $i\in\{1,\dots,L\}$, suppose $\pazocal{C}:J(w,i)\equiv W_0\to\dots\to W_t$ is a maximal one-machine subcomputation of a multi-machine computation $\pazocal{D}$ such that $\pazocal{C}$ serves as a prefix of $\pazocal{D}$. Let $H$ be the history of $\pazocal{C}$. Then $w\in\pazocal{L}$ and either

\begin{addmargin}[1em]{0em}

(1) $I(w)$ is $H$-admissible with $I(w)\cdot H\equiv W_{ac}$ or $I(w)\cdot H\equiv I(w')$ for some $w'\in\pazocal{L}$

(2) $J(w)$ is $H$-admissible with $J(w)\cdot H\equiv W_{ac}$ or $J(w)\cdot H\equiv J(w')$ for some $w'\in\pazocal{L}$

\end{addmargin}

\end{lemma}

\begin{proof}

If $i\geq2$ or $w=1$, then $J(w,i)\equiv I(w,i)$, so that the statement follows from Lemma \ref{starts with I}.

So, assume $w\neq1$ and $i=1$. Then Lemma \ref{subword return to start} implies that $w\in\pazocal{L}$ 

Suppose $\pazocal{C}$ is a computation of the first machine. The lack of symmetry in $J(w,1)$ (the `special' input sector is empty while its mirror image contains a copy of $w^{-1}$) implies that the step history of $\pazocal{C}$ cannot have prefix $(s)_1(1)_1(12)_1$. But then it must have step history $(s)_1(1)_1(s)_1^{-1}$. The application of Lemma \ref{multiply one letter} to the restriction of $\pazocal{C}$ to the $R_0(1)P_1(1)$-sector then implies that the maximal subcomputation with steph history $(1)_1$ is empty, leading to a contradiction.

Then, as $\pazocal{C}$ is a computation of the second machine, the construction in the proof of Lemma \ref{extending one-machine} implies that $J(w)$ is $H$-admissible.

The rest of the statements follow from the same arguments as those presented in the proof of Lemma \ref{starts with I}.

\end{proof}

\begin{lemma} \label{starts with I,J m}

Suppose $\pazocal{C}:W_0\to\dots\to W_t$ is a maximal one-machine subcomputation of a multi-machine computation $\pazocal{D}$ such that $\pazocal{C}$ serves as a prefix of $\pazocal{D}$. Suppose further that $W_0\equiv I(w,i,m)$ or $W_0\equiv J(w,i,m)$ for some $i\in\{1,\dots,L\}$. Let $H$ be the history of $\pazocal{C}$. Then $w\in\pazocal{L}$ and either

\begin{addmargin}[1em]{0em}

(1) $I(w)$ is $H$-admissible with $I(w)\cdot H\equiv W_{ac}$ or $I(w)\cdot H\equiv I(w')$ for some $w'\in\pazocal{L}$

(2) $J(w)$ is $H$-admissible with $J(w)\cdot H\equiv W_{ac}$ or $J(w)\cdot H\equiv J(w')$ for some $w'\in\pazocal{L}$

\end{addmargin}

\end{lemma}

\begin{proof}

Using Lemma \ref{extending one-machine m}, this follows from the same arguments as the ones presented as the proofs of Lemmas \ref{starts with I} and \ref{starts with J}.

\end{proof}

The next lemma follows immediately from Lemmas \ref{input sectors locked} through \ref{starts with I,J m}.

\begin{lemma} \label{extending computations}

Suppose $\pazocal{C}:A(i,m)\to\dots\to A(i,m)$ is a reduced computation of $\textbf{M}$ with history $H$. Let $H\equiv H_1\dots H_l$ be the factorization so that each $H_j$ is a maximal one-machine subcomputation. Set $W_j\equiv A(i,m)\cdot(H_1\dots H_j)$ for each $j=0,\dots l$. Then either (1) $W_j\equiv A(i,m)$ or (2) there exists a $w_j\in\pazocal{L}$ such that $W_j\equiv I(w_j,i,m)$ or $J(w_j,i,m)$.

In case (1), set $W_j'(1)\equiv W_j'(2)\equiv W_{ac}$; in case (2), set $W_j'(1)\equiv I(w_j)$ and $W_j'(2)\equiv J(w_j)$.

Then for each $j$, there exists $z_j\in\{1,2\}$ and a reduced computation in the standard base $\pazocal{C}_j':W_{j-1}'(z_j)\to\dots\to W_j'(z_j)$ with history $H_j$. (Note that $z_j$ is such that $H_j$ is the history of a one-machine computation of the $z_j$-th machine).

\end{lemma}

In other words, Lemma \ref{extending computations} says that except for the insertion/removal of words of $\pazocal{L}$ from the `special' input sector, the computation $\pazocal{C}$ can be extended to a reduced computation $\pazocal{C}':W_{ac}\to\dots\to W_{ac}$ (though no such computation exists).

\begin{lemma} \label{projection admissible configuration not}

Let $W$ be an accepted configuration and $\theta\in\Theta$. Suppose $W$ is not $\theta$-admissible but $W(i)$ is $\theta$-admissible. Then $i\geq2$ and either:

\begin{addmargin}[1em]{0em}

(1) $\theta=\theta(s)_2$ and $W\equiv I(u^n)$ for some $u^n\in\pazocal{L}$.

(2) $\theta=\theta(12)_1$ and $W$ has $u^n$ written in the `special' input sector for some nonempty $u^n\in\pazocal{L}$.

\end{addmargin}

In particular, the configuration $V$ obtained by emptying the `special' input sector of $W$ is $\theta$-admissible.

\end{lemma}

\begin{proof}

Suppose $i=1$. Since each rule operates on the copies of $(B_3')^{-1}$ in parallel and $W$ is accepted, each subword of $W$ with such a base is a coordinate shift of the others. As $W(1)$ is $\theta$-admissible, it follows that each subword with base a copy of $(B_3')^{-1}$ is $\theta$-admissible. The symmetry of the rules' operation in the bases $B_4(j)$ for $j\geq2$ then implies that $W$ is $\theta$-admissible. So, $i\geq2$.

By a similar argument, it follows that $W(j)$ is $\theta$-admissible for each $j\geq2$. So, since $W$ is not $\theta$-admissible, $W(1)$ must not be $\theta$-admissible. Again, it follows that the subword of $W(1)$ whose base is a copy of $(B_3')^{-1}$ is $\theta$-admissible. So, by the definition of the rules, $\theta$ must lock the `special' input sector while that sector is not empty in $W$.

\underline{Case 1:} Suppose $\theta\in\Theta_2$.

As $W$ is accepted, it must be $\theta'$-admissible for some $\theta'\in\Theta$. But since each rule of $\Theta_2$ locks the `special' input sector, it follows that $\theta'\in\Theta_1$. Since $W(i)$ is admissible for rules from both $\Theta_1$ and $\Theta_2$, $W$ must be either a start or an end configuration. The only end configuration that is accepted, though, is $W_{ac}$, so that $W$ must be a start configuration. Clearly, then, $\theta=\theta(s)_2$.

Lemma \ref{M language} then implies that $W\equiv I(u^n)$ or $J(u^n)$ for some $u^n\in\pazocal{L}$; but since the `special' input sector must be nonempty, it follows that $W\equiv I(u^n)$.

\underline{Case 2:} Suppose $\theta\in\Theta_1$.

As the working rules of $\Theta_1$ operate in parallel as $\textbf{M}_5$, $W(1)$ cannot be a coordinate shift of $W(i)$. 

Let $\pazocal{C}':W_{ac}\equiv V_0\to\dots\to V_t\equiv W$ be the inverse of an accepting computation for $W$. Again noting that the rules of $\Theta_1$ operate in parallel as $\textbf{M}_5$, there exists a maximal subcomputation $\pazocal{D}':V_r\to\dots\to V_s$ such that $\pazocal{C}'=\pazocal{C}'_0\pazocal{D}'\pazocal{C}'_1$, $\pazocal{D}'$ is a one-machine computation of the second machine, and $\pazocal{C}'_1$ is a one-machine computation of the first machine.

Since $W(i)$ is $\theta$-admissible for $\theta\in\Theta_1$, $V_s$ must be an accepted start or end configuration. If it were an accepted end configuration, then $V_s\equiv W_{ac}$ and, since $\pazocal{C}'_1$ is a one-machine computation of the first machine, $W(1)$ is a coordinate shift of $W(i)$. So, $V_s$ must be an accepted start configuration, and so Lemma \ref{M language} implies that $V_s\equiv I(w)$ or $V_s\equiv J(w)$ for some $w\in\pazocal{L}$.

Again, if $V_s\equiv I(w)$ or $w$ is empty, then $W(1)$ is a coordinate shift of $W(i)$. So, $V_s\equiv J(w)$ for $w\in\pazocal{L}$ nonempty.

Clearly, the step history of $\pazocal{C}'_1$ must then have prefix $(s)_1(1)_1$. Letting $\pazocal{C}''_1$ be the restriction to the base $\{t(i)\}B_4(i)$, any subsequent letter of the step history must be $(12)_1$ by Lemma \ref{multiply one letter}. But since $V_s$ has empty `special' input sector, $V_t$ cannot be $\theta(12)_1$-admissible. So, the entire step history of $\pazocal{C}'_1$ is $(s)_1(1)_1$.

It follows from Lemma \ref{multiply one letter} that $W(i)$ is admissible only for rules of the step history $(1)_1$ or $(12)_1$. But $W(1)$ is admissible for all rules of step history $(1)_1$, so that $\theta=\theta(12)_1$.

Since $W(i)$ is $\theta$-admissible, Lemma \ref{multiply one letter} applied to $\pazocal{C}'_1$ implies that the `special' input sector of $W$ contains a copy of the word $w^{-1}\in\pazocal{L}$. Note that removing $w^{-1}$ from the `special' input sector of $W$ yields a configuration that is $\theta$-admissible.

\end{proof}

\begin{lemma} \label{accepted configuration a-length}

Let $W$ be an accepted configuration. Then $\frac{1}{2}|W(j)|_a\leq|W(1)|_a\leq\frac{3}{2}|W(j)|_a$ for all $2\leq j\leq L$.

\end{lemma}

\begin{proof}

Let $\pazocal{C}$ be an accepting computation for $W$ and $\bar{\pazocal{C}}$ the inverse computation.

For $i,j\geq2$, each rule operates on the bases $\{t(i)\}B_4(i)$ and $\{t(j)\}B_4(j)$ in the same way, so that $W(i)$ and $W(j)$ are coordinate shifts of one another.

Moreover, the mirror symmetry of the rules implies that $W(j)$ can be factored into two words, one with base $\{t(j)\}B_3(j)$ and the other the mirror copy, with identical $a$-lengths. Meanwhile, the mirror copy contained in $\{t(1)\}B_4(1)$ is operated on similarly, so that $W(1)$ can be factored into two words with at least one of them having the same $a$-length as each part of $W(j)$. As a result, $\frac{1}{2}|W(j)|_a\leq|W(1)|_a$.

If $\bar{\pazocal{C}}$ is a one-machine computation of the first machine, then $W(1)$ is also a copy of $W(j)$, so that $|W(1)|_a=|W(j)|_a$. If it is a one-machine computation of the second machine, then $|W(1)|_a\leq|W(j)|_a$. 

So, assume $\bar{\pazocal{C}}$ is a multi-machine computation and factor its history $H$ as $H\equiv H_1\dots H_l$ so that each $H_i$ is the history of a maximal one-machine subcomputation.

Inducting on $l$, one can assume that $W_{ac}\cdot H_1$ is not $W_{ac}$. So, we assume that there exists $u^n\in\pazocal{L}$ such that $W_{ac}\cdot H_1\equiv I(u^n)$ (i.e $H_1$ is the history of a one-machine computation of the first machine) or $W_{ac}\cdot H_1\equiv J(u^n)$ (i.e of the second machine).

If $W_{ac}\cdot H_1\equiv I(u^n)$, then since $\bar{\pazocal{C}}$ is not a one-machine computation, $I(u^n)$ must be $\theta(s)_2$-admissible. But this is only possible if $u^n$ is empty. Then either $W_{ac}\cdot H_1H_2\equiv J(v^n)$ for some $v^n\in\pazocal{L}$ or $W_{ac}\cdot H_1H_2\equiv W_{ac}$. In either case, each intermediate configuration $W'$ of the subcomputation with history $H_2$ satisfies $|W'(1)|_a\leq|W'(j)|_a$ for all $j\geq2$. If $W_{ac}\cdot H_1H_2\equiv W_{ac}$, then we can then disregard $H_1H_2$ and consider the shorter computation of history $H_3\dots H_l$; while if $W_{ac}\cdot H_1H_2\equiv J(v^n)$, then we can then replace $H_1H_2$ with the history $H_1'$ of a one-machine computation of the second machine satisfying $W_{ac}\cdot H_1'\equiv J(v^n)$. Either way, the number of steps has been reduced, so that we again induct on $l$.

As a result, we can assume that $W_{ac}\cdot H_1\equiv J(u^n)$. If $u^n$ is empty, then this is similar to the previous case. Otherwise, $H\equiv H_1H_2$ with $H_2$ the history of a computation of step history $(s)_1(1)_1$ (as in the proof of Lemma \ref{projection admissible configuration not}). As the computation with step history $(1)_1$ operates as $\textbf{LR}$ on each subword $Q_0(i)R_0(i)P_1(i)$ (and its mirror copy), it satisfies the hypotheses of Lemma \ref{multiply one letter}. Considering the computation outlined in Lemma \ref{multiply one letter}, note that a copy of $u^n$ would be written in each $R_0(i)P_1(i)$-sector and its mirror copy; but a copy of its inverse would be written in the `special' input sector (while the copy of $u^n$ is erased from all other input sectors and mirror copies of input sectors). The inequality then follows from Lemma \ref{multiply one letter}$(c)$.

\end{proof}

\begin{lemma} \label{M accepted configurations}

For every accepted configuration $W$ of $\textbf{M}$, there exists an accepting computation $W\equiv W_0\to\dots\to W_t\equiv W_{ac}$ of length at most $c_4\|W(i)\|$ and such that $|W_j(i)|_a\leq c_4|W(i)|_a$ for all $0\leq j\leq t$ and $1\leq i\leq L$. Moreover, this accepting computation consists of at most two maximal one-machine subcomputations.

\end{lemma}

\begin{proof}

Let $A(W)$ be the set of accepting computations of $W$. For $\pazocal{C}\in A(W)$, define $\ell(\pazocal{C})$ as the number of maximal one-machine subcomputations of $\pazocal{C}$. Then, define $\ell(W)=\min\{\ell(\pazocal{C})\mid \pazocal{C}\in A(W)\}$.

We then induct on $\ell(W)$. For the base case, suppose $\ell(W)=1$, so that $W$ is accepted by a one-machine computation of the $j$-th machine $\pazocal{C}:W\equiv W_0\to\dots\to W_t\equiv W_{ac}$. If the first rule of $\pazocal{C}$ is $\theta(a)_j^{-1}$ or $\theta(s)_j$ (i.e if $W$ is either a start or end configuration), then note that $|W_1|_a=|W|_a$, so that it suffices to prove the statement for $W_1$. So, assume $W$ is neither a start or end configuration. 

Set $\pazocal{C}'$ as the subcomputation $W_0\to\dots\to W_{t-1}$. By the definition of the machine and the assumption that $\pazocal{C}'$ is reduced, the step history of $\pazocal{C}'$ has no occurrence of the letters $(s)_j^{\pm1}$ or $(a)_j^{\pm1}$. Let $\pazocal{D}$ be the restriction of $\pazocal{C}'$ to the base $\{t(2)\}B_4(2)$. Then one can identify $\pazocal{D}$ with a reduced computation of $\textbf{M}_5$ in the standard base, which is further an accepting computation for its initial configuration $V$ (a copy of $W(2)$). Lemma \ref{M_5 accepted configurations} then supplies a reduced computation $\pazocal{E}$ of $\textbf{M}_5$ accepting $V$ so that the length of $\pazocal{E}$ has length at most $c_3\|V\|$ and the $a$-length of each configuration of $\pazocal{E}$ is at most $c_3|V|_a$. Identifying the rules of $\textbf{M}_5$ with the rules of $\Theta_j$ then yields a one-machine computation $\pazocal{E}'$ of $\textbf{M}$ with initial admissible word $W(2)$ and final admissible word a copy of $A_5$ (specifically $W_{ac}\cdot\theta(a)_j^{-1}$).

Applying Lemma \ref{extending one-machine} then yields a one-machine computation of the $j$-th machine $\pazocal{F}:W_0'\to\dots\to W_r'$ in the standard base such that the restriction to the base $\{t(2)\}B_4(2)$ agrees with $\pazocal{E}'$. As the final configuration of $\pazocal{E}'$ is a copy of $A_5$, the constructions in Lemma \ref{extending one-machine} imply that $W_r'$ is $\theta(a)_j$-admissible, so that $W_0'$ is accepted by $\pazocal{F}$. So, $W_0'(2)\equiv W(2)$, and since both $W_0'$ and $W$ are accepted, $W_0'(i)\equiv W(i)$ for all $i\geq2$. Any difference between $W$ and $W_0'$ must then be in their projections onto the `special' input sector.

Suppose $W_0'(1)\neq W(1)$. Since at least one of $W$ and $W_0'$ must have nonempty `special' input sector and both are accepted by one-machine computations of the $j$-th machine, it follows immediately that $j=1$. Then the construction of Lemma \ref{extending one-machine} implies that $W_0'(1)$ is a coordinate shift of $W_0'(2)$, so that $W(1)$ is not a coordinate shift of $W(2)$; but then $\pazocal{C}$ is a one-machine computation of the first machine that accepts $W$, which is impossible as the first machine acts in parallel on all bases $\{t(i)\}B_4(i)$.

So, $W_0'\equiv W$, and so augmenting $\pazocal{F}$ yields an accepting computation $\bar{\pazocal{F}}:W\equiv W_0'\to\dots\to W_r'\to W_{ac}$. Since the restriction of $\pazocal{F}$ to the base $\{t(2)\}B_4(2)$ is $\pazocal{E}'$, it follows that the restriction of $\pazocal{F}$ to the base $\{t(i)\}B_4(i)$ is a copy of $\pazocal{D}$ for all $i\geq2$. So, for all $i\geq2$, the length of $\bar{\pazocal{F}}$ is at most $c_3\|W(i)\|+1$ and $|W_j'(i)|_a\leq c_3|W(i)|_a$ for all $0\leq j\leq r$.

Meanwhile, as each $W_j'$ is accepted, Lemma \ref{accepted configuration a-length} implies that $|W_j'(1)|_a\leq\frac{3}{2}|W_j'(2)|_a\leq\frac{3}{2}c_3|W(2)|_a\leq 3c_3|W(1)|_a$ for all $0\leq j\leq r$ and the length of the computation is at most $2c_3\|W(1)\|+1$.

Now, suppose $\ell(W)=\ell>1$ and set $\pazocal{C}:W\equiv W_0\to\dots\to W_t\equiv W_{ac}$ as an accepting computation with exactly $\ell$ maximal one-machine subcomputations. Set $\bar{\pazocal{C}}$ as the inverse computation of $\pazocal{C}$ and $\bar{\pazocal{C}}_1:W_t\to\dots\to W_s$ as the maximal one-machine subcomputation, say of the $i$-th machine, serving as a prefix for $\bar{\pazocal{C}}$. Then the rule corresponding to the transition $W_{s+1}\to W_s$ must be either $\theta(s)_i^{-1}$ or $\theta(a)_i$. If it is $\theta(a)_i$, then $W_s\equiv W_{ac}$, so that $W_0\to\dots\to W_s$ is an accepting computation made up of one less maximal one-machine computation.

So, we assume that $W_{s+1}\to W_s$ is given by the rule $\theta(s)_i^{-1}$. By Lemma \ref{M language}, there exists $u^n\in\pazocal{L}$ such that $W_s\equiv I(u^n)$ or $W_s\equiv J(u^n)$.

It follows from the definition of the rules that $W_s\to W_{s-1}$ is given by the rule $\theta(s)_j$ for $j\neq i$. So, Lemma \ref{turn} implies that $W_s$ has empty `special' input sector.

In the case that $i=1$, it then follows that $u^n=1$, so that $W_s$ is the start input configuration with empty input. By the definition of the machine, there exists an accepting one-machine computation $W_s\equiv V_s\to\dots\to V_{t'}\equiv W_t$ of the second machine consisting only of transition rules, so that $t'\leq t$. Then, we can consider the shorter accepting computation $$W\equiv W_0\to\dots\to W_{s-1}\equiv V_{s+1}\to\dots\to V_{t'}\equiv W_t\equiv W_{ac}$$ (with perhaps some reductions) made up of at least one less maximal one-machine computation.

So, $i=2$, so that $W_s\equiv J(u^n)$ for some $u^n\in\pazocal{L}$. If $u^n=1$, then again one can construct a one-machine computation of the first machine consisting only of transition rules and, as above, construct an accepting computation of $W$ made up of at least one less maximal one-machine computation. So, $u^n$ is nonempty.

Then, consider the maximal one-machine subcomputation $W_s\to\dots\to W_r$. If this contains just one rule, then $W\equiv W_s\cdot\theta(s)_1$, so that the appropriate bounds arise from applying Lemma \ref{M_5 accepted configurations} to the one-machine computation $W_s\to\dots\to W_t$ arising above.

Otherwise, the step history of this maximal one-machine subcomputation has prefix $(s)_1(1)_1$. 

Since the `special' input sector is empty for $J(u^n)$, the application of Lemma \ref{multiply one letter} to the restriction of the computation to this sector shows that only an empty computation of step history $(1)_1$ can result in a $(12)_1$-admissible configuration; but since $u$ is nonempty, this resulting configuration is not $(12)_1$-admissible.

Similarly, the application of Lemma \ref{multiply one letter} to the restriction of the computation to the copies of the $R_0P_1$-sector imply that only an empty computation of step history $(1)_1$ can yield a $(s)_1^{-1}$-admissible word. However, this would yield a computation that is not reduced.

So, the entire subcomputation $W_s\to\dots\to W_0$ has step history $(s)_1(1)_1$. Applying Lemma \ref{multiply one letter}$(b,c,d)$ to the restriction of this subcomputation to the `special' input sector then gives the inequalities $s\leq\|W(1)\|$ and $|W_j(1)|_a\leq|W(1)|_a$ for all $0\leq j\leq s$.

Since $W_t\to\dots\to W_s$ is a one-machine computation of the second machine, as above we find that $t-s\leq 2c_3\|W(1)\|+1$ and $|W_j(1)|_a\leq 3c_3|W_s(1)|_a$ for all $s\leq j\leq t$. So, $t\leq (2c_3+1)\|W(1)\|+1$ and $|W_j(1)|_a\leq 3c_3|W(1)|_a+1$ for all $0\leq j\leq t$.

Lemma \ref{accepted configuration a-length} then gives appropriate bounds in terms of $|W(x)|_a$ and $\|W(x)\|$ for $x\geq2$.

\end{proof}

\smallskip


\subsection{Computations of $\textbf{M}$ with long history} \

The next two lemmas are easily seen using Lemmas \ref{M_5 special computations} and \ref{M_5 factor} and restricting any computation to the base $\{t(x)\}B_4(x)$ for $x\geq2$.

\begin{lemma} \label{M special computations}

Fix $i\in\{1,\dots,L\}$. Let $\pazocal{C}:W_0\to\dots\to W_t$ be a one-machine computation of the $j$-th machine for $j=1$ (respectively $j=2$) and such that (1) $W_0\equiv A(i)$ or $W_0\equiv I(u^n,i)$ (respectively $W_0\equiv J(u^n,i)$) for some $u^n\in\pazocal{L}$ and (2) $W_t\equiv A(i)$ or $W_t\equiv I(v^n,i)$ (respectively $W_t\equiv J(v^n,i)$) for some $v^n\in\pazocal{L}$. Then the sum of the lengths of all subcomputations of $\pazocal{C}$ with step histories $(12)_j(2)_j(23)_j$, $(32)_j(2)_j(21)_j$, $(34)_j(4)_j(45)_j$, and $(54)_j(4)_j(43)_j$ is at least $0.99t$.

\end{lemma}

\begin{lemma} \label{M one-machine factor}

Fix $i\in\{1,\dots,L\}$. Let $\pazocal{C}:W_0\to\dots\to W_t$ be a one-machine computation of the $j$-th machine for $j=1$ (respectively $j=2$) and such that $W_0\equiv A(i)$ or $W_0\equiv I(u^n,i)$ (respectively $W_0\equiv J(u^n,i)$) for some $u^n\in\pazocal{L}$. Let $H$ be the history of $\pazocal{C}$. Then there exists a factorization $H\equiv H_1H_2$ such that:

\begin{addmargin}[1em]{0em}

(1) for $\pazocal{C}_1$ the subcomputation of history $H_1$, the sum of the lengths of all subcomputations of $\pazocal{C}_1$ with step histories $(12)_j(2)_j(23)_j$, $(32)_j(2)_j(21)_j$, $(34)_j(4)_j(45)_j$, and $(54)_j(4)_j(43)_j$ is at least $0.99\|H_1\|$

(2) $\|H_2\|\leq c_4\|W_t\|$ or $\|H_2\|\leq\|H_1\|/200$

\end{addmargin}

\end{lemma}

For any accepted configuration $W$, fix an accepting computation $\pazocal{C}(W)$ according to Lemma \ref{M accepted configurations}.

\begin{lemma} \label{M projected long history}

Let $W_0$ be an accepted configuration and $\pazocal{C}:W_0(i)\equiv V_0\to\dots\to V_t$ be a reduced computation of $\textbf{M}$ for some $i\in\{2,\dots,L\}$. Then there exists an accepted configuration $W_t$ such that $W_t(i)\equiv V_t$. Let $H_0,H_t$ be the histories of $\pazocal{C}(W_0),\pazocal{C}(W_t)$, respectively. Then for the parameters $c_5$ and $c_6$, either:

\begin{addmargin}[1em]{0em}

$(a)$ $t\leq c_5\max(\|V_0\|,\|V_t\|)$ and $\|V_j\|\leq c_6\max(\|W_0(i)\|,\|W_t(i)\|)$ for every $j=0,\dots,t$ or

$(b)$ $\|H_0\|+\|H_t\|\leq t/500$ and the sum of the lengths of all subcomputations of $\pazocal{C}$ with step histories $(12)_i(2)_i(23)_i$, $(32)_i(2)_i(21)_i$, $(34)_i(4)_i(45)_i$, and $(54)_i(4)_i(43)_i$ is at least $0.98t$.

\end{addmargin}

\end{lemma}

\begin{proof}

Let $H'$ be the history of $\pazocal{C}$ and factor it as $H'=H'_1\dots H'_l$ for $l\geq1$ where each $H'_j$ is the history of a maximal one-machine subcomputation.

Let $\pazocal{C}_j:V_x\to\dots\to V_y$ be the subcomputation of $\pazocal{C}$ with history $H'_j$. Applying Lemma \ref{extending one-machine} to $\pazocal{C}_j$ then produces a reduced computation of the standard base $\pazocal{C}_j':W_x^{(j)}\to\dots\to W_y^{(j)}$ of history $H_j'$ such that $W_z^{(j)}(i)\equiv V_z$ for $x\leq z\leq y$.

Suppose $W_0$ is not $H_1'$-admissible and let $H_1''$ be the maximal prefix for which it is admissible (perhaps $H_1''$ is empty). Then, let $W_r\equiv W_0\cdot H_1''$. Then $W_r$ is not $\theta$-admissible for $\theta$ the subsequent rule of $H_1'$, so that Lemma \ref{projection admissible configuration not} applies, i.e either $W_r\equiv I(u^n)$ for some $u^n\in\pazocal{L}$ or $W_r$ has $u^n$ written in the `special' input sector for some nonempty $u^n\in\pazocal{L}$ while all other input sectors are empty. In either case, it is easy to see that the construction of Lemma \ref{extending one-machine} applied to $W_r(i)$ yields an accepted configuration $W_r^{(1)}$. As a result, $W_0^{(1)}$ is also accepted.

This allows one to conclude that all configurations of $\pazocal{C}_1'$ are accepted. If $l\geq2$, then the final configuration of $\pazocal{C}_1'$, $W_s^{(1)}$, must either be $W_{ac}$ or $I(u^n)$ or $J(u^n)$ for some $u^n\in\pazocal{L}$. In each case, it follows that $W_s^{(2)}$ is again of one of these three forms, so that each configuration of $\pazocal{C}_2'$ is accepted. Continuing in this way, $W_t=W_t^{(l)}$ is an accepted configuration.

As in the proof of Lemma \ref{M_5 long history}, it suffices to assume that $t>c_5\max(\|V_0\|,\|V_t\|)$. Then, by Lemma \ref{M accepted configurations}, $\|H_0\|+\|H_t\|\leq 2c_4\max(\|V_0\|,\|V_t\|)\leq t/500$.

If $l=1$, then the rest of the statement follows from Lemma \ref{M_5 long history}. So, assume $l>1$.

For $j=1,\dots,l-1$, letting the terminal configuration of the subcomputation $\pazocal{C}_j$ of history $H_j'$ be $W_j'$, we see that $W_j'$ is one of $W_{ac}$, $I(u^n)$, or $J(u^n)$ for some $u^n\in\pazocal{L}$. Lemma \ref{M special computations} then implies that for $2\leq j\leq l-1$, the sum of the lengths of all maximal subcomputations of $\pazocal{C}_j$ with the relevant step histories is at least $0.99\|H_j'\|$.

Set $\|H_1'\|=x$. If $W_1'\equiv W_{ac}$, then $\|V_0\|\geq\|V_x\|$; in this case, set $x=x'$. Otherwise, applying Lemma \ref{M one-machine factor} to the inverse computation $\bar{\pazocal{C}}_1$, we can find an $x'\leq x$ such that the sum of the maximal subcomputations of $V_{x'}\to\dots\to V_x$ with relevant step history is at least $0.99(x-x')$ and either $x'\leq c_4\|V_0\|$ or $x'\leq(x-x')/200$. In the latter case, clearly $x'<t/200$; for the former case,
$$x'\leq c_4\|V_0\|\leq c_4\max(\|V_0\|,\|V_t\|)<c_5\max(\|V_0\|,\|V_t\|)/200<t/200$$
Similarly, setting $\|H_l'\|=y$, there exists $y'\leq y$ such that the sum of the lengths of the subcomputations $V_{t-y}\to\dots\to V_{t-y'}$ with the relevant step histories is at least $0.99(y-y')$ and $y'<t/200$.

So, the sum of the subcomputations of $\pazocal{C}$ with the relevant step histories is at least
$$0.99(t-x'-y')\geq0.99^2t>0.98t$$

%
%

\end{proof}

The next lemma is proved in exactly the same way that Lemma \ref{M_5 long history has controlled} is proved, using Lemma \ref{M projected long history} in place of Lemma \ref{M_5 long history}.

\begin{lemma} \label{M projected long history controlled}

Let $W_0$ be an accepted configuration and $\pazocal{C}:W_0(i)\equiv V_0\to\dots\to V_t$ be a reduced computation of $\textbf{M}$ for some $i\in\{2,\dots,L\}$. Then the history of any subcomputation $\pazocal{D}:V_r\to\dots\to V_s$ of $\pazocal{C}$ (or the inverse of $\pazocal{D}$) of length at least $0.4t$ contains a subcomputation with controlled history.

\end{lemma}

\begin{lemma} \label{long step history}

Let $\pazocal{C}$ be a reduced computation with base $\{t(i)\}B_4(i)$ for some $i\geq2$. Then the step history of $\pazocal{C}$ either:

\begin{addmargin}[1em]{0em}

(A) contains a subword of the form $(34)_i(4)_i(45)_i$, $(54)_i(4)_i(43)_i$, $(12)_i(2)_i(23)_i$, or $(32)_i(2)_i(21)_i$

(B) has length at most $8$

\end{addmargin}

\end{lemma}

\begin{proof}

Suppose (A) is not satisfied. By Lemma \ref{(A) and (B)}, any such one-machine computation has step history of length at most 5, while any such one-machine computation ending (or starting) with a subword of a start or end configuration has step history of length at most 4. What's more, the same lemma implies that there are at most two maximal one-machine subcomputations of $\pazocal{C}$.

\end{proof}

\begin{lemma} \label{M one-step}

Let $W_0$ be an accepted configuration and $\pazocal{C}:W_0(i)\equiv V_0\to\dots\to V_t$ be a reduced computation of $\textbf{M}$ for some $i\in\{2,\dots,L\}$. Suppose $\pazocal{C}$ has step history of length 1 and $|V_j|_a>3|V_0|_a$ for some $1\leq j\leq t$. Then there is a sector $QQ'$ such that a state letter from $Q$ or from $Q'$ inserts an $a$-letter increasing the length of the sector for each rule of the subcomputation $V_j\to\dots\to V_t$.

\end{lemma}

\begin{proof}

This follows immediately from Lemma \ref{M_5 one-step} and the definition of $\textbf{M}$.

\end{proof}

\medskip


\section{Groups Associated to an $S$-machine and their Diagrams}

\subsection{The groups} \

As in previous literature (for example [19], [23], [26]), to a cyclic $S$-machine $\textbf{S}$, we now associate two groups, $M(\textbf{S})$ and $G(\textbf{S})$, which `simulate' the work of $\textbf{S}$ (in the precise sense described in Section 6.3).

Let $\textbf{S}$ be a cyclic recognizing $S$-machine with hardware $(Y,Q)$, where $Q=\sqcup_{i=0}^s Q_i$ ($Q_0=Q_{s+1}$) and $Y=\sqcup_{i=1}^s Y_i$, and software the set of rules $\Theta=\Theta^+\sqcup\Theta^-$. Further define $Y_{s+1}=\emptyset$, so that every rule locks the $Q_sQ_0$-sector. Denote the accept word of $\textbf{S}$ by $W_{ac}$.

For $\theta\in\Theta^+$, applying Lemma \ref{simplify rules} allows us to assume that $\theta$ takes the form $$\theta=[q_0\to q_0'u_1, \ q_1\to v_1q_1'u_2, \ \dots, \ q_{s-1}\to v_{s-1}q_{s-1}'u_s, \ q_s\to v_sq_s']$$ where $q_i,q_i'\in Q_i$, $u_i$ and $v_i$ are either empty or letters in $Y_i^{\pm1}$, and some of the arrows can have the form $\xrightarrow{\ell}$. Note that if $\theta$ locks the $i$-th sector, then both $u_i$ and $v_i$ are necessarily empty.

Then, define $R=\{\theta_i: \theta\in\Theta^+,0\leq i\leq s\}$. 

The group $M(\textbf{S})$ is then defined by taking the (finite) generating set $\pazocal{X}=Q\cup Y\cup R$ and subjecting it to the (finite number of) relations:

\begin{addmargin}[1em]{0em}

$\bullet$ $q_i\theta_{i+1}=\theta_i v_iq_i'u_{i+1}$ for all $\theta\in\Theta^+$ and $0\leq i\leq s$

$\bullet$ $\theta_ia=a\theta_i$ for all $0\leq i\leq s$ and $a\in Y_i(\theta)$

\end{addmargin}

Note that  the number of $a$-letters in any part of $\theta$, and so in any relation of the above forms, is at most two.

As in the language of computations of $S$-machines, letters from $Q$ are called \textit{$q$-letters}, those from $Y$ are called \textit{$a$-letters}, and those from $R$ are called \textit{$\theta$-letters}. The relations of the form $q_i\theta_{i+1}=\theta_iv_iq_i'u_{i+1}$ are called \textit{$(\theta,q)$-relations}, while those of the form $\theta_ia=a\theta_i$ are called \textit{$(\theta,a)$-relations}.

To make these formulas less muddled, it is convenient to omit the reference to the indices of the letters of $R$. This notational quirk may make it appear as though $\theta$ commutes with the letters of $Y_i(\theta)$ and conjugates $q_i$ to $v_iq_i'u_{i+1}$ for each $i$; it should be noted that these statements are not strictly true. Further, it is useful to note that if $\theta$ locks the $i$-th sector, then $Y_i(\theta)=\emptyset$ so that $\theta$ has no relation with the elements of $Y_i$.

However, this group evidently lacks any reference to the accept configuration. To amend this, the group $G(\textbf{S})$ is constructed by adding one more relation to the presentation of $M(\textbf{S})$, namely the \textit{hub-relation} $W_{ac}=1$. In other words, $G(\textbf{S})\cong M(\textbf{S})/\gen{\gen{W_{ac}}}$.

For the purposes of this paper, though, it is useful to consider extra relations within the language of tape letters called \textit{$a$-relations}. The groups $M_a(\textbf{S})$ and $G_a(\textbf{S})$ introduced here then correspond to this construction. For our purposes, the $a$-relations will correspond to the words over the alphabet of the `special' input sector (this alphabet is identified with $\pazocal{A}$) that represent the trivial element in $B(2,n)$. In other words, given $\Omega$ the set of words over $\pazocal{A}$ representing the identity in $B(2,n)$, we have $M_a(\textbf{S})\cong M(\textbf{S})/\gen{\gen{\Omega}}$ and $G_a(\textbf{S})\cong G(\textbf{S})/\gen{\gen{\Omega}}$.

An important note is that, though they remain finitely generated, $M_a(\textbf{S})$ and $G_a(\textbf{S})$ may no longer be finitely presented.


\subsection{Bands and annuli} \

Many of the rest of the arguments presented throughout the rest of this paper rely on diagrams over the presentations (Section 2.1) of the groups constructed in Section 6.1. To present these arguments as simply as possible, we first differentiate between the types of edges and cells that arise in such diagrams in a way similar to [19] and [26]. 

An edge labelled by a state letter is called a \textit{$q$-edge}. Similarly, an edge labelled by a tape letter is called an \textit{$a$-edge}, and one labelled by a $\theta$-letter is a \textit{$\theta$-edge}. For a path \textbf{p} in $\Delta$, $\|\textbf{p}\|$ denotes its length while $|\textbf{p}|_a$, $|\textbf{p}|_{\theta}$, and $|\textbf{p}|_q$ are its \textit{$a$-length}, \textit{$\theta$-length}, and \textit{$q$-length}, i.e the number of such edges in the path.

Cells corresponding to $(\theta,q)$-relations are called \textit{$(\theta,q)$-cells}. Similarly, there are \textit{$(\theta,a)$-cells}, \textit{$a$-cells}, and \textit{hubs}.

Let $\Delta$ be a reduced van Kampen diagram over the presentation $\gen{X\mid\pazocal{R}}$ and $\pazocal{Z}\subseteq X$. Let $\pazocal{B}$ be a sequence of (distinct) cells $(\Pi_1,\dots\Pi_n)$ in $\Delta$. Then $\pazocal{B}$ is called a \textit{$\pazocal{Z}$-band} if:

\begin{addmargin}[1em]{0em}

$\bullet$ every two consecutive cells $\Pi_i$ and $\Pi_{i+1}$ have a common boundary edge $\textbf{e}_i$ labeled by a letter from $\pazocal{Z}^{\pm1}$

$\bullet$ every cell $\Pi_i$ has exactly two $\pazocal{Z}$-edges in its boundary, $\textbf{e}_{i-1}^{-1}$ and $\textbf{e}_i$, so that $\text{Lab}(\textbf{e}_{i-1})$ and $\text{Lab}(\textbf{e}_i)$ are either both positive or both negative

$\bullet$ if $n=0$, then $\pazocal{B}$ is a single $\pazocal{Z}$-edge

\end{addmargin}

A $\pazocal{Z}$-band $\pazocal{B}$ is \textit{maximal} if it is not contained in any other $\pazocal{Z}$-band. Note that every $\pazocal{Z}$-edge is contained in a maximal $\pazocal{Z}$-band. The \textit{length} of a band is the number of $\pazocal{R}$-cells that comprise it.

In a $\pazocal{Z}$-band $\pazocal{B}$ of length $n$, using only edges from the contours of $\pi_1,\dots,\pi_n$, there exists a closed path $\textbf{e}_0^{-1}\textbf{q}_1\textbf{e}_n\textbf{q}_2^{-1}$ with $\textbf{q}_1$ and $\textbf{q}_2$ simple paths. In this case, $\textbf{q}_1$ is called the \text{bottom} of $\pazocal{B}$, denoted $\textbf{bot}(\pazocal{B})$, while $\textbf{q}_2$ is called the \textit{top} of $\pazocal{B}$ and denoted $\textbf{top}(\pazocal{B})$. 

If $\textbf{e}_0=\textbf{e}_n$ is a $\pazocal{Z}$-band of length $n$, then $\pazocal{B}$ is called a \textit{$\pazocal{Z}$-annulus}. If $\pazocal{B}$ is a non-annular $\pazocal{Z}$-band, then $\textbf{e}_0^{-1}\textbf{q}_1\textbf{e}_n\textbf{q}_2^{-1}$ is the \textit{standard factorization} of the contour of $\pazocal{B}$. If either $(\textbf{e}_0^{-1}\textbf{q}_1\textbf{e}_n)^{\pm1}$ or $(\textbf{e}_n\textbf{q}_2^{-1}\textbf{e}_0^{-1})^{\pm1}$ is a subpath of $\partial\Delta$, then $\pazocal{B}$ is called a \textit{rim band}.

A $\pazocal{Z}_1$-band and a $\pazocal{Z}_2$-band \textit{cross} if they have a common cell and $\pazocal{Z}_1\cap\pazocal{Z}_2=\emptyset$.

In diagrams over the canonical presentations of the groups of interest, there exist \textit{$q$-bands} corresponding to bands arising from $\pazocal{Z}=Q_i$ for some $i$, where every cell is a $(\theta,q)$-cell. Similarly, there exist \textit{$\theta$-bands} for $\theta\in\Theta^+$ and \textit{$a$-bands} for $a\in Y$. For $a$-bands, however, it is useful to restrict the definition of a band to disallow the inclusion of $(\theta,q)$-cells and $a$-cells, so that $a$-bands consist only of $(\theta,a)$-cells.

By the makeup of the cells, distinct maximal $q$-bands ($\theta$-bands, $a$-bands) cannot intersect.

Given an $a$-band $\pazocal{B}$, the makeup of the relations of the groups dictates that each of the $a$-edges $\textbf{e}_0,\dots,\textbf{e}_n$ is labelled identically, i.e by the same $a$-letter. Similarly, the $\theta$-edges of a $\theta$-band correspond to the same rule; however, the (suppressed) index of two such $\theta$-edges may differ.

If a maximal $a$-band contains a cell with an $a$-edge that is also on the contour of a $(\theta,q)$-cell, then the $a$-band is said to \textit{end} (or \textit{start}) on that $(\theta,q)$-cell and the corresponding $a$-edge is said to be the \textit{end} (or \textit{start}) of the band. This definition extends similarly, so that: 

\begin{addmargin}[1em]{0em}

$\bullet$ a maximal $a$-band can end on a $(\theta,q)$-cell, on an $a$-cell, or on the diagram's contour

$\bullet$ a maximal $\theta$-band can end only on the diagram's contour

$\bullet$ a maximal $q$-band can end on a hub or on the diagram's contour

\end{addmargin}

Note that if a maximal $\theta$-band ($a$-band, $q$-band) ends as above in one part of the diagram, then it must also end in another part of the diagram as it cannot be a $\theta$-annulus ($a$-annulus, $q$-annulus).

The projection of the label of the top (or bottom) of a $q$-band onto $F(\Theta^+)$ is called the \textit{history} of the band; the \textit{step history} of the band is defined further in the obvious way. The projection of the top (or bottom) of a $\theta$-band onto the alphabet $\{Q_0,\dots,Q_s\}$ is called the \textit{base} of the band.

Suppose the sequence of cells $(\pi_0,\pi_1,\dots,\pi_n)$ comprises a $\theta$-band and $(\gamma_0,\gamma_1,\dots,\gamma_k)$ a $q$-band such that $\pi_0=\gamma_0$, $\pi_n=\gamma_k$, and no other cells are shared. Suppose further that $\pi_0$ and $\pi_n$ both have edges on the outer countour of the annulus bounded by the two bands. Then the union of these two bands is called a \textit{$(\theta,q)$-annulus} and $\pi_0$ and $\pi_n$ are called its \textit{corner} cells. A \textit{$(\theta,a)$-annulus} is defined similarly.

\smallskip

The following Lemma is proved in a more general setting in [18]:

\begin{lemma} \label{M(S) annuli}

\textit{(Lemma 6.1 of [18])} A reduced van Kampen diagram $\Delta$ over $M(\textbf{S})$ has no:

\begin{addmargin}[1em]{0em}

(1) $q$-annuli

(2) $\theta$-annuli

(3) $a$-annuli. 

(4) $(\theta,q)$-annuli

(5) $(\theta,a)$-annuli

\end{addmargin}

\end{lemma}

As a result, in a reduced diagram $\Delta$ over $M(\textbf{S})$, if a maximal $\theta$-band and a maximal $q$-band ($a$-band) cross, then their intersection is exactly one $(\theta,q)$-cell ($(\theta,a)$-cell). Further, every maximal $\theta$-band and maximal $q$-band ends on $\partial\Delta$ in two places.

\begin{lemma} \label{annuli lower bound}

If $\Delta$ is a reduced diagram over $G_a(\textbf{S})$ and $S$ is a $(\theta,q)$-annulus (respectively a $(\theta,a)$-annulus) with boundary $q$-band (respectively $a$-band) $\pazocal{Q}$, then the length of $\pazocal{Q}$ is at least three.

\end{lemma}

\begin{proof}

By the definition of the annulus, the history of the $q$-band must be of the form $\theta w\theta^{-1}$ for some $\theta\in\Theta$ and $w\in F(\Theta^+)$. If the length of $\pazocal{Q}$ is two, then this history is the unreduced word $\theta\theta^{-1}$, meaning $\pazocal{Q}$ is a pair of cancellable $(\theta,q)$-cells.

The argument is identical for a $(\theta,a)$-annulus by considering the label of the top (or bottom) of the corresponding $a$-band.

\end{proof}

\medskip


\subsection{Trapezia} \

Let $\Delta$ be a reduced diagram over the canonical presentation of $M(\textbf{S})$ whose contour is of the form $\textbf{p}_1^{-1}\textbf{q}_1\textbf{p}_2\textbf{q}_2^{-1}$, where $\textbf{p}_1$ and $\textbf{p}_2$ are sides of $q$-bands and $\textbf{q}_1$ and $\textbf{q}_2$ are maximal parts of the sides of $\theta$-bands whose labels start and end with $q$-letters. Then $\Delta$ is called a \textit{trapezium}.

In this case, $\textbf{q}_1$ and $\textbf{q}_2$ are called the \textit{top} and \textit{bottom} of the trapezium, respectively, while $\textbf{p}_1$ and $\textbf{p}_2$ are the \textit{left} and \textit{right} sides. Further, $\textbf{p}_1^{-1}\textbf{q}_1\textbf{p}_2\textbf{q}_2^{-1}$ is called the \textit{standard factorization} of the contour.

The \textit{(step) history} of the trapezium is the (step) history of the rim $q$-band with $\textbf{p}_2$ as one of its sides and the length of this history is the trapezium's \textit{height}. The base of $\text{Lab}(\textbf{q}_1)$ is called the \textit{base} of the trapezium.

It's easy to see from this definition that a $\theta$-band $\pazocal{T}$ can be viewed as a trapezium of height 1 as long as its top and bottom start and end with $q$-edges. This extends to all $\theta$-bands starting and ending with a $(\theta,q)$-cell if one merely disregards any $a$-edges of the top and bottom that precede the first $q$-edge or follow the final $q$-edge. These are called the \textit{trimmed} top and bottom of the band, denoted $\textbf{ttop}(\pazocal{T})$ and $\textbf{tbot}(\pazocal{T})$.

\begin{lemma} \label{theta-bands are one-rule computations}

Let $\pazocal{T}$ be a $\theta$-band in a reduced diagram $\Delta$ over the canonical presentation of $M(\textbf{S})$. Suppose $\pazocal{T}$ consists of at least two $(\theta,q)$-cells. Then $\lab(\textbf{tbot}(\pazocal{T}))$ and $\lab(\textbf{ttop}(\pazocal{T}))$ are admissible words. Moreover, if $\theta$ is the rule corresponding to the band $\pazocal{T}$, then $\lab(\textbf{tbot}(\pazocal{T}))$ is $\theta$-admissible and $\lab(\textbf{tbot}(\pazocal{T}))\cdot\theta\equiv\lab(\textbf{ttop}(\pazocal{T}))$.

\end{lemma}

\begin{proof}

Suppose $\theta\in\Theta^+$.

Let $\textbf{q}_1,\textbf{q}_2$ be the first two $q$-edges of $\textbf{bot}(\pazocal{T})$ and $q_1=\lab(\textbf{q}_1),q_2=\lab(\textbf{q}_2)$. So, $\lab(\textbf{tbot}(\pazocal{T}))$ has prefix $q_1wq_2$ for some $w\in F(Y)$. Let $\pi_1,\pi_2$ be the cells of $\pazocal{T}$ with $\textbf{q}_1,\textbf{q}_2$ on its contour, respectively.

For $0\leq i\leq s-1$, suppose $q_1\in Q_i$. Then the $i$-th part of $\theta$ must be $q_1\to u_iq_1'v_{i+1}$ for some $q_1'\in Q_i$, $u_i\in F(Y_i(\theta))$, and $v_{i+1}\in F(Y_{i+1}(\theta))$ with $\|u_i\|,\|v_{i+1}\|\leq1$. So, $\lab(\partial\pi_1)\equiv\theta_i^{-1}q_1\theta_{i+1}v_{i+1}^{-1}(q_1')^{-1}u_i^{-1}$. If there exists any cell of $\pazocal{T}$ between $\pi_1$ and $\pi_2$, it must be a $(\theta,a)$-cell with an edge labelled by $\theta_{i+1}$ on its contour; so, $w\in F(Y_{i+1}(\theta))$. 

What's more, the label of $\partial\pi_2$ must have a subword $\theta_{i+1}^{-1}q_2$. By the definition of the $(\theta,q)$-relations, this means one of two things:

\begin{addmargin}[1em]{0em}

$(1)$ $q_2\in Q_{i+1}$ and the $(i+1)$-th part of $\theta$ is $q_2\to u_{i+1}q_2'v_{i+2}$ for some $q_2'\in Q_{i+1}$, $u_{i+1}\in F(Y_{i+1}(\theta))$, and $v_{i+2}\in F(Y_{i+2}(\theta))$ with $\|u_{i+1}\|,\|v_{i+2}\|\leq1$; or

$(2)$ $q_2=q_1^{-1}$

\end{addmargin}

In case (1), the subword $q_1wq_2$ of $\lab(\textbf{tbot}(\pazocal{T}))$ satisfies (1) in the requirements for such subwords of admissible words. In case (2), it satisfies condition (2) as long as there is some $(\theta,a)$-cell between them; but this is required in the band, as otherwise $\pi_1$ and $\pi_2$ would be a pair of cancellable cells.

Let $\pazocal{T}_1=(\pi_1,\dots,\pi_2)$ be the subband of $\pazocal{T}$. Then $\lab(\textbf{tbot}(\pazocal{T}_1))\equiv q_1wq_2$. The above arguments make it clear that $q_1wq_2$ is $\theta$-admissible. Further, it is easy to see that $\lab(\textbf{ttop}(\pazocal{T}_1))\equiv (q_1wq_2)\cdot\theta$.

If $q_1\in Q_i^{-1}$, then an analogous argument yields the same conclusion.

Further, if $\textbf{tbot}(\pazocal{T})$ has more than two $q$-edges, then the argument above can be iterated to apply to the whole band.

Finally, if $\theta\in\Theta^-$, then one can apply the analogous argument to $\textbf{ttop}(\pazocal{T})$ to show that $\lab(\textbf{ttop}(\pazocal{T}))$ is $\theta^{-1}$-admissible with $\lab(\textbf{ttop}(\pazocal{T}))\cdot \theta^{-1}\equiv \lab(\textbf{tbot}(\pazocal{T}))$.

\end{proof}

\begin{lemma} \label{one-rule computations are theta-bands}

Let $U\to V$ be a computation of $\textbf{S}$ with history $H$ of length $1$, so that $H=\theta\in\Theta$. Then there exists a trapezium $\Delta$ consisting of one $\theta$-band $\pazocal{T}$ corresponding to the rule $\theta$ such that $\lab(\textbf{tbot}(\pazocal{T}))\equiv U$ and $\lab(\textbf{ttop}(\pazocal{T}))\equiv V$.

\end{lemma}

\begin{proof}

Suppose $\theta\in\Theta^+$ and set $U\equiv q_0^{\eps_0}w_1q_1^{\eps_1}\dots w_lq_l^{\eps_l}$ so that for each $0\leq i\leq l$, $q_i\in Q_{j(i)}$ for some $0\leq j(i)\leq s-1$ and $\eps_i\in\{\pm1\}$.

Then $q_i\in Q(\theta)$ for each $0\leq i\leq l$, so that the $j(i)$-th part of $\theta$ takes the form $q_i\to u_{j(i)}q_i'v_{j(i)+1}$ for some $q_i'\in Q_{j(i)}$, $u_{j(i)}\in F(Y_{j(i)}(\theta))$, and $v_{j(i)+1}\in F(Y_{j(i)+1}(\theta))$. So, there are relations of $M(\textbf{S})$ of the form $R_i=\theta_{j(i)}^{-1}q_i\theta_{j(i)+1}v_{j(i)+1}^{-1}(q_i')^{-1}u_{j(i)}^{-1}$ for all $i$.

If $\eps_i=1$, then each letter of $w_{i+1}$ is an element of $Y_{j(i)+1}(\theta)$ since $U$ is $\theta$-admissible. So, there are relations of $M(\textbf{S})$ the form $\theta_{j(i)+1}^{-1}a\theta_{j(i)+1}a^{-1}$ for each letter $a$ of $w_{i+1}$. So, gluing along the edges labelled by $\theta_{j(i)+1}^{\pm1}$, one can construct a $\theta$-band $\pazocal{T}_{i+1}$ with contour label $\theta_{j(i)+1}^{-1}w_{i+1}\theta_{j(i)+1}w_{i+1}^{-1}$.

If $\eps_i=-1$, then each letter of of $w_{i+1}$ is in $Y_{j(i)}(\theta)$ since $U$ is $\theta$-admissible. So, there are relations of $M(\textbf{S})$ of the form $\theta_{j(i)}^{-1}a\theta_{j(i)+1}a^{-1}$ for each letter $a$ of $w_{i+1}$. So, gluing along the edges labelled by $\theta_{j(i)}^{\pm1}$, one can construct a $\theta$-band $\pazocal{T}_{i+1}$ with contour label $\theta_{j(i)}^{-1}w_{i+1}\theta_{j(i)}w_{i+1}^{-1}$.

Now, let $\pi_i$ be a cell with boundary labelled by $R_i^{\eps_i}$. Then, for either possibility of $\eps_i$, one can glue $\pazocal{T}_i$ and $\pazocal{T}_{i+1}$ to the left and right of $\pi_i$, respectively.

After $0$-refinement to cancel any adjacent edges with mutually inverse labels, this process produces a $\theta$-band $\pazocal{T}$ corresponding to the rule $\theta$ with $\lab(\textbf{bot}(\pazocal{T}))\equiv U$. By the makeup of the band, it is easy to see that $\lab(\textbf{ttop}(\pazocal{T}))\equiv V$.

If $\theta\in\Theta^-$, then the same construction forms a $\theta$-band $\pazocal{T}$ corresponding to the rule $\theta^{-1}$ with $\lab(\textbf{bot}(\pazocal{T}))\equiv V$ and $\lab(\textbf{ttop}(\pazocal{T}))\equiv U$. Taking the `inverse' of this band (i.e inverting the label of each cell) produces a $\theta$-band corresponding to $\theta$ as in the statement.

\end{proof}

Note that the any trapezium $\Delta$ of height $h\geq1$ can be decomposed into $\theta$-bands $\pazocal{T}_1,\dots,\pazocal{T}_h$ connecting the left and right sides of the trapezium, with $\textbf{tbot}(\pazocal{T}_1)$ and $\textbf{ttop}(\pazocal{T}_h)$ making up the bottom and top of $\Delta$, respectively, and $\textbf{ttop}(\pazocal{T}_i)=\textbf{tbot}(\pazocal{T}_{i+1})$ for all $1\leq i\leq h-1$.

The following two lemmas are clear from the previous two lemmas and exemplify how the group $M(\textbf{S})$ simulates the work of the $S$-machine:

\begin{lemma} \label{trapezia are computations}

If $\Delta$ is a trapezium with history $H=\theta_1\dots\theta_k$ for $k\geq1$ with maximal $\theta$-bands $\pazocal{T}_1,\dots,\pazocal{T}_k$ and $U_j\equiv\textbf{tbot}(\pazocal{T}_j)$, $V_j\equiv\textbf{ttop}(\pazocal{T}_j)$ for all $j$, then $H$ is a reduced word, $U_j$ and $V_j$ are admissible words, and $V_j\equiv U_j\cdot\theta_j$ for all $j$.

\end{lemma}

\smallskip

\begin{lemma} \label{computations are trapezia}

For any reduced computation $U\to\dots\to U\cdot H\equiv V$ of the $S$-machine $\textbf{S}$ with $\|H\|\geq1$, there exists a trapezium $\Delta$ with (trimmed) bottom label $U$, (trimmed) top label $V$, and with history $H$.

\end{lemma}

\medskip


\section{Modified length and area functions}

\subsection{Modified length function} \

To assist with the proofs to come, we now modify the length function on group words over the groups associated to an $S$-machine and paths in diagrams over their presentations in the same way as was done in [19], [26], etc. The standard length of a word/path will henceforth be referred to as its \textit{combinatorial length} and the modified length simply as its \textit{length}.

Define a word consisting of no $q$-letters, one $\theta$-letter, and at most two $a$-letters as a \textit{$(\theta,a)$-syllable}. Then, define the length of:

\begin{addmargin}[1em]{0em}

$\bullet$ any $q$-letter as 1

$\bullet$ any $\theta$-letter as 1

$\bullet$ any $a$-letter as the parameter $\delta$ (as indicated in Section 3.3, this should be thought of as a very small positive number)

$\bullet$ any $(\theta,a)$-syllable as 1

\end{addmargin}

For a word $w$ over the generators of the canonical presentation of $G_a(\textbf{S})$ (or any group associated to $\textbf{S}$), define a \textit{decomposition} of $w$ as a factorization of $w$ into a product of letters and $(\theta,a)$-syllables. The length of a decomposition of $w$ is then assigned as the sum of the lengths of the factors. Finally, the length of $w$, denoted $|w|$, is defined to be the smallest length of any of its decompositions.

Naturally, the length of a path in a diagram over the presentations of the groups associated to $\textbf{S}$ is defined by the length of its label.

The following provides basic properties of the length function and is easily proved using Lemma \ref{simplify rules}.

\begin{lemma} \label{lengths}

Let \textbf{s} be a path in a diagram $\Delta$ over the canonical presentation of $G_a(\textbf{S})$ (or any of the groups associated to $\textbf{S}$) consisting of $c$ $\theta$-edges and $d$ $a$-edges. Then:

\begin{addmargin}[1em]{0em} 

$(a)$ $|\textbf{s}|\geq\max(c,c+(d-2c)\delta)$

$(b)$ $|\textbf{s}|=c$ if $\textbf{s}$ is a top or a bottom of a $q$-band

$(c)$ For any product $\textbf{s}=\textbf{s}_1\textbf{s}_2$ of two paths in a diagram, $$|\textbf{s}_1|+|\textbf{s}_2|\geq|\textbf{s}|\geq|\textbf{s}_1|+|\textbf{s}_2|-2\delta$$

$(d)$ Let $\pazocal{T}$ be a $\theta$-band with base of length $l_b$ and $l_a$ the number of $a$-edges in \textbf{top}$(\pazocal{T})$ (or $\textbf{bot}(\pazocal{T})$). Then the number of cells in $\pazocal{T}$ is between $l_a-l_b$ and $l_a+3l_b$.

\end{addmargin}

\end{lemma}

The following is an immediate consequence of the choice of parameters (specifically the choice $J<<\delta^{-1}$) and aids with removing long rim $\theta$-bands from potential counterexample diagrams in future arguments.

\begin{lemma} \label{small rim bands}

\textit{(Lemma 6.3 of [22])} Let $\Delta$ be a disk van Kampen diagram over the canonical presentation of $G_a(\textbf{S})$ with rim $\theta$-band $\pazocal{T}$ having a base of at most $K$ letters, where $K$ is the parameter listed in Section 3.3. Denote by $\Delta'$ the subdiagram $\Delta\setminus\pazocal{T}$. Then $|\partial\Delta|-|\partial\Delta'|>1$.

\end{lemma}

%
%
%

\smallskip


\subsection{Disks} \

Next, we add extra relations to the groups $G(\textbf{S})$ and $G_a(\textbf{S})$ that will aid with later estimates. This is done in the same way as in [19], [26], etc, though no group $G_a(\textbf{S})$ was present in those sources.

These relations, called \textit{disk relations}, are of the form $W=1$ for any configuration $W$ accepted by the machine $\textbf{S}$.

\begin{lemma} \label{disks are relations}

If the configuration $W$ is accepted by the machine $\textbf{S}$, then the word $W$ is trivial over the groups $G(\textbf{S})$ and $G_a(\textbf{S})$.

\end{lemma}

\begin{proof}

As $W$ is accepted, there exists an accepting computation of it with history $H$. By Lemma \ref{computations are trapezia}, there exists a trapezium $\Delta$ corresponding to this accepting computation with trimmed bottom label $W$ and trimmed top label $W_{ac}$. As this is a computation of the standard base and every rule locks the $Q_sQ_0$-sector, one can further assume that no trimming was necessary in $\Delta$, i.e the labels of the bottom and top of $\Delta$ are $W$ and $W_{ac}$, respectively. Finally, it follows that the sides of the trapezium are labelled identically; specifically, they are labelled by the copy of $H$ obtained by adding the index $0$ to each letter. 

So, $W$ and $W_{ac}$ are conjugate in $M(\textbf{S})$. Taking into account the hub relation in both $G(\textbf{S})$ and $G_a(\textbf{S})$ then implies the relation $W=1$.

\end{proof}

As a result of Lemma \ref{disks are relations}, the presentation obtained by adding the disk relations to the group $G(\textbf{S})$ (respectively $G_a(\textbf{S})$) defines a group isomorphic to the group $G(\textbf{S})$ (respectively $G_a(\textbf{S})$). The presentation containing disk relations will be referred to in what follows as the \textit{disk presentation} of the group $G(\textbf{S})$ (respectively $G_a(\textbf{S})$). A cell of a diagram over the disk presentation corresponding to a disk relation (or its inverse) is referred to simply as a \textit{disk}.

One should note the following when considering diagrams over a disk presentation rather than diagrams over a canonical presentation:

\begin{addmargin}[1em]{0em}

$\bullet$ The disk presentation of $G(\textbf{S})$ or of $G_a(\textbf{S})$ need not be finite. In particular, there may be infinitely many disk relations in this presentation.

$\bullet$ For a word $w\in F(\pazocal{X})$ that represents the trivial element of $G(\textbf{S})$, the minimal area of diagrams over the disk presentation with contour label $w$ can be drastically different than that of diagrams over the canonical presentation of $G(\textbf{S})$.

$\bullet$ In a diagram over the disk presentation of $G_a(\textbf{S})$, a maximal $a$-band can end on a disk in addition to the other possibilities outlined in Section 6.3.

\end{addmargin}

\smallskip


\subsection{Modified area function} \

Similar to how we modified the length function in Section 7.1, we now alter the definition of the area of a van Kampen diagram $\Delta$ over the disk presentations of $G(\textbf{S})$ and $G_a(\textbf{S})$. We do this by introducing a weight function on the cells of diagrams, $\text{wt}$, defined as follows:

\begin{addmargin}[1em]{0em}

$\bullet$ $\text{wt}(\Pi)=1$ \ \ \ \ \ \ \ \ \ \ \ if $\Pi$ is a $(\theta,q)$-cell or a $(\theta,a)$-cell

$\bullet$ $\text{wt}(\Pi)=c_7|\partial\Pi|^2$ \ \ \  if $\Pi$ is a disk

$\bullet$ $\text{wt}(\Pi)=c_7\|\partial\Pi\|^2$ \ \ if $\Pi$ is an $a$-cell

\end{addmargin}

Naturally, we extend this to define the weight of a disk diagram $\Delta$, $\text{wt}(\Delta)$, as the sum of the weights of its cells.

\smallskip

\section{The groups associated to the machine $\textbf{M}$}

To this point, a few of the stated definitions may seem unmotivated. In particular, the use of $a$-relations is probably unclear, as is the reason for the choice of the assignment of the weights of $a$-cells and disks.

In what follows, we deal specifically with the case of our machine of interest, $\textbf{M}$, to elucidate these choices.

\subsection{Minimal diagrams} \

A $q$-letter of the form $t(i)$, i.e one that is the only letter of its part in the hardware of $\textbf{M}$, is called a \textit{$t$-letter}. Accordingly, a $(\theta,q)$-relation corresponding to a $t$-letter is called a \textit{$(\theta,t)$-relation}. Note that for each $\theta$-letter and each $t$-letter, the corresponding $(\theta,t)$-relation is simply $\theta_jt(i)=t(i)\theta_{j+1}$.

Now, we modify the definition of a reduced disk diagram over the canonical presentation of $M_a(\textbf{M})$ or over the disk presentation of $G_a(\textbf{M})$. To this end, we introduce the \textit{signature} of such a diagram $\Delta$ as the five-tuple $s(\Delta)=(\a_1,\a_2,\a_3,\a_4,\a_5)$ where: 

\begin{addmargin}[1em]{0em}

$\bullet$ $\a_1$ is the number of disks in $\Delta$ (of course, this is zero if $\Delta$ is a diagram over $M_a(\textbf{S})$), 

$\bullet$ $\a_2$ is the number of $(\theta,t)$-cells, 

$\bullet$ $\a_3$ is the total number of $(\theta,q)$-cells, 

$\bullet$ $\a_4$ is the number of $a$-cells, and

$\bullet$ $\a_5$ is the total weight

\end{addmargin}

The signatures of diagrams are ordered lexicographically, i.e if $\Delta$ and $\Gamma$ are two diagrams over $G_a(\textbf{M})$ with $s(\Delta)=(\a_1,\dots,\a_5)$ and $s(\Gamma)=(\b_1,\dots,\b_5)$, then $s(\Delta)\leq s(\Gamma)$ if:

\begin{addmargin}[1em]{0em}

$\bullet$ $\a_1\leq\a_2$

$\bullet$ for $i=2,\dots,5$, if $\a_j=\b_j$ for all $j<i$, then $\a_i\leq\b_i$

\end{addmargin}

Finally, a reduced disk diagram $\Delta$ over $G_a(\textbf{M})$ is called \textit{minimal} if for any other such disk diagram $\Gamma$ with $\text{Lab}(\Delta)\equiv\text{Lab}(\Gamma)$, $s(\Delta)\leq s(\Gamma)$.

It follows immediately from the definition that a subdiagram of a minimal diagram is minimal. Moreover, as the signature introduces a grading on the presentation of $G_a(\textbf{M})$ (see Section 2.7), it is easily verified that a word $w\in F(\pazocal{X})$ represents the trivial element of $G_a(\textbf{M})$ if and only if there exists a minimal diagram $\Delta$ over $G_a(\textbf{M})$ with $\text{Lab}(\Delta)\equiv w$.

In what follows, it is taken implicitly that all minimal diagrams over $G_a(\textbf{M})$ are formed over its disk presentation (rather than its canonical presentation).

\smallskip


\subsection{$a$-relations} \

As mentioned in the introduction to the groups of interest in Section 6, the $a$-relations adjoined to $M(\textbf{M})$ and $G(\textbf{M})$ to form $M_a(\textbf{M})$ and $G_a(\textbf{M})$ are all relations of the form $w=1$ where, for $\pazocal{A}$ the tape alphabet of the `special' input sector, $w\in F(\pazocal{A})$ is trivial over the free Burnside group $B(2,n)$ with basis $\pazocal{A}$.

The following Lemma sheds some light on why these particular relations are adjoined to the group presentation.

\begin{lemma} \label{a-relations are relations}

For any word $u\in F(\pazocal{A})$, the relation $u^n=1$ holds in the group $G(\textbf{M})$.

\end{lemma}

\begin{proof}

Lemmas \ref{M language} and \ref{disks are relations} imply that the words corresponding to the configurations $I(u^n)$ and $J(u^n)$ are trivial over the group $G(\textbf{M})$. These two words differ only by the insertion of the word $u^n$ in the `special' input sector, so that $u^n=1$ over $G(\textbf{M})$.

\end{proof}

\begin{lemma} \label{G isomorphic to G_a}

The groups $G(\textbf{M})$ and $G_a(\textbf{M})$ are isomorphic.

\end{lemma}

\begin{proof}

Identify $B(2,n)$ with the presentation $\gen{\pazocal{A}\mid w=1, w\in\pazocal{L}}$.

Then consider the map $\varphi:\pazocal{A}\to G(\textbf{M})$ sending each letter to its natural copy in the tape alphabet of the `special' input sector. By the theorem of von Dyck (Theorem 4.5 of [17]), Lemma \ref{a-relations are relations} implies that $\varphi$ extends to a homomorphism $B(2,n)\to G(\textbf{M})$. This shows that for any word $w$ corresponding to an $a$-relation $w=1$, the relation $w=1$ holds in $G(\textbf{M})$. 

The theorem of von Dyck then implies that the map sending the generators of the canonical presentation of $G(\textbf{M})$ to the generators of the disk presentation of $G_a(\textbf{M})$ extends to an isomorphism between the two groups.

\end{proof}

\smallskip

\subsection{Minimal diagrams over $G_a(\textbf{M})$} \

Now we wish to justify our assignment of weights to $a$-cells and disks over the disk presentation of $G_a(\textbf{M})$. To do so, we first study areas of a diagram over the canonical presentation of $G(\textbf{M})$ with contour label corresponding to a disk relation.

\begin{lemma} \label{disks are quadratic}

(1) For any configuration $W$ accepted by $\textbf{M}$, there exists a diagram $\Delta$ over the canonical presentation of $G(\textbf{M})$ such that $\lab(\partial\Delta)\equiv W$ and $\text{Area}(\Delta)\leq c_7|W|^2$.

(2) For any nontrivial $u^n\in\pazocal{L}$, there exists a diagram $\Delta$ over the canonical presentation of $G(\textbf{M})$ with $\lab(\partial\Delta)\equiv u^n$ and $\text{Area}(\Delta)\leq c_7\|u\|^2$.

\end{lemma}

\begin{proof}

(1) By Lemmas \ref{M accepted configurations} and \ref{computations are trapezia}, one can build a diagram $\Delta$ over the canonical presentation of $G(\textbf{M})$ with $\lab(\partial\Delta)\equiv W$ made of one hub and $L$ trapezia $\Gamma_1,\dots,\Gamma_L$ satisfying $\text{Area}(\Gamma_i)\leq c_4^2\|W(i)\|^2$. The inequality follows as we choose $c_7$ after $c_4$, $L$, and $\delta$.

(2) As in (1), we can build diagrams $\Delta_1$ and $\Delta_2$ over the canonical presentations of $G(\textbf{M})$ where $\Delta_j$ is made of one disk and $L$ trapezia $\Gamma_{1,j},\dots,\Gamma_{L,j}$ satisfying: 
\begin{addmargin}[1em]{0em}
$\bullet$ $\text{Lab}(\Delta_1)\equiv I(u^n)$ and $\text{Area}(\Gamma_{i,1})\leq c_4^2\|I(u^n,i)\|^2$ for each $i$

$\bullet$ $\text{Lab}(\Delta_2)\equiv J(u^n)$ and $\text{Area}(\Gamma_{i,2})\leq c_4^2\|J(u^n,i)\|^2$ for each $i$
\end{addmargin}

Note that $\|I(u^n,i)\|,\|J(u^n,i)\|\leq N+2n\|u\|$ for $1\leq i\leq L$. So, since $c_7$ is chosen after $c_4$, $N$, and $L$, we can assume that $\text{Area}(\Delta_j)\leq \frac{c_7}{2}\|u\|^2$ for $j=1,2$.

Gluing the $\Delta_1$ and $\Delta_2$ along their common contours then yields a diagram $\Delta$ satisfying the statement.

\end{proof}

\begin{lemma} \label{a-cells are quadratic} If $w$ is a reduced word over the alphabet $\pazocal{A}$ such that $w=1$ in $B(2,n)$, then there exists a diagram $\Delta$ over the canonical presentation of $G(\textbf{M})$ with $\lab(\partial\Delta)\equiv w$ and satisfying $\text{Area}(\Delta)\leq c_7\|w\|^2$.

\end{lemma}

\begin{proof}

Let $\Delta_0$ be a van Kampen diagram over the presentation $\gen{\pazocal{A}\mid\pazocal{R}}$ (see Section 2.8) with $\lab(\partial\Delta_0)\equiv w$. For each cell $\Pi_0$ in $\Delta_0$, $\lab(\partial\Pi_0)\in\pazocal{R}\subset\pazocal{L}$. Setting $\lab(\partial\Pi_0)\equiv (u(\Pi_0))^n$, Lemma \ref{disks are quadratic}(2) then implies that there exists a diagram $\Pi$ over the canonical presentation of $G(\textbf{M})$ satisfying $\lab(\partial\Pi)\equiv (u(\Pi_0))^n$ and $\text{Area}(\Pi)\leq c_7\|u(\Pi_0)\|^2$.

Pasting $\Pi$ in place of $\Pi_0$ for each cell of $\Delta_0$ then produces a van Kampen diagram $\Delta$ over the canonical presentation of $G(\textbf{M})$ satsifying $\lab(\partial\Delta)\equiv w$ and $$\text{Area}(\Delta)=\sum\text{Area}(\Pi)\leq\sum\limits_{\Pi_0\in\Delta_0} c_7\|u(\Pi_0)\|^2$$

But defining $\rho(\Pi_0)=\|u(\Pi_0)\|^2$ as in the definition of mass in Section 2.9, it follows from Lemma \ref{a-cells are quadratic B(m,n)} that $$\sum\limits_{\Pi_0\in\Delta_0}\|u(\Pi_0)\|^2=\sum\limits_{\Pi_0\in\Delta_0}\rho(\Pi_0)\defeq\rho(\Delta_0)\leq\|\partial\Delta_0\|^2$$
Hence, $\text{Area}(\Delta)\leq c_7\|\partial\Delta_0\|^2\leq c_7\|w\|^2$.

\end{proof}

\medskip


\section{Diagrams without disks}

\subsection{Annuli} \

In this section, we bound the weight of a minimal diagram over $M_a(\textbf{M})$ in terms of its perimeter. To do this, we first dismiss the possibility of certain subdiagrams in minimal diagrams, yielding an analogue of Lemma \ref{M(S) annuli}.

\begin{lemma} \label{M_a no annuli}

If $\Delta$ is a minimal diagram over $M_a(\textbf{M})$, then it has no:
\begin{addmargin}[1em]{0em}

(1) $a$-annuli

(2) $q$-annuli

(3) $\theta$-annuli

(4) $(\theta,q)$-annuli

(5) $(\theta,a)$-annuli

\end{addmargin}

\end{lemma}

\begin{proof}

The proof follows simultaneous induction on the minimalilty (with respect to signature) of a subdiagram $\Delta_S$ containing a conjectural counterexample, i.e an annulus $S$. So, the contour of $\Delta_S$ is one component of the contour of $S$ and $\Delta$ contains no annulus $S'$ of the above types so that the minimal subdiagram containing $S'$, $\Delta_{S'}$, satisfies $s(\Delta_{S'})<s(\Delta_S)$.

(1) Assume $S$ is an $a$-annulus. Then every cell of $S$ is a $(\theta,a)$-cell, so that any one, say $\pi_1$, has contour sharing a $\theta$-edge with $\partial\Delta_S$. The maximal $\theta$-band $\pazocal{T}$ starting at this edge then must end on $\partial\Delta_S$, i.e on a $\theta$-edge of another cell $\pi_2$ of $S$. 

Suppose $\pi_1$ and $\pi_2$ are adjacent. Then the $a$-letters corresponding to the $(\theta,a)$-relations of $\pi_1$ and $\pi_2$ must be the same, while the $\theta$-letters must be mutually inverse. So, these cells are cancellable, contradicting the assumption that $\Delta$ is minimal. 

Assuming $\pi_1$ and $\pi_2$ are not adjacent, let $\pi$ be a cell of $S$ between $\pi_1$ and $\pi_2$. But then letting $S'$ be the $(\theta,a)$-annulus bounded by $\pazocal{T}$ and the part of $S$ between $\pi_1$ and $\pi_2$ not containing $\pi$, $s(\Delta_{S'})<s(\Delta_S)$ contradicting the minimality of $\Delta_S$.

(2) By an identical argument, $S$ being a $q$-annulus leads to a contradiction since every cell of the annulus must be a $(\theta,q)$-cell, so that either $S$ contains a pair of cancellable cells or there exists a $(\theta,q)$-annulus $S'$ satisfying $s(\Delta_{S'})<s(\Delta_S)$.

(3) Assume $S$ is a $\theta$-annulus. 

If $S$ contains a $(\theta,q)$-cell, then there exists a $q$-edge on $\partial\Delta_S$ marking the start of a $q$-band in $\Delta_S$. However, since there are no disks, this band must end on $\partial\Delta_S$, creating a $(\theta,q)$-annulus. Similar to the argument in (1), this either produces cancellable $(\theta,q)$-cells in $\Delta_S$ or a $(\theta,q)$-annulus $S'$ with $s(\Delta_{S'})<s(\Delta_S)$. Either way, this contradicts the minimality of $\Delta_S$.

So, assume that every cell of $S$ is a $(\theta,a)$-cell, so that the label of the boundary of $\Delta_S$ is a word in $F(Y)$. If any maximal $a$-band has both ends on $\partial\Delta_S$, then as above it will either yield a pair of cancellable cells in $S$ or a $(\theta,a)$-annulus to contradict (5). 

The existence of any $(\theta,q)$- or $(\theta,a)$-cell in $\Delta_S$ not on $S$ would give rise to a maximal $\theta$-band that cannot intersect $S$, i.e a $\theta$-annulus $S'$. However, this would then yield a subdiagram $\Delta_{S'}$ of $\Delta_S$ not containing the cells of $S$, contradicting the minimality of $\Delta_S$. 

As a result, the subdiagram $\Delta_S'=\Delta_S\setminus S$ must consist only of $a$-cells. This means that the label of the contour of $\Delta_S'$ is trivial over $B(2,n)$, so that it corresponds to an $a$-relation. So, as $\Delta$ is minimal, $\Delta_S$ consists only of $S$ and an $a$-cell whose contour label is visually equal to the contour label of $\Delta_S$. But then letting $\Gamma$ be the diagram formed from $\Delta$ by excising $\Delta_S$ and pasting in the subdiagram consisting only of this $a$-cell, $s(\Gamma)<s(\Delta)$ and $\text{Lab}(\Gamma)\equiv\text{Lab}(\Delta)$, contradicting the minimality of $\Delta$.

(4) If $S$ is a $(\theta,q)$-annulus, then by Lemma \ref{annuli lower bound} the defining rim $q$-band $\pazocal{Q}$ contains a $(\theta,q)$-cell with neither $q$-edge on $\partial\Delta_S$. The maximal $\theta$-band in $\Delta_S$ starting at this cell, $\pazocal{T}$, then must end on $\partial\Delta_S$, so that it must intersect another $(\theta,q)$-cell of $\pazocal{T}$. The $(\theta,q)$-annulus bounded by $\pazocal{T}$ and the subband of $\pazocal{Q}$ bounded by the first and last cells of $\pazocal{T}$ then contradicts the minimality of $\Delta_S$.

(5) is proved by an identical argument to (4).

\end{proof}

Thus, every maximal $\theta$-band in a diagram over the canonical presentation of $M_a(\textbf{M})$ can cross any maximal $q$-band (maximal $a$-band) in at most one cell and must have two ends on the diagram's contour. Similarly, every maximal $q$-band must have two ends on the diagram's contour.

\begin{lemma} \label{a-bands between a-cells}

Let $\Delta$ be a minimal diagram over $G_a(\textbf{M})$. Then no $a$-band can have two ends on $a$-cells.

\end{lemma}

\begin{proof}

\underline{Case 1:} First suppose $\Delta$ contains two different $a$-cells, $\pi_1$ and $\pi_2$, connected by an $a$-band. Let $\Delta_0$ be the subdiagram bounded by $\pi_1$, $\pi_2$, and this $a$-band. As an $a$-band consists only of $(\theta,a)$-cells, the top and bottom of the $a$-band have equivalent labels visually equal to a word $H\in F(R)$.

So, $\lab(\partial\Delta_0)\equiv uHvH^{-1}$ for some words $u,v\in F(Y)$. Note that for any rule $\theta$ corresponding to a letter of $H$, the makeup of the $a$-band implies the existence of a $(\theta,a)$-relation corresponding to $\theta$ and an $a$-letter from the `special' input sector. This then implies that the domain of $\theta$ in the `special' input sector is nonempty, which in turn implies that the domain of $\theta$ in this sector is the entire alphabet.

As a result, we can build an annular diagram $\Gamma'$ over the canonical presentation of $M(\textbf{M})$ with outer label $uHvH^{-1}$, inner label $uv$, and entirely made up of $(\theta,a)$-cells (and 0-cells). Then, since $\Delta_0$ is a diagram over $M_a(\textbf{M})$, we have $uv=1$ in $M_a(\textbf{M})$. 

Letting $\Psi$ be a minimal diagram over $M_a(\textbf{M})$ satisfying $\lab(\partial\Gamma)\equiv uv$, Lemma \ref{M_a no annuli} implies that $\Psi$ consists only of $a$-cells. But then $uv=1$ over $B(2,n)$, so that $\Psi$ consists of only one $a$-cell. Pasting $\Psi$ into the middle of $\Gamma'$ then yields a disk diagram $\Gamma_0$ over $M_a(\textbf{M})$ with contour label $uHvH^{-1}$.

Letting $\Gamma$ be the diagram obtained from $\Delta$ by excising $\Delta_0$ and pasting in $\Gamma_0$ yields a diagram with the same contour label, one less $a$-cell, and the same number of disks, $(\theta,t)$-cells, and $(\theta,q)$-cells. This contradicts the assumption that $\Delta$ is a minimal diagram.

\underline{Case 2:} Now suppose $\Delta$ has an $a$-band $T$ ending twice on the same $a$-cell $\pi$.

Consider the subdiagram $\Delta_0$ of $\Delta$ bounded by $T$ and the portion of $\partial\pi$ between the two ends of $T$.  Since each cell of $T$ is a $(\theta,a)$-cell, the portion of $\partial\Delta_0$ coinciding with the top of $T$ is comprised entirely of $\theta$-edges. As the edges of $\partial\pi$ are all $a$-edges, it follows that all $\theta$-edges of $\partial\Delta_0$ lie on the side of $T$.

As $a$-bands cannot have two ends on distinct $a$-cells by Case 1, it follows that every maximal $a$-band of $\Delta_0$ starting on $\partial\pi$ must also end on $\partial\pi$. Passing to further subdiagrams bounded by these $a$-bands, it follows that there must be an $a$-band starting and ending on adjacent edges of $\partial\pi$. But this implies that $\Delta$ is not reduced, and so not minimal.

\end{proof}

\smallskip


\subsection{Combs and Subcombs} \

If $\Delta$ is a minimal diagram over $M_a(\textbf{M})$ containing a maximal $q$-band $\pazocal{Q}$ such that $\textbf{bot}(\pazocal{Q})$ is a subpath of $\partial\Delta$ and every maximal $\theta$-band of $\Delta$ ends at an edge of $\textbf{bot}(\pazocal{Q})$, then $\Delta$ is called a \textit{comb} and $\pazocal{Q}$ its \textit{handle}.

The number of cells in the handle of $\pazocal{Q}$ is the comb's \textit{height} and the maximal length of the bases of the $\theta$-bands is its \textit{basic width}.

Note that every trapezium is a comb.

%
%
%
%

\smallskip

\begin{lemma} \label{one q-band comb}

Suppose $\Gamma$ is a comb of height $h$ containing no maximal $q$-bands other than its handle. Then $\text{wt}(\Gamma)\leq c_8|\partial\Gamma|^2$.

\end{lemma}

\begin{proof}

Note that each $(\theta,q)$-cell on the handle shares at least one $\theta$-edge with $\partial\Gamma$. So, $h\leq\|\partial\Gamma\|$.

By Lemma \ref{a-bands between a-cells}, every maximal $a$-band with one end on an $a$-cell has another end on either the contour of the diagram or on a $(\theta,q)$-cell. So, since the only $(\theta,q)$-cells are a part of the handle $\pazocal{Q}$, every $a$-band starting on an $a$-cell must end on either $\pazocal{Q}$ or $\partial\Gamma$, so that there are at most $2h+\|\partial\Gamma\|\leq3\|\partial\Gamma\|$ such $a$-bands. This means that the sum of the (combinatorial) perimeters of all $a$-cells in $\Gamma$ is at most $3\|\partial\Gamma\|$, meaning the sum of their weights is at most $9c_7\|\partial\Gamma\|^2$.

Meanwhile, every cell off of the handle that is not an $a$-cell is a $(\theta,a)$-cell which is a part of a maximal $a$-band. Since these $a$-bands must have at least one end on either $\partial\Gamma$ or on the handle $\pazocal{Q}$, there are at most $\|\partial\Gamma\|+2h\leq3\|\partial\Gamma\|$ maximal $a$-bands. By Lemma \ref{M_a no annuli}, each of these $a$-bands shares at most one cell with any of the $h$ maximal $\theta$-bands, so that its length is at most $h$. So, there are at most $3h\|\partial\Gamma\|\leq3\|\partial\Gamma\|^2$ $(\theta,a)$-cells in $\Gamma$.

Adding in the $h\leq\|\partial\Gamma\|$ cells of $\pazocal{Q}$ and taking into account the definition of length, the statement follows from $c_8>>c_7>>\delta^{-1}$.

\end{proof}

If $\Delta$ is a minimal diagram over $M_a(\textbf{M})$, then a subdiagram $\Gamma$ is a \textit{subcomb} of $\Delta$ if $\Gamma$ is a comb and its handle divides $\Delta$ into two parts, one of which is $\Gamma$.

\smallskip


\subsection{Transpositions of a $\theta$-band with an $a$-cell} \

Let $\Delta$ be a minimal diagram over $G_a(\textbf{M})$ containing an $a$-cell $\pi$ and a $\theta$-band $\pazocal{T}$ subsequently crossing some of the $a$-bands starting at $\pi$. As the cells shared by these bands and  $\pazocal{T}$ are $(\theta,a)$-cells, the rule $\theta$ corresponding to $\pazocal{T}$ contains at least one letter of $\pazocal{A}$ in its domain; so, by the definition of the rules of $\textbf{M}$, $\theta$ contains all letters of $\pazocal{A}$ in its domain.

Suppose there are no other cells between $\pi$ and the bottom of $\pazocal{T}$, i.e there is a subdiagram formed by $\pi$ and $\pazocal{T}$.

Let $\textbf{s}_1$ be the subpath of $\partial\pi$ that coincides with the bottom of $\pazocal{T}$, $\textbf{s}_2$ its complement so that $\partial\pi=\textbf{s}_1\textbf{s}_2$, and $\pazocal{T}'$ the subband of $\pazocal{T}$ satisfying $\textbf{bot}(\pazocal{T}')=\textbf{s}_1$. Let $V_1\equiv\text{Lab}(\textbf{s}_1)$ and $V_2\equiv\text{Lab}(\textbf{s}_2)$. Further, let $\Gamma$ be the subdiagram formed by $\pi$ and $\pazocal{T}'$.

Then, we can construct the $\theta$-band $\pazocal{S}$ corresponding to $\theta$ with top labelled by $V_2^{-1}$. As we have $\text{Lab}(\textbf{top}(\pazocal{T}'))\equiv V_1$, we can replace the subdiagram formed by $\pazocal{T}'$ and $\pi$ with this band pasted to the subpath $\textbf{s}_2$ of $\pi$, attaching the first and last cells of $\pazocal{S}$ to the complement of $\pazocal{T}'$ in $\pazocal{T}$ and making any necessary cancellations in the resulting band. Letting $\pazocal{T}''$ be this resulting band, we have replaced $\Gamma$ with a new copy of $\pi$ and the band $\pazocal{T}''$.

This process is called the \textit{transposition} of the $\theta$-band with the $a$-cell. 

It is a useful observation that, should $\|\textbf{s}_1\|>\frac{1}{2}\|\partial\pi\|$, then the transposition of $\pazocal{T}$ and $\pi$ would result in a diagram with the same contour label and strictly smaller signature, meaning $\Delta$ cannot be minimal.

Conversely, the diagram resulting from a transposition need not be minimal, as there may be an increase in the number of $(\theta,a)$-cells in $\Delta$. Nonetheless, this process will prove useful even in this case.

\smallskip


\subsection{Cubic upper bound on weights} \

\begin{lemma} \label{no q-edge quadratic}

If $\Delta$ is a minimal diagram over $M_a(\textbf{M})$ with no $q$-edges on its boundary, then $\text{wt}(\Delta)\leq c_8|\partial\Delta|^2$.

\end{lemma}

\begin{proof}

Since any $q$-edge in $\Delta$ would give rise to a maximal $q$-band which, by Lemma \ref{M_a no annuli}, can only end on the contour of the diagram, $\Delta$ cannot have any $q$-edges. So, $\Delta$ is comprised entirely of $(\theta,a)$-cells and $a$-cells. 

Lemmas \ref{a-bands between a-cells} and \ref{M_a no annuli} then imply that any maximal $a$-band with one end on an $a$-cell must have its other end on the boundary, so that the sum of the (combinatorial) perimeters of the $a$-cells is at most $\|\partial\Delta\|$. This means that the sum of the weights of the $a$-cells is at most $c_7\|\partial\Delta\|^2$. 

As any maximal $\theta$-band must start and end on $\partial\Delta$, there are at most $\frac{1}{2}\|\partial\Delta\|$ maximal $\theta$-bands in $\Delta$. Further, as there are at most $\|\partial\Delta\|$ maximal $a$-bands and each $\theta$-band intersects each $a$-band in at most one cell, the length of each $\theta$-band is at most $\|\partial\Delta\|$. So, the sum of the area of all maximal $\theta$-bands, and so the number of $(\theta,a)$-cells, is at most $\frac{1}{2}\|\partial\Delta\|^2$.

So, taking into account the weighting, the statement follows from an appropriate choice of $c_8$ in terms of $c_7$ and $\delta$.

\end{proof}

\begin{lemma} \label{M_a cubic}

If $\Delta$ is a minimal diagram over $M_a(\textbf{M})$, then $\text{wt}(\Delta)\leq N_1|\partial\Delta|^3$ for the parameter $N_1$ specified in Section 3.3.

\end{lemma}

\begin{proof}

By way of contradiction, assume $\Delta$ is a counter-example diagram of minimal perimeter, i.e $\text{wt}(\Delta)>N_1|\partial\Delta|^3$ and for all minimal diagrams $\Gamma$ over $M_a(\textbf{M})$ with $|\partial\Gamma|<|\partial\Delta|$, the inequality $\text{wt}(\Gamma)\leq N_1|\partial\Gamma|^3$ holds.

By Lemma \ref{no q-edge quadratic}, $\partial\Delta$ must contain a $q$-edge. This $q$-edge corresponds to a maximal $q$-band, which must have both ends on $\partial\Delta$.

This $q$-band breaks $\Delta$ into two subdiagrams that, as two maximal $q$-bands can not intersect, have disjoint sets of maximal $q$-bands. Passing to one of these and iterating, there exists a $q$-band $\pazocal{Q}$ separating the diagram $\Delta$ into two subdiagrams, one of which, $\Gamma$, has no $q$-bands other than $\pazocal{Q}$. 

\underline{Case 1.} Suppose $\Gamma$ is a subcomb of $\Delta$ with handle $\pazocal{Q}$. 

Let $\Delta_0$ be the diagram formed by excising $\Gamma$ from $\Delta$, i.e cutting along $\textbf{t}=\textbf{bot}(\pazocal{Q})$. Set $\partial\Delta=\textbf{s}_0\textbf{s}_1$ where $\textbf{s}_0$ is a subpath of $\partial\Delta_0$ and $\textbf{s}_1$ is a subpath of $\partial\Gamma$. Then, $\partial\Delta_0=\textbf{s}_0\textbf{t}^{-1}$ and $\partial\Gamma=\textbf{t}\textbf{s}_1$.

Letting $h$ be the height of $\Gamma$, $|\textbf{t}|=h$ by Lemma \ref{lengths}$(b)$. Since every $\theta$-band intersecting $\pazocal{Q}$ must have one end on $\partial\Gamma$ and, more specifically, on $\textbf{s}_1$, $|\textbf{s}_1|\geq h+2$.

So, $|\partial\Delta_0|\leq|\textbf{s}_0|+|\textbf{t}|=|\textbf{s}_0|+h\leq|\textbf{s}_0|+|\textbf{s}_1|-2\leq|\partial\Delta|+2\delta-2\leq|\partial\Delta|-1$ for sufficiently small $\delta$. By the inductive hypothesis, this yields $\text{wt}(\Delta_0)\leq N_1(|\partial\Delta|-1)^3$.

By Lemma \ref{one q-band comb}, it then follows that $\text{wt}(\Delta)\leq N_1(|\partial\Delta|-1)^3+c_8|\partial\Gamma|^2$. 

Each $\theta$-edge of $\textbf{t}$ marks the start of a maximal $\theta$-band of $\Delta_0$. By Lemma \ref{M_a no annuli}, this $\theta$-band must end on $\textbf{s}_0$, so that $|\textbf{s}_0|\geq h=|\textbf{t}|$. So, applying Lemma \ref{lengths}, 
$$|\partial\Gamma|\leq|\textbf{t}|+|\textbf{s}_1|\leq|\textbf{s}_0|+|\textbf{s}_1|\leq|\partial\Delta|+2\delta\leq(1+\delta)|\partial\Delta|$$

So, $\text{wt}(\Delta)\leq N_1(|\partial\Delta|-1)^3+c_8(1+\delta)^2|\partial\Delta|^2$.

Taking $N_1>c_8(1+\delta)^2$ then gives 
$$\text{wt}(\Delta)\leq N_1(|\partial\Delta|-1)^3+N_1|\partial\Delta|^2\leq N_1|\partial\Delta|^3-2N_1|\partial\Delta|^2+3N_1|\partial\Delta|-1$$
But since $|\partial\Delta|\geq2$, we have $|\partial\Delta|^2\geq2|\partial\Delta|$, giving the desired contradiction $$\text{wt}(\Delta)\leq N_1|\partial\Delta|^3$$

\underline{Case 2.} Suppose $\Gamma$ is not a subcomb of $\Delta$. 

This implies that there exists a $\theta$-band $\pazocal{T}$ in $\Gamma$ that does not end on $\pazocal{Q}$. Lemma \ref{M_a no annuli} then implies that $\pazocal{T}$ has both ends on the shared contour of $\Delta$ and $\Gamma$.

Passing to further $\theta$-bands, we can assume that the subdiagram $\Gamma_0$ of $\Gamma$ bounded by $\pazocal{T}$ and $\partial\Delta$ contains no other $\theta$-bands. Lemma \ref{M_a no annuli} then implies that $\Gamma_0\setminus\pazocal{T}$ consists only of $a$-cells (or is empty). Set $\Delta'=\Delta\setminus\Gamma_0$.

If $\Gamma_0\setminus\pazocal{T}$ is empty, then $\pazocal{T}$ is a rim $\theta$-band with base of length zero, so that Lemma \ref{small rim bands} implies that $|\partial\Delta|-|\partial\Delta'|>1$. Further, since $|\textbf{top}(\pazocal{T})|=|\textbf{bot}(\pazocal{T})|$, we have $|\partial\Gamma_0|\leq2|\partial\Delta|$.

Otherwise, let $\pi$ be an $a$-cell of $\Gamma_0$. Then Lemma \ref{a-bands between a-cells} implies that each edge of $\pi$ is shared with either $\partial\Delta$ or the side of $\pazocal{T}$. If more than half these edges are shared with the top (or bottom) of $\pazocal{T}$, though, then the transposition of $\pazocal{T}$ with $\pi$ would produce a diagram with the same contour label as $\Delta$ and strictly smaller signature, contradicting the minimality of $\Delta$.

It then follows that the number of $a$-edges on the shared contour of $\Delta$ and $\Gamma_0$ is at least as large as the number on the top (bottom) of $\pazocal{T}$. It follows immediately that $|\partial\Gamma_0|\leq2|\partial\Delta|$, while the argument in Case 1 implies again that $|\partial\Delta|-|\partial\Delta'|>1$.

Lemma \ref{no q-edge quadratic} then implies that $\text{wt}(\Gamma_0)\leq c_8|\partial\Gamma_0|^2\leq4c_8|\partial\Delta|^2$, while the inductive hypothesis implies that $\text{wt}(\Delta')\leq N_1(|\partial\Delta|-1)^3$. Combining these gives $\text{wt}(\Delta)\leq N_1(|\partial\Delta|-1)^3+4c_8|\partial\Delta|^2$. 

Noting that $|\partial\Delta|\geq4$ and $N_1>>c_8$, we then reach the desired contradiction $\text{wt}(\Delta)\leq N_1|\partial\Delta|^3$.

\end{proof}

\medskip


\section{Diagrams with disks}

\subsection{$t$-spokes} \

When considering minimal diagrams over $G_a(\textbf{M})$ in what follows, many arguments rely on the $q$-bands corresponding to the one-letter parts $\{t(i)\}$ of the standard base, i.e to $t$-letters. To distinguish these from bands corresponding to other parts of the base, we adopt the convention of [19] and [26] and refer to them as \textit{$t$-bands}. Note the very simple makeup of these bands: Each cell has contour label $\theta t(i)\theta^{-1}t(i)^{-1}$ for some $\theta\in\Theta$ (with subscripts of $\theta$ suppressed; see Section 6.1), so that the top and the bottom of the band are labelled by a copy of the band's history.

In a minimal diagram over $G_a(\textbf{M})$, a maximal $q$-band with one end on a disk $\Pi$ is called a \textit{spoke} of $\Pi$. A \textit{$t$-spoke} is then defined in the natural way.

The pairs $(t(1),t(2)),\dots,(t(L-1),t(L)),(t(L),t(1))$ are called \textit{adjacent} $t$-letters. In a minimal diagram over $G_a(\textbf{M})$, two $t$-spokes of the same disk are called \textit{consecutive} if they correspond to adjacent $t$-letters.

\begin{lemma} \label{extend}

Let $\pazocal{C}:A(i,m)\to\dots\to A(i,m)$ be a reduced computation of $\textbf{M}$ with history $H$ and $H(0)$ be the copy of $H$ in $F(R)$ obtained by adding the subscript 0 to each letter. Then there exists a minimal diagram $\Delta$ over $M_a(\textbf{M})$ with contour label $H(0)^{-1}W_{ac}H(0)W_{ac}^{-1}$.

\end{lemma}

\begin{proof}

Consider the factorization $H\equiv H_1\cdots H_l$ for $l\geq1$ given by Lemma \ref{extending computations}.

Define $H_i(0)$ as the word in $F(R)$ obtained from $H_i$ by adding a subscript $0$ to each letter. By Lemma \ref{computations are trapezia}, for each $1\leq j\leq l$, there exists a trapezium $\Delta_j$ with contour label $$H_j(0)^{-1}W_{j-1}'(z_j)H_j(0)(W_j'(z_j))^{-1}$$ where $W_j'(z_j)$ is defined as in Lemma \ref{extending computations}.

Recall that for $1\leq j\leq l-1$, $W_j'(z_j)$ differs from $W_j'(z_{j+1})$ only by the insertion/deletion of elements of $\pazocal{L}$, while $W_0'(z_1)\equiv W_l'(z_l)\equiv W_{ac}$. So, after gluing $a$-cells corresponding to these elements of $\pazocal{L}$ to the top of $\Delta_j$, the top of $\Delta_j$ can be glued to the bottom of $\Delta_{j+1}$.

Letting $\Delta$ be the diagram that results from pasting together $\Delta_1,\dots,\Delta_l$, it then follows that $\text{Lab}(\partial\Delta)\equiv H(0)^{-1}W_{ac}H(0)W_{ac}^{-1}$.

\end{proof}

\begin{lemma} \label{t-spokes between disks}

Suppose $\Delta$ is a minimal diagram over $G_a(\textbf{M})$ with consecutive $t$-spokes of the same two disks. Then the subdiagram $\Psi$ bounded by these two $t$-bands and the contours of the disks (and not containing the disks) must contain a disk.

\end{lemma}

\begin{proof}

Assume that the two disks are hubs. Then, note that if either of these two bands has zero length, then the two hubs are cancellable, meaning $\Delta$ is not reduced (and so not minimal).

Arguing toward a contradiction, suppose $\Psi$ is a diagram over $M_a(\textbf{M})$. By Lemma \ref{M_a no annuli}, it then follows that the two $t$-bands have the same history, say $H$. Then, one can remove one of the two $t$-bands so that the label of the shared contour of the resulting diagram and one of the hubs is $A(i)$ for some $i$. Further, restrict to the subdiagram $\Psi_0$ that is bounded by the $q$-bands corresponding to the first and last letters of $A(i,m)$.

By Lemma \ref{a-bands between a-cells}, there can be no $a$-cells in $\Psi_0$ as one would necessitate the existence of an $a$-edge on $\partial\Psi_0$. So, $\Psi_0$ is a trapezium. Lemma \ref{trapezia are computations} then says there exists a computation $\pazocal{C}:A(i,m)\to\dots\to A(i,m)$ with history $H$. So, we can apply Lemma \ref{extend} to find a minimal diagram $\Gamma_1$ over $M_a(\textbf{M})$ with contour label $H(0)^{-1}W_{ac}H(0)W_{ac}^{-1}$.

The subdiagram $\Gamma_0$ of $\Gamma_1$ bounded by the two $t$-bands corresponding to the letters of the original $t$-spokes has the same contour label as $\Psi$. Since $\textbf{M}$ is a cyclic $S$-machine, we can assume that $\Gamma_0$ can be cut from $\Gamma_1$ to produce a disk diagram $\Gamma$ over $M_a(\textbf{M})$.

Letting $\Psi'$ be the subdiagram of $\Delta$ bounded by the two hubs and the two $t$-spokes (and containing the hubs), though, $\Psi'$ has the same contour label as $\Gamma$. Hence, excising $\Psi'$ from $\Delta$ and pasting in $\Gamma$ reduces the number of hubs (and so disks) by two, contradicting the minimality of $\Delta$.

Conversely, if the two disks are not hubs, then we can replace them with hubs with $L$ trapezia glued to them (see Lemma \ref{disks are relations}) and apply the procedure above to reduce the number of disks.

\end{proof}

For each minimal diagram $\Delta$ over $G_a(\textbf{M})$, there corresponds a planar graph $\Gamma\equiv\Gamma(\Delta)$ defined by:

\begin{addmargin}[1em]{0em}

(1) $V(\Gamma)=\{v_0,v_1,\dots,v_l\}$ where each $v_i$ for $i\geq1$ corresponds to one of the $l$ disks of $\Delta$ and $v_0$ is one exterior vertex

(2) For $i,j\geq1$, each shared $t$-spoke of the disks corresponding to $v_i$ and $v_j$ corresponds to an edge $(v_i,v_j)\in E(\Gamma)$

(3) For $i\geq1$, each $t$-spoke of the disk corresponding to $v_i$ which ends on $\partial\Delta$ corresponds to an edge $(v_0,v_i)\in E(\Gamma)$

\end{addmargin}

The following lemma follows from Lemma \ref{t-spokes between disks} and a graph theoretical argument applied to the corresponding graph. A detailed proof can be found in [18].

\begin{lemma} \label{graph}

If $\Delta$ is a minimal diagram over $G_a(\textbf{M})$ containing at least one disk, then $\Delta$ contains a disk $\Pi$ such that $L-3$ consecutive $t$-spokes $\pazocal{Q}_1,\dots,\pazocal{Q}_{L-3}$ of $\Pi$ have one end on $\partial\Delta$ and such that every subdiagram $\Gamma_i$ bounded by $\pazocal{Q}_i$, $\pazocal{Q}_{i+1}$, $\partial\Pi$, and $\partial\Delta$ ($i=1,\dots,L-4$) contains no disks.

\end{lemma}

Applying induction on the number of hubs then implies:

\begin{lemma} \label{number of q-edges}

\textit{(Lemma 5.19 of [27])} In a minimal diagram $\Delta$ over $G_a(\textbf{M})$ with $l\geq1$ disks, the number of spokes ending on $\partial\Delta$, and therefore the number of $q$-edges of $\partial\Delta$, is greater than $lLN/2$.

\end{lemma}

\smallskip


\subsection{Transposition of a $\theta$-band and a disk} \

Similar to the construction in Section 9.3, we now describe a procedure for moving a $\theta$-band about a disk.

Let $\Delta$ be a minimal diagram over $G_a(\textbf{M})$ containing a disk $\Pi$ and a $\theta$-band $\pazocal{T}$ subsequently crossing the $t$-spokes $\pazocal{Q}_1,\dots,\pazocal{Q}_l$ of $\Pi$. Assume $l\geq2$ is maximal for $\Pi$ and $\pazocal{T}$ and there are no other cells between $\Pi$ and the bottom of $\pazocal{T}$, i.e there is a subdiagram formed by $\Pi$ and $\pazocal{T}$.

Let $\pazocal{T}'$ be the subband of $\pazocal{T}$ whose bottom path, $\textbf{s}_1$, starts with the $t$-letter corresponding to $\pazocal{Q}_1$ and ends with that of $\pazocal{Q}_l$. Further, let $\textbf{s}_2$ be the complement of $\textbf{s}_1$ in $\partial\Pi$ so that $\partial\Pi=\textbf{s}_1\textbf{s}_2$, let $W\equiv\text{Lab}(\partial\Pi)^{\pm1}$, let $V\equiv\text{Lab}(\textbf{s}_1)$, and let $\theta$ be the rule corresponding to $\pazocal{T}$. Then, $V\cdot\theta\equiv\lab(\textbf{ttop}(\pazocal{T}'))=\lab(\textbf{top}(\pazocal{T}'))$ by Lemma \ref{theta-bands are one-rule computations}. Further, let $\Gamma$ be the subdiagram formed by $\Pi$ and $\pazocal{T}'$.

Suppose $\text{Lab}(\textbf{s}_2)$ is $\theta$-admissible. Then there exists a disk $\bar{\Pi}$ with contour labelled by $W\cdot\theta$. Construct the auxiliary $\theta$-band $\pazocal{T}''$ corresponding to $\theta$ whose top is labelled by $(\text{Lab}(\textbf{s}_2)\cdot\theta)^{-1}$ and attach it to $\bar{\Pi}$. We can then excise $\Gamma$ from $\Delta$ and paste in $\bar{\Pi}$ and $\pazocal{T}''$, attaching the first and last cells of $\pazocal{T}''$ to the complement of $\pazocal{T}'$ in $\pazocal{T}$ (and perhaps make cancellations in the resulting $\theta$-band). This results in a diagram $\bar{\Delta}$ with the same contour label as that of $\Delta$.

Conversely, suppose $\text{Lab}(\textbf{s}_2)$ is not $\theta$-admissible. Then Lemma \ref{projection admissible configuration not} applies to $W$, so that $\lab(\textbf{s}_2)$ contains the `special' input sector and would be $\theta$-admissible with the insertion/deletion of some $u^n\in\pazocal{L}$. In this case, we add an $a$-cell corresponding to $u^n$, construct the auxiliary $\theta$-band $\pazocal{T}''$ as above, and attach the mirror $a$-cell on the other side to produce a diagram $\bar{\Delta}$ with similar properties to that as above.

The procedure of excising $\Gamma$ from $\Delta$ to create $\bar{\Delta}$ is called the \textit{transposition} of the disk $\Pi$ and the band $\pazocal{T}$ in $\Delta$.

\smallskip

Now, consider the situation where there may be $a$-cells between the $\theta$-band and the disk. In particular, assume $l\geq3$ is maximal for $\Pi$ and $\pazocal{T}$ and every cell between $\Pi$ and the bottom of $\pazocal{T}$ is an $a$-cell.

Since $l\geq3$, there exists $1\leq i\leq l-1$ such that the adjacent $t$-letters corresponding to $\pazocal{Q}_i$ and $\pazocal{Q}_{i+1}$ are not $t_1$ and $t_2$, i.e the `special' input sector is not between them.

If there exists an $a$-cell in the subdiagram bounded by $\pazocal{Q}_i$, $\pazocal{Q}_{i+1}$, and $\pazocal{T}'$, then Lemmas \ref{a-bands between a-cells} and \ref{M_a no annuli} imply that it sits between two spokes of $\Pi$, say $\pazocal{B}_1$ and $\pazocal{B}_2$, and that every $a$-band ending on it gives rise to a $(\theta,a)$-cell of $\pazocal{T}'$. We can then apply the transposition of $\pazocal{T}$ with the $a$-cell to contradict the minimality of $\Delta$.

So, there is no cell between $\pazocal{T}$ and the portion of $\Pi$ between $\pazocal{Q}_i$ and $\pazocal{Q}_{i+1}$, so that $W(i)$ is $\theta$-admissible.

As above, Lemma \ref{projection admissible configuration not} then implies that, perhaps after attaching $a$-cells, we can create a new disk and auxiliary band that, perhaps after attaching more $a$-cells, functions as the transposition of $\Pi$ with $\pazocal{T}$.

\smallskip

Note that the definition of transposition above differs from that in [19] and [26] only by the presence of $a$-cells. So, observing that the number of $a$-cells comes after the number of $(\theta,t)$-cells in the signature, the following is proved in the same way as in [26].

\begin{lemma} \label{G_a theta-annuli}

\textit{(Lemma 7.7 of [26])} Let $\Delta$ be a minimal diagram over $G_a(\textbf{M})$.

\begin{addmargin}[1em]{0em}

(1) Suppose a $\theta$-band $\pazocal{T}_0$ crosses $l$ $t$-spokes starting of a disk $\Pi$ and there are no disks in the subdiagram bounded by these spokes, $\pazocal{T}_0$, and $\partial\Pi$. Then $l\leq L/2$.

(2) Suppose $\pazocal{T}$ and $\pazocal{S}$ are disjoint $\theta$-bands crossing $l$ and $l'$ $t$-spokes, respectively, of a disk $\Pi$. Suppose further that every cell between the bottom of $\pazocal{T}$ (of $\pazocal{S}$) and $\Pi$ is an $a$-cell. Further, suppose these bands correspond to the same rule $\theta$ if the history is read toward the disk. Then $l+l'\leq L/2$.

(3) $\Delta$ contains no $\theta$-annuli

(4) $\Delta$ contains no $(\theta,q)$-annuli

\end{addmargin}

\end{lemma}

%
%
%
%
%

\smallskip


\subsection{Quasi-trapezia} \

We now introduce a generalization of the concept of trapezium similar to the one used in [19] and [26]. A \textit{quasi-trapezium} is a minimal diagram defined in the same way as a trapezium (see section 6.3) except that it is permitted to contain disks and $a$-cells (the allowance of $a$-cells obviously differs from previous constructions). In other words, a quasi-trapezium is a minimal diagram over $G_a(\textbf{M})$ (as opposed to a reduced diagram over $M(\textbf{M})$) with contour of the form $\textbf{p}_1^{-1}\textbf{q}_1\textbf{p}_2\textbf{q}_2^{-1}$, where each $\textbf{p}_i$ is the side of a $q$-band and each $\textbf{q}_i$ is the maximal subpath of the side of a $\theta$-band where the subpath starts and ends with a $q$-letter.

Note that trapezia form a speecial subclass of quasi-trapezia.

The \textit{(step) history} of a quasi-trapezium is defined in the same way as for a trapezium, as are the \textit{base}, \textit{height}, and \textit{standard factorization}.

\begin{lemma} \label{M_a reduced quasi-trapezia}

Suppose $\Gamma$ is a reduced (not necessarily minimal) diagram over $M_a(\textbf{M})$ with contour $\textbf{p}_1^{-1}\textbf{q}_1\textbf{p}_2\textbf{q}_2^{-1}$ where each $\textbf{p}_j$ is the side of a $q$-band and each $\textbf{q}_j$ is the maximal subpath of the side of a $\theta$-band that starts and ends with a $q$-letter. Then there exists a minimal diagram $\Gamma'$ over $M_a(\textbf{M})$ such that:

\begin{addmargin}[1em]{0em}

(1) $\partial\Gamma'=(\textbf{p}_1')^{-1}\textbf{q}_1'\textbf{p}_2'(\textbf{q}_2')^{-1}$, where $\text{Lab}(\textbf{p}_j')\equiv\text{Lab}(\textbf{p}_j)$ and $\text{Lab}(\textbf{q}_j')\equiv\text{Lab}(\textbf{q}_j)$ for $j=1,2$

(2) the vertices $(\textbf{p}_1')_-$ and $(\textbf{p}_2')_-$ (the vertices $(\textbf{p}_1')_+$ and $(\textbf{p}_2')_+$) are connected by a simple path $\textbf{s}_1$ (a simple path $\textbf{s}_2$) such that there exist three subdiagrams $\Gamma_1',\Gamma_2',\Gamma_3'$ of $\Gamma'$ where $\Gamma_2'$ is a quasi-trapezium over $M_a(\textbf{M})$ with standard factorization $(\textbf{p}_1')^{-1}\textbf{s}_1\textbf{p}_2'\textbf{s}_2^{-1}$, and all cells of the subdiagrams $\Gamma_1'$ and $\Gamma_3'$ with contours $\textbf{q}_1'\textbf{s}_1^{-1}$ and $\textbf{s}_2(\textbf{q}_2')^{-1}$ are $a$-cells

\end{addmargin}

\end{lemma}

\begin{proof}

Let $\Gamma'$ be a minimal diagram with the same contour label as $\Gamma$. Then, the contour of $\Gamma'$ can be factored as in (1).

Every $q$-edge of $\textbf{q}_j'$ gives rise to a maximal $q$-band of $\Gamma'$. Suppose such a band $\pazocal{Q}$ starts and ends on $\textbf{q}_j'$ and consider the subdiagram $\Delta$ bounded by the top of $\pazocal{Q}$ and $\textbf{q}_j'$. Since $q$-bands are comprised entirely of $(\theta,q)$-cells, the top of $\pazocal{Q}$ contains $\theta$-edges that give rise to maximal $\theta$-bands in $\Delta$. Lemma \ref{M_a no annuli} then implies that no such $\theta$-band can have both ends on the top of $\pazocal{Q}$, so that it must end on $\textbf{q}_j'$. But then $\textbf{q}_j'$ contains a $\theta$-edge, which is a contradiction as it has the same label as $\textbf{q}_j$. So, by Lemma \ref{M_a no annuli}, every maximal $q$-band of $\Gamma'$ must connect an edge of $\textbf{q}_1'$ with an edge of $\textbf{q}_2'$. 

Now suppose a maximal $\theta$-band of $\Gamma'$ has two ends on $\textbf{p}_j'$. Then, as no two $\theta$-bands can cross, there exists a $\theta$-band connecting adjacent $\theta$-letters (with perhaps $a$-letters between them) of $\textbf{p}_j'$. But then this implies that the corresponding $\theta$-edges of $\textbf{p}_j$ in $\Gamma$ are mutually inverse adjacent $\theta$-edges, so that the corresponding cells of the $q$-band with side $\textbf{p}_j$ has cancellable cells, contradicting the assumption that $\Gamma$ is reduced. So, by Lemma \ref{M_a no annuli}, every maximal $\theta$-band of $\Gamma'$ must connect an edge of $\textbf{p}_1'$ with an edge of $\textbf{p}_2'$, and so we can enumerate them from bottom to top $\pazocal{T}_1,\dots,\pazocal{T}_h$ for $h=|\textbf{p}_j|$.

Letting $\pazocal{Q}$ (respectively $\pazocal{Q}'$) be the maximal $q$-band starting on the first (respectively last) edge of $\textbf{q}_1'$, it then follows that $\text{Lab}(\textbf{top}(\pazocal{Q}))\equiv\text{Lab}(\textbf{p}_1')$ (respectively $\text{Lab}(\textbf{bot}(\pazocal{Q}'))\equiv\text{Lab}(\textbf{p}_2')$). Since these paths bound a subdiagram of the minimal diagram $\Gamma'$, they must coincide.

Now let $\textbf{s}_1=\textbf{bot}(\pazocal{T}_1)$ and $\textbf{s}_2=\textbf{top}(\pazocal{T}_h)$. Then defining the subdiagrams $\Gamma_1'$, $\Gamma_2'$, $\Gamma_3'$ as in (2), $\Gamma_1'$ and $\Gamma_3'$ have no $\theta$-edges in their contours and so must contain only $a$-cells, while $\Gamma_2'$ is a quasi-trapezium over $M_a(\textbf{M})$ by definition.

\end{proof}

\begin{lemma} \label{quasi-trapezia}

Let $\Gamma$ be a quasi-trapezium with standard factorization of its contour $\textbf{p}_1^{-1}\textbf{q}_1\textbf{p}_2\textbf{q}_2^{-1}$. Then there exists a minimal diagram $\Gamma'$ such that

\begin{addmargin}[1em]{0em}

(1) $\partial\Gamma'=(\textbf{p}_1')^{-1}\textbf{q}_1'\textbf{p}_2'(\textbf{q}_2')^{-1}$, where $\text{Lab}(\textbf{p}_j')\equiv\text{Lab}(\textbf{p}_j)$ and $\text{Lab}(\textbf{q}_j')\equiv\text{Lab}(\textbf{q}_j)$ for $j=1,2$.

(2) the number of disks and $(\theta,q)$-cells in $\Gamma'$ are the same as in $\Gamma$.

(3) the vertices $(\textbf{p}_1')_-$ and $(\textbf{p}_2')_-$ (the vertices $(\textbf{p}_1')_+$ and $(\textbf{p}_2')_+$) are connected by a simple path $\textbf{s}_1$ (a simple path $\textbf{s}_2$) such that we have the three subdiagrams $\Gamma_1$, $\Gamma_2$, $\Gamma_3$ of $\Gamma'$ where $\Gamma_2$ is a quasi-trapezium over $M_a(\textbf{M})$ with standard factorization $(\textbf{p}_1')^{-1}\textbf{s}_1\textbf{p}_2'\textbf{s}_2^{-1}$, and all cells of the subdiagrams $\Gamma_1$ and $\Gamma_3$ with boundaries $\textbf{q}_1'\textbf{s}_1^{-1}$ and $\textbf{s}_2(\textbf{q}_2')^{-1}$ are disks and $a$-cells.

(4) all maximal $\theta$-bands of $\Gamma$ and all maximal $\theta$-bands of $\Gamma_2$ have the same number of $(\theta,t)$-cells (equal for $\Gamma$ and $\Gamma_2$).

\end{addmargin}

\end{lemma}

\begin{proof}

By Lemma \ref{G_a theta-annuli}(3), every maximal $\theta$-band of $\Gamma$ must connect an edge of $\textbf{p}_1$ with an edge of $\textbf{p}_2$. So, we can enumerate these bands from bottom to top as $\pazocal{T}_1,\dots,\pazocal{T}_h$ for $h=|\textbf{p}_1|=|\textbf{p}_2|$.

If $\Gamma$ contains a disk, then by Lemma \ref{graph} there is a disk $\Pi$ such that at least $L-3$ of its $t$-spokes end on $\textbf{q}_1$ or $\textbf{q}_2$, and such that there are no disks between these spokes. By Lemma \ref{G_a theta-annuli}(1), at least $L-3-L/2\geq3$ of these spokes must end on $\textbf{q}_1$ (on $\textbf{q}_2$).

If $\Pi$ lies between $\pazocal{T}_j$ and $\pazocal{T}_{j+1}$, then the number of its $t$-spokes crossing $\pazocal{T}_j$ (crossing $\pazocal{T}_{j+1}$) is at least 3. So we can move each of these two $\theta$-bands around $\Pi$ by transposition. 


Choose $i$ such that the number of $(\theta,t)$-cells in $\pazocal{T}_i$, $m$, is minimal. It follows that $\Gamma$ has at least $hm$ $(\theta,t)$-cells.

If the disk $\Pi$ lies above $\pazocal{T}_i$ (i.e $j\geq i$), then move it upwards by transposing it with $\pazocal{T}_{j+1}$, and iterate until it is transposed with $\pazocal{T}_h$. Otherwise, move it down in the same way until it is transposed with $\pazocal{T}_1$. 

Then consider the subdiagram formed by excising $\Pi$. If it contains disks, then again find a disk with $L-3$ of its $t$-spokes ending on the top or bottom. Iterating this process, every disk can be transposed until it is above the $h$-th $\theta$-band or below first $\theta$-band. This forms a new diagram, $\Gamma''$, with subdiagram $\Gamma_2''$ bounded by $\pazocal{T}_1''$ and $\pazocal{T}_h''$ which contains no disks. Note, however, that this process may add $a$-cells to the diagram; perhaps after some cancellations, we can then assume that $\Gamma_2''$ is reduced and so satisfies the hypotheses of Lemma \ref{M_a reduced quasi-trapezia}.

%
%
%
%

Since the $\theta$-band $\pazocal{T}_i$ did not participate in any of the transpositions and the history of $\Gamma$ is reduced, the resulting $\theta$-band $\pazocal{T}_i''$ of $\Gamma_2''$ has the same number of $(\theta,t)$-cells as $\pazocal{T}_i$. Moreover, since $\Gamma_2''$ contains no disks, every maximal $q$-band of $\Gamma_2''$ must cross every maximal $\theta$-band, so that there are exactly $hm$ $(\theta,t)$-cells, which does not exceed the number of $(\theta,t)$-cells of $\Gamma$. As $\Gamma$ is minimal, these numbers are equal.

Applying Lemma \ref{M_a reduced quasi-trapezia} to $\Gamma_2''$ results in the diagram $\Gamma'$ as described in (1) with subdiagrams $\Gamma_1,\Gamma_2,\Gamma_3$ as described in (3). Statement (2) is clear from the construction and the minimality of both $\Gamma$ and $\Gamma'$. 

Finally, since $\Gamma_2$ is minimal, it contains at most as many $(\theta,t)$-cells as $\Gamma_2''$, and so the same number as $\Gamma$. As a result, each $\theta$-band contains $m$ $(\theta,t)$-cells.

\end{proof}

%
%
%
%
%
%

\smallskip


\subsection{Shafts} \

\begin{lemma} \label{controlled has reduced}

Suppose $\Delta$ is a quasi-trapezium such that its history $H$ is controlled. Then the base of $\Delta$ is a reduced word.

\end{lemma}

\begin{proof}

Suppose there exists a disk or $a$-cell in $\Delta$. Then, Lemma \ref{quasi-trapezia} implies that there exists a quasi-trapezium $\Gamma$ over $M_a(\textbf{M})$ with the same base as $\Delta$. If there exists an $a$-cell in $\Gamma$, then Lemma \ref{a-bands between a-cells} implies that there exists an $a$-edge corresponding to the `special' input sector on either a $(\theta,q)$-cell of $\Gamma$ or on $\partial\Gamma$. However, since every rule of a controlled history locks the `special' input sector, this is impossible.

So, we can assume $\Gamma$ is a trapezium, so that it suffices to consider trapezia with a controlled history. The statement then follows from Lemmas \ref{M controlled} and \ref{computations are trapezia}. 

\end{proof}

The following definitions, as they were used in [19] and [26], are used to define a useful measure on minimal diagrams.

A trapezium $\Delta$ over the canonical presentation of $M(\textbf{M})$ is called \textit{standard} if its base is the standard base (or its inverse) and its history $H$ contains a subword $H_0^{\pm1}$ for some controlled history $H_0$. Naturally, a history $H$ (i.e $H\in F(\Theta^+)$) is called \textit{standard} if there exists a standard trapezium with history $H$. By Lemmas \ref{M controlled} and \ref{computations are trapezia}, a standard trapezium is uniquely determined by its history.

Suppose $\Pi$ is a disk contained in a minimal diagram with some $t$-spoke $\pazocal{B}$. Further, suppose there is a subband $\pazocal{C}$ of $\pazocal{B}$ that also starts on $\Pi$ and has standard history $H$ such that $\text{Lab}(\partial\Pi)$ is $H$-admissible. Then the $t$-band $\pazocal{C}$ is called a \textit{shaft} at $\Pi$. 

For $0\leq\lambda<1/2$, a shaft $\pazocal{C}$ at a disk $\Pi$ is a \textit{$\lambda$-shaft} at $\Pi$ if for every factorization $H\equiv H_1H_2H_3$ satisfying $\|H_1\|+\|H_3\|<\lambda\|H\|$, $H_2$ is still a standard history. Note also that shafts are 0-shafts.

The following adapts Lemma 7.11 of [26] to the case of potential $a$-cells.

\begin{lemma} \label{shafts}

Let $\Pi$ be a disk in a minimal diagram $\Delta$ and $\pazocal{C}$ be a $\lambda$-shaft at $\Pi$ with history $H$. Then $\pazocal{C}$ has no factorization $\pazocal{C}=\pazocal{C}_1\pazocal{C}_2\pazocal{C}_3$ such that

\begin{addmargin}[1em]{0em}

$(1)$ the sum of the lengths of $\pazocal{C}_1$ and $\pazocal{C}_3$ do not exceed $\lambda\|H\|$ and

$(2)$ $\Delta$ contains a quasi-trapezium $\Gamma$ such that the top (or bottom) label of $\Gamma$ has $L+1$ $t$-letters and $\pazocal{C}_2$ starts on the bottom and ends on the top of $\Gamma$.

\end{addmargin}

\end{lemma}

\begin{proof}

Assuming such a factorization exists, we first construct from $\Gamma$ the quasi-trapezium $\Gamma_2$ over $M_a(\textbf{M})$ obtained from Lemma \ref{quasi-trapezia}. Since $\Gamma_2$ and $\Gamma$ have the same base and $\pazocal{C}$ is left unchanged by the process, we merely assume for simplicity that $\Gamma$ is already a quasi-trapezium over $M_a(\textbf{M})$.

By the definition of a $\lambda$-shaft, $\Gamma$ contains a subdiagram $\Gamma'$ which is a quasi-trapezium over $M_a(\textbf{M})$ with controlled history. Lemma \ref{controlled has reduced} then implies that the base of $\Gamma'$ is a reduced word (in fact, the proof of Lemma \ref{controlled has reduced} implies that $\Gamma'$ is a trapezium), so that Lemma \ref{M_a no annuli} then implies that the base of $\Gamma$ is a reduced word. 

Perhaps chopping off the sides of the diagram, we now assume that the bottom (and top) label of $\Gamma$ starts and ends with one of its $L+1$ $t$-letters. Then, removing one of these end $t$-bands, we assume that the base (or its inverse) is, up to cyclic permutation, standard.

Let $\Gamma'_1$ (respectively $\Gamma'_2$) be the subquasi-trapezium over $M_a(\textbf{M})$ bounded by the bottom (respectively top) of $\Gamma$ and the bottom (respectively top) of $\Gamma'$.

By Lemma \ref{computations are trapezia} and the definition of shaft, there exists a trapezium $\Psi$ with bottom label $W\equiv\text{Lab}(\partial\Pi)$ and history $H$ (and so top label $V\equiv W\cdot H$). Let $\Psi'$ be the subtrapezium with the same controlled history as $\Gamma'$ and $\Psi'_1$ (respectively $\Psi'_2$) be the subtrapezium whose whose top (respectively bottom) is the bottom (respectively top) of $\Psi'$ and whose history is the same as $\Gamma'_1$ (respectively $\Gamma'_2$).

By Lemma \ref{M controlled}, the top labels (respectively bottom labels) of $\Psi'_1$ and $\Gamma'_1$ (respectively $\Psi'_2$ and $\Gamma'_2$) are equivalent. Further, the side labels of $\Psi_j'$ are the same as those of $\Gamma_j'$ for $j=1,2$.

Gluing these parts of the contours together, we obtain diagrams $\Lambda_1$ and $\Lambda_2$ over $M_a(\textbf{M})$ with contours labelled by $V_1(V_1')^{-1}$ and $V_2(V_2')^{-1}$, where $V_1$ and $V_2$ (respectively $V_1'$ and $V_2'$) are the bottom and top labels of $\Psi$ (respectively of $\Gamma$). 

Since $V_1'$ and $V_1$ ($V_2'$ and $V_2$) are admissible, any maximal $a$-band with one end on an $a$-cell must have its other end on a $(\theta,q)$-cell or on the contour corresponding to one of the words' `special' input sectors. So, $V_1'$ and $V_1$ ($V_2'$ and $V_2$) differ only by their projections to the `special' input sector.

As a result, the contour labels of the diagrams $\Lambda_1$ and $\Lambda_2$ are freely equal to $a$-relations, so that we can replace them with $a$-cells.

Now, assume $\lambda=0$. Then $\Gamma$ and $\Pi$ have a common edge, so that they form a subdiagram $\Delta'$. Perhaps adding two pairs of cancellable $a$-cells to $\Delta'$, we assume that $\Delta'$ has a subdiagram $\Delta''$ with contour label freely equal to $W\cdot H$ (and such that the complement of $\Delta''$ in $\Delta'$ is at most two $a$-cells). But since $W$ is accepted, there exists a disk relation corresponding to $W\cdot H$, so that we can replace $\Delta''$ with one disk. Excising $\Delta'$ from $\Delta$ and replacing it with this new diagram reduces the number of $(\theta,t)$-cells of $\Delta$ while leaving the number of disks fixed, yielding a contradiction.

We now assume the general case of $\lambda>0$. Let $\pazocal{C}=\pazocal{C}_1\pazocal{C}_2\pazocal{C}_3$, with $H=H_1H_2H_3$ such that $H_i$ is the history of $\pazocal{C}_i$. Consider the diagram $\Delta'=\Pi\cup \pazocal{C}_1\pazocal{C}_2\cup\Gamma$. 

Let $\Psi_1''$ be the subtrapezium of $\Psi$ with the same bottom label and top label $V_1$. Let $E$ be the diagram obtained from attaching the appropriate $a$-cell to the top of $\Psi_1''$ so that its top label is $V_1'$. Further, let $\exists$ be the mirror image of $E$ and $E\exists$ be the diagram fomed by gluing $\exists$ to $E$ along the bottom of $\Psi_1''$. Note that there are at most $\lambda\|H\|L$ $(\theta,t)$-cells in $\exists$. Then glue the top of $E\exists$ to the bottom of $\Gamma$ and along $\pazocal{C}_1$.

As in the previous case, we set $\Delta'$ as the subdiagram formed by $\Pi$, $\Gamma$, and $E$, and form the diagram $\Delta''$ made of one disk, $\exists$, and perhaps some new $a$-cells. Then $\Delta''$ has at most $\lambda\|H\|L$ $(\theta,t)$-cells, while $\Delta_1'$ least $\|H_2\|L\geq(1-\lambda)\|H\|L$. Since $\lambda<1/2$, it follows that $\Delta''$ has less $(\theta,t)$-cells than $\Delta'$ and an equal number of disks. Excising $\Delta'$ from $\Delta$ and replacing it with $\Delta''$ then contradicts the minimality of $\Delta$.

\end{proof}

\medskip


\subsection{Designs on a Disk} \

In this section, we recall several definitions from [19], yielding a measure on minimal diagrams supplementary to weight that will prove useful in future arguments.

Let $\pazocal{D}$ be a disk in the Euclidean plane, $\textbf{T}$ be a finite set of disjoint chords, and $\textbf{Q}$ be a finite set of disjoint simple curves in $\pazocal{D}$, called \textit{arcs} (as to differentiate them from the chords). 

Assume that arcs belong to the open disk $\pazocal{D}^\circ$ and that each chord crosses any arc transversely and at most one, with the intersection not coming at either of the arc's endpoints.

With these assumptions, the pair $(\textbf{T},\textbf{Q})$ is called a \textit{design} on the disk.

The length of an arc $C\in\textbf{Q}$, denoted $|C|$, is the number of chords crossing it. \textit{Subarcs} are defined in the natural way, so that the inequality $|D|\leq|C|$ is clear for $D$ a subarc of $C$.

An arc $C_1$ is \textit{parallel} to an arc $C_2$, denoted $C_1 \ \| \ C_2$, if every chord crossing $C_1$ also crosses $C_2$. Note that this relation is transitive but not symmetric.

For $\lambda\in(0,1)$ and $n$ a positive integer (note that $\lambda$ and $n$ are here not as specified in other sections), we say that a design $(\textbf{T},\textbf{Q})$ satisfies property $P(\lambda,n)$ if for any collection of $n$ distinct arcs $C_1,\dots,C_n$, there are no subarcs $D_1,\dots,D_n$, respectively, such that $|D_i|>(1-\lambda)|C_i|$ for all $i$ and $D_1 \ \| \ D_2 \ \| \dots \| \ D_n$.

For a design $(\textbf{T},\textbf{Q})$, define the length of $\textbf{Q}$, $\ell(\textbf{Q})$, to be $$\ell(\textbf{Q})=\sum_{C\in\textbf{Q}}|C|$$

Every chord $T\in\textbf{T}$ divides $\pazocal{D}$ into two half-disks. If one of these half-disks contains no other chords, then the chord $T$ is called \textit{peripheral} and the relevant half-disk $O_T$ (not including $T$) a \textit{peripheral half-disk}.

For $C\in\textbf{Q}$, an arc $D$ is called an \textit{extension} of $C$ if $C$ is a subarc of $D$. Note that extensions need not be elements of $\textbf{Q}$; however, we only consider extensions of elements of $\textbf{Q}$ that, upon replacing the arcs with that extensions, yield a new design $(\textbf{T},\textbf{Q}')$. An arc $C\in\textbf{Q}$ is \textit{maximal} if there is no extension $D$ of $C$ such that $|D|>|C|$.

\medskip

\begin{lemma} \label{basic design}

\textit{(Lemma 8.3 of [19])} Let $(\textbf{T},\textbf{Q})$ be a design with $\#\textbf{T}\geq1$. Then every arc $C\in\textbf{Q}$ has a maximal extension $D$ ending in two different peripheral half-disks and so that the set $\textbf{Q}'$ of such extensions yields a design $(\textbf{T},\textbf{Q}')$.

\end{lemma}

\begin{lemma} \label{design}

\textit{(Lemma 8.2 of [19])} There is a constant $c=c(\lambda,n)$ such that for any design $(\textbf{T},\textbf{Q})$ satisfying property $P(\lambda,n)$, the inequality $\ell(\textbf{Q})\leq c(\#\textbf{T})$ holds.

\end{lemma}

Now, consider $\lambda\in[0,1/2)$ as in the parameters of Section 2.4. For every $t$-spoke $\pazocal{Q}$ of a minimal diagram $\Delta$, choose the $\lambda$-shaft of maximal length that is a subband of $\pazocal{Q}$. Note that if $\pazocal{Q}$ connects two disks, then it can produce two maximal maximal $\lambda$-shafts, one for each disk. Define $\sigma_\lambda(\Delta)$ as the sum of the lengths of all $\lambda$-shafts.

If $\Delta$ is a minimal diagram over $G_a(\textbf{M})$, identify $\Delta$ as a disk and construct the following design: Let the middle lines of maximal $\theta$-bands be the chords and the middle lines of maximal $\lambda$-shafts be the arcs. 

Note that there is a minor subtlety in this construction: If a maximal $t$-spoke connects two disks, then it may correspond to two distinct $\lambda$-shafts, and these $\lambda$-shafts may overlap; however, this issue can be remedied by `making room' in the spoke for both arcs to fit and be disjoint. 

Then, the length $|C|$ of an arc is the number of cells in the $\lambda$-shaft and $\#\textbf{T}\leq|\partial\Delta|/2$ since every maximal $\theta$-band ends twice on $\partial\Delta$.

\begin{lemma} \label{G_a design}

\textit{(Lemma 8.5 of [19])} If $\Delta$ is a minimal diagram over $G_a(\textbf{M})$, then the design constructed above satisfies $P(\lambda,2L+1)$. In particular, there is a constant $c$ depending on $\lambda$ and $L$ such that $\sigma_\lambda(\Delta)\leq c|\partial\Delta|$.

\end{lemma}

\medskip


\section{Upper bound on the weight of minimal diagrams}

\subsection{Weakly minimal diagrams} \

The goal in this section is to show that $\text{wt}(\Delta)\leq N_3|\partial\Delta|^3$ for all minimal diagrams $\Delta$, where $N_3$ is the parameter as assigned in Section 3.3. To prove this property, we consider a larger class of diagrams over the disk presentation of $G_a(\textbf{M})$ called weakly minimal.

First, we define a \textit{diminished} diagram over the disk presentation similar to how we define a minimal diagram, but with the first coordinate of the diagram's signature omitted. In other words, a diagram is diminished if it has the minimal number of $(\theta,t)$-cells amongst diagrams of the same contour label, the minimal number of $(\theta,q)$-cells amongst diagrams with the same contour label and same number of $(\theta,t)$-cells, etc; however, a diminished diagram may contain more disks than another diagram with the same contour label. Note that this means that a minimal diagram need not be diminished, though a diminished diagram over $M_a(\textbf{M})$ is necessarily minimal.

Let $\pazocal{C}$ be a cutting $q$-band of a diminished diagram $\Delta$ containing disks, i.e $\pazocal{C}$ ends on $\partial\Delta$ twice. Then $\pazocal{C}$ is called a \textit{stem band} if it is either a rim band of $\Delta$ or both components of $\Delta\setminus\pazocal{C}$ contain disks. The unique maximal subdiagram of $\Delta$ satisfying the property that every cutting $q$-band is a stem band is called the \textit{stem} of $\Delta$ and denoted $\Delta^*$.

If $\pazocal{C}$ is a cutting $q$-band that is not a stem band, then a component $\Gamma$ of $\Delta\setminus\pazocal{C}$ contains no disks. In this situation, the cells of $\Gamma$ are called \textit{crown} cells. Note that one can construct $\Delta^*$ from $\Delta$ simply by cutting off all of the crown cells.

Finally, a diminished diagram $\Delta$ is said to be \textit{weakly minimal} if its stem $\Delta^*$ is a minimal diagram. 

\begin{lemma} \label{weakly minimal} \textit{(Lemma 7.17 of [26])}

$(a)$ If $\Delta_1$ is a subdiagram of a weakly minimal diagram $\Delta$, then $\Delta_1$ is weakly minimal and $\Delta_1^*\subset\Delta^*$

$(b)$ In the same setting as $(a)$, $\sigma_\lambda(\Delta_1^*)\leq\sigma_\lambda(\Delta^*)$

$(c)$ There exists a constant $c$ dependent on $\lambda$ and $L$ such that $\sigma_\lambda(\Delta^*)\leq c|\partial\Delta|$ for every weakly minimal diagram $\Delta$

$(d)$ Let $\pazocal{C}$ be a cutting $q$-band of a diminished diagram $\Delta$ and $\Delta_1$, $\Delta_2$ be the components of $\Delta\setminus\pazocal{C}$. Suppose $\Delta_1\cup\pazocal{C}$ is a diminished diagram without disks (i.e a minimal diagram over $M_a(\textbf{M})$) and $\pazocal{C}\cup\Delta_2$ is a weakly minimal diagram. Then $\Delta$ is weakly minimal

$(e)$ A weakly minimal diagram $\Delta$ contains no $\theta$-annuli or $(\theta,q)$-annuli

\end{lemma}

\smallskip


\subsection{Definition of the minimal counterexample and cloves} \

We now turn our attention to weakly minimal diagrams and look to prove an upper bound for the weight of such diagrams in terms of their contour length.

In particular, our goal is to prove that for any weakly minimal diagram $\Delta$, we have the inequality $$\text{wt}(\Delta)\leq N_2(|\partial\Delta|+\sigma_\lambda(\Delta^*))^3$$ for large enough choice of the parameter parameter $N_2$. The proof of this follows a similar path as that presented in Section 7 of [26] (though mixtures are not used here) and Section 9 of [19] (taking $F(x)=x^2$ in that setting). As such, many of the proofs of the statements that follow will either be omitted or amount to a remark on how it differs from the corresponding proof in [26].

To this end, let $\Delta$ be a weakly minimal counterexample diagram with minimal possible sum $|\partial\Delta|+\sigma_\lambda(\Delta^*)$, i.e so that 
$$\text{wt}(\Delta)>N_2(|\partial\Delta|+\sigma_\lambda(\Delta^*))^3$$ 
but for any weakly minimal diagram $\Gamma$ such that $|\partial\Gamma|+\sigma_\lambda(\Gamma^*)<|\partial\Delta|+\sigma_\lambda(\Delta^*)$, we have
$$\text{wt}(\Gamma)\leq N_2(|\partial\Gamma|+\sigma_\lambda(\Gamma^*))^3$$ 
By Lemma \ref{M_a cubic} and the assignment of parameters, $\Delta$ must contain a disk. Since $\Delta^*$ is minimal and contains every disk and spoke of $\Delta$, it follows that we can apply Lemma \ref{graph} to find a disk $\Pi$ with $L-3$ consecutive $t$-spokes $\pazocal{Q}_1,\dots,\pazocal{Q}_{L-3}$ ending on $\partial\Delta$ and bounding subdiagrams $\Gamma_1,\dots,\Gamma_{L-4}$ containing no disks.

The subdiagrams of $\Delta$ bounded by $\Pi$, $\pazocal{Q}_i$, and $\pazocal{Q}_j$ for $1\leq i<j\leq L-3$ are called \textit{cloves} and denoted $\Psi_{ij}$. The maximal clove $\Psi_{1,L-3}$ is simply denoted $\Psi$.

\begin{lemma} \label{rim theta-bands}

\textit{(Lemma 7.19 of [26])} If $\Delta$ contains a rim $\theta$-band with base of length $s$, then $s>K$.

\end{lemma}

\smallskip


\subsection{Properties of the cloves of $\Delta$} \

\begin{lemma} \label{no subcombs}

No subcomb of $\Delta$ is contained in $\Psi$.

\end{lemma}

\begin{proof}

As the subdiagram of a diminished diagram is clearly diminished and $\Psi$ contains no disks, it follows that $\Psi$ is a minimal diagram over $M_a(\textbf{M})$.

Assuming the statement false, it is possible to find subcomb $\Gamma$ in $\Psi$ containing no maximal $q$-bands other than its handle.

But then the same proof as Lemma \ref{M_a cubic} can be applied with just a few adjustments: 

The subdiagram $\Delta_0$ obtained from excising $\Gamma$ from $\Delta$ again satisfies $|\partial\Delta_0|\leq|\partial\Delta|-1$. By Lemma \ref{weakly minimal}$(a)$,$(b)$, $\Delta_0$ is weakly minimal with $\sigma_\lambda(\Delta_0^*)\leq\sigma_\lambda(\Delta^*)$. By the inductive hypothesis, this then implies that $$\text{wt}(\Delta_0)\leq N_2(|\partial\Delta|+\sigma_\lambda(\Delta^*)-1)^3$$
so that Lemma \ref{one q-band comb} yields $$\text{wt}(\Delta)\leq N_2(|\partial\Delta|+\sigma_\lambda(\Delta^*)-1)^3+c_8|\partial\Gamma|^2$$
Again noting that $|\partial\Gamma|\leq|\partial\Delta|+2\delta$ by Lemma \ref{lengths}, we can take $N_2$ large enough with respect to $c_8$ and $\delta^{-1}$ to then give the desired contradiction.

\end{proof}

\begin{lemma} \label{theta-bands in cloves} \textit{(Lemma 7.22 of [26])}

(1) Every maximal $\theta$-band of $\Psi$ crosses either $\pazocal{Q}_1$ or $\pazocal{Q}_{L-3}$

(2) There exists an $r$ satisfying $L/2-3\leq r\leq L/2$ such that the $\theta$-bands of $\Psi$ crossing $\pazocal{Q}_{L-3}$ do not cross $\pazocal{Q}_r$ and the $\theta$-bands of $\Psi$ crossing $\pazocal{Q}_1$ do not cross $\pazocal{Q}_{r+1}$

\end{lemma}

For $1\leq i<j\leq L-3$, denote $\textbf{p}_{ij}$ as the shared subpath of $\partial\Psi_{ij}$ and $\partial\Delta$. For simplicity, denote the path $\textbf{p}_{1,L-3}$ associated to the maximal clove simply as $\textbf{p}$.

Let $\bar{\Delta}$ be the subdiagram formed by $\Pi$ and $\Psi$. Further, let $\bar{\textbf{p}}$ be the path $\textbf{top}(\pazocal{Q}_1)\textbf{u}^{-1}\textbf{bot}(\pazocal{Q}_{L-3})^{-1}$, where $\textbf{u}$ is a subpath of $\partial\Pi$, such that $\bar{\textbf{p}}$ separates $\bar{\Delta}$ from the remaining subdiagram $\Psi'$ of $\Delta$.

Similarly, define $\bar{\Delta}_{ij}$, paths $\bar{\textbf{p}}_{ij}=\textbf{top}(\pazocal{Q}_i)\textbf{u}_{ij}^{-1}\textbf{bot}(\pazocal{Q}_j)^{-1}$, where $\textbf{u}_{ij}$ is a subpath of $\partial\Pi$ separating from the subdiagram $\Psi_{ij}'$.

Further, define $H_1,\dots,H_{L-3}$ as the histories of the spokes $\pazocal{Q}_1,\dots,\pazocal{Q}_{L-3}$, read starting from the disk $\Pi$, and $h_1,\dots,h_{L-3}$ as these histories' lengths. Lemma \ref{theta-bands in cloves} then implies the inequalities
$$h_1\geq h_2\geq\dots\geq h_r; \ \ h_{r+1}\leq\dots\leq h_{L-3}$$ 

where $L/2-3\leq r\leq L/2$. It then follows that $H_{i+1}$ is a prefix of $H_i$ for $i=1,\dots,r-1$ while $H_j$ is a prefix of $H_{j+1}$ for $j=r+1,\dots,L-4$.

Let $W$ be the accepted configuration corresponding to the label of $\partial\Pi$. Then, using the notation of Section 5.3, $W\equiv W(1)W(2)\dots W(L)$, where $W(2),\dots,W(L)$ are all copies of the same configuration $V$ of $\textbf{M}_5$. Further, by Lemma \ref{accepted configuration a-length}, we have $\frac{1}{2}\|V\|\leq\|W(1)\|\leq\frac{3}{2}\|V\|$.

\begin{lemma} \label{path inequalities} \textit{(Lemma 7.24 of [26])}

(1) If $i\leq r$ and $j\geq r+1$, then $|\textbf{p}_{ij}|\geq|\textbf{p}_{ij}|_\theta+|\textbf{p}_{ij}|_q\geq h_i+h_j+(j-i)N+1$

(2) $|\bar{\textbf{p}}_{ij}|\leq h_i+h_j+(L-j+i+1)|V|-1$

\end{lemma}

\begin{proof}

%

The difference between the inequality presented here and the one presented in [26] is attributed to the fact that $W(1)$ need not be a copy of $W(i)$ for $i\geq2$, whereas this is not the case in [26].

\end{proof}

The following is the analogue of Lemma 7.26 of [26].

\begin{lemma} \label{eps intro}

If $j-i>L/2$, then $|\textbf{p}_{ij}|+\sigma_\lambda(\bar{\Delta}_{ij}^*)<(1+\eps)|\bar{\textbf{p}}_{ij}|$, where $\eps=N_2^{-\frac{1}{4}}$.

\end{lemma}

\begin{proof}

Set $y=|\textbf{p}_{ij}|+\sigma_\lambda(\bar{\Delta}_{ij}^*)$ and $d=y-|\bar{\textbf{p}}_{ij}|$. Suppose $d\geq\eps|\bar{\textbf{p}}_{ij}|$.

Then $d\geq y-\eps^{-1}d$, so that $d\geq(1+\eps^{-1})^{-1}y\geq\frac{\eps y}{2}$ for large enough choice of $N_2$. Note that, since $\Psi_{ij}'$ and $\bar{\Delta}_{ij}$ have no common spokes, $\sigma_\lambda(\bar{\Delta}_{ij}^*)+\sigma_\lambda((\Psi_{ij}')^*)\leq\sigma_\lambda(\Delta^*)$. Then, since $|\partial\Delta|-|\partial\Psi_{ij}'|\geq|\textbf{p}_{ij}|-|\bar{\textbf{p}}_{ij}|$, we have
$$(|\partial\Delta|+\sigma_\lambda(\Delta^*))-(|\partial\Psi_{ij}'|+\sigma_\lambda((\Psi_{ij}')^*))\geq d>0$$
So, setting $x=|\partial\Delta|+\sigma_\lambda(\Delta^*)$, it follows from the inductive hypothesis that
$$\text{wt}(\Psi_{ij}')\leq N_2(x-d)^3$$
By Lemma \ref{path inequalities}, $|\bar{\textbf{p}}_{ij}|<|\textbf{p}_{ij}|+|\partial\Pi|$, and so $|\partial\Psi_{ij}|<2|\textbf{p}_{ij}|+|\partial\Pi|$. Since $|\partial\Pi|\leq(L+1)|\bar{\textbf{p}}_{ij}|$, we get
$$|\partial\Psi_{ij}|<(L+3)|\textbf{p}_{ij}|\leq (L+3)y$$
Then, Lemma \ref{M_a cubic} implies that $\text{wt}(\Psi_{ij})\leq N_1(L+3)^3y^3$. 

Further, since $|\partial\Pi|\leq(L+1)|\bar{\textbf{p}}_{ij}|<(L+1)|\textbf{p}_{ij}|\leq(L+1)y$, we have 
$$\text{wt}(\Pi)\leq c_7(L+1)^2y^2$$
Together, these inequalities give 
$$
\text{wt}(\Delta)\leq N_2(x-d)^3+N_1(L+3)^3y^3+c_7(L+1)^2y^2 
$$
Hence, a contradiction has been reached as long as $N_2(x-d)^3+N_1(L+3)^3y^3+c_7(L+1)^2y^2\leq N_2x^3$, i.e 
$$N_1(L+3)^3y^3+c_7(L+1)^2y^2\leq 3N_2xd(x-d)+N_2d^3$$
As $x\geq d$, it suffices to show that
$$N_2d^3\geq N_1(L+3)^3y^3+c_7(L+1)^2y^2$$
But then $d^3\geq\frac{\eps^3}{8}y^3$, so that $N_2d^3\geq\frac{\sqrt[4]{N_2}}{8}y^3$. Letting $N_2$ be large enough in comparison to $N_1$, $c_7$, and $L$ then justifies the desired inequality.

\end{proof}

For $i=1,\dots,L-4$, if the adjacent pair of $t$-letters associated to $\pazocal{Q}_i$ and to $\pazocal{Q}_{i+1}$ are $t(1)$ and $t(2)$, then $\Psi_{i,i+1}$ is said to be the \textit{distinguished clove}.

\begin{lemma} \label{non-distinguished a-cells}

If $\Psi_{i,i+1}$ is not the distinguished clove, then there exists a simple path $\textbf{q}_{i,i+1}$ homotopic to $\textbf{p}_{i,i+1}$ such that:

\begin{addmargin}[1em]{0em}

(1) the subdiagram $\Lambda_{i,i+1}$ with contour $\textbf{q}_{i,i+1}\textbf{p}_{i,i+1}^{-1}$ contains only $(\theta,a)$-cells and $a$-cells;

(2) the diagram $\Psi_{i,i+1}^0$ formed by cutting $\Lambda_{i,i+1}$ off of $\Psi_{i,i+1}$ contains no $a$-cells;

(3) $\textbf{q}_{i,i+1}$ contains no $a$-edges labelled by a letter from the `special' input sector;

(4) all maximal $\theta$-bands of $\Lambda_{i,i+1}$ connect $\textbf{p}_{i,i+1}$ to $\textbf{q}_{i,i+1}$

\end{addmargin}

\end{lemma}

\begin{proof}

If $\Psi_{i,i+1}$ contains no $a$-cells, then the statement holds for $\textbf{q}_{i,i+1}=\textbf{p}_{i,i+1}$. So, assume there exists an $a$-cell $\pi$ in $\Psi_{i,i+1}$. 

By Lemma \ref{no subcombs}, no $q$-edge of $\Psi_{i,i+1}$ is labelled by a $q$-edge of the `special' input sector, i.e $Q_0(1)$ or $R_0(1)$. So, by Lemmas \ref{a-bands between a-cells} and \ref{M_a no annuli}, every $a$-band starting on $\pi$ must end on $\textbf{p}_{i,i+1}$.

Let $\textbf{e}$ (respectively $\textbf{f}$) be the first (respectively last) edge of $\textbf{p}_{i,i+1}$ that marks the end of an $a$-band $\textbf{A}_e$ (respectively $\textbf{A}_f)$ starting on $\pi$.

Suppose a $\theta$-band $\pazocal{T}_0$ crosses both $\textbf{A}_e$ and $\textbf{A}_f$. Then Lemmas \ref{a-bands between a-cells} and \ref{M_a no annuli} imply that every $a$-edge of $\pi$ marks the start of an $a$-band that crosses $\pazocal{T}_0$. So, we can find a $\theta$-band $\pazocal{T}$ crossing every such band and such that there are no cells between $\pazocal{T}$ and $\pi$. However, we can then transpose $\pazocal{T}$ with $\pi$ and reduce the number of $(\theta,a)$-cells while leaving all other numbers fixed, contradicting the minimality of $\Psi_{i,i+1}$.

Now, letting $\textbf{p}_\pi$ be a subpath of $\textbf{p}_{i,i+1}$ starting with $\textbf{e}$ and ending with $\textbf{f}$, define the homotopic path $\textbf{q}_\pi=\textbf{bot}(\textbf{A}_e)\textbf{top}(\textbf{A}_f)^{-1}$. Note that all edges of $\textbf{q}_\pi$ are $\theta$-edges.

For the subdiagram $\Lambda_\pi$ with contour label $\textbf{p}_\pi\textbf{q}_\pi^{-1}$, no edge of the contour is a $q$-edge, so that every cell is either a $(\theta,a)$- or an $a$-cell.

Then, define the path $\textbf{q}_{i,i+1}'$ homotopic to $\textbf{p}_{i,i+1}$ by replacing $\textbf{p}_\pi$ in $\textbf{p}_{i,i+1}$ with the path $\textbf{q}_\pi$ for all $a$-cells $\pi$.

For any $a$-edge $\textbf{e}'$ of $\textbf{q}_{i,i+1}'$ labelled by a letter from the `special' input sector, there exists an $a$-band $\pazocal{B}$ starting on it. As above, this band must end on $\textbf{q}_{i,i+1}'$, say at $\textbf{f}'$. Letting $\textbf{q}_{\pazocal{B}}'$ be the subpath of $\textbf{q}_{i,i+1}'$ starting on $\textbf{e}'$ and ending on $\textbf{f}'$, consider the path homotopic to it $\textbf{q}_{\pazocal{B}}=\textbf{bot}(\pazocal{B})$. Note that $\textbf{q}_{\pazocal{B}}$ is made entirely of $\theta$-edges and that the subdiagram $\Lambda_{\pazocal{B}}$ with contour $\textbf{q}_{\pazocal{B}}'\textbf{q}_{\pazocal{B}}^{-1}$ contains only $(\theta,a)$-cells.

Define the path $\textbf{q}_{i,i+1}$ homotopic to $\textbf{q}_{i,i+1}'$ by replacing $\textbf{q}_{\pazocal{B}}$ in $\textbf{q}_{i,i+1}'$ with the path $\textbf{q}_{\pazocal{B}}$ for all such bands $\pazocal{B}$.

It is then easy to see that the statement holds for this path.

%
%

\end{proof}

To extend the definitions arising from Lemma \ref{non-distinguished a-cells}, if $\Psi_{i,i+1}$ is the distinguished clove, then define $\textbf{q}_{i,i+1}$ as $\textbf{p}_{i,i+1}$. We then define $\textbf{q}_{ij}$ and $\Psi_{ij}^0$ for $1\leq i<j\leq L-3$ in the natural way.

\begin{lemma} \label{q_{ij} inequalities}

\textit{(Lemma 7.27 of [26])} If $i\leq r$ and $j\geq r+1$, then $|\textbf{q}_{ij}|\geq|\textbf{q}_{ij}|_\theta+|\textbf{q}_{ij}|_q\geq h_i+h_j+(j-i)N+1$

\end{lemma}

%
%

It follows from Lemmas \ref{no subcombs} and \ref{theta-bands in cloves} that for $1\leq j\leq r-1$, if $\Psi_{j,j+1}$ is not the distinguished clove, then $\Psi_{j,j+1}^0$ contains a trapezium $\Gamma_j$ of height $h_{j+1}$ with side $t$-bands. Similarly, we have the trapezia $\Gamma_j$ of height $h_j$ in $\Psi_{j,j+1}^0$ for $r+1\leq j\leq L-4$ given that $\Psi_{j,j+1}$ is not distinguished.

For each such $j$, the bottom of $\Gamma_j$, $\textbf{y}_j$, is shared with $\partial\Pi$ and has label $Vt(l)$ for some $l$. Its top is denoted $\textbf{z}_j$. For $2\leq j\leq r-1$, given $\Gamma_j$ and $\Gamma_{j-1}$ both exist (i.e neither $\Psi_{j-1,j}$ and $\Psi_{j,j+1}$ are the distinguished clove), their bottom labels are coordinate shifts of one another while the history $H_j$ is a prefix of $H_{j-1}$; so, $h_{j+1}$ different $\theta$-bands of $\Gamma_{j-1}$ form a copy $\Gamma_j'$ of the trapezium $\Gamma_j$ with top and bottom paths $\textbf{z}_j'$ and $\textbf{y}_j'=\textbf{y}_{j-1}$.

If $\Psi_{j,j+1}$ is not the distinguised clove for $1\leq j\leq r-1$, we denote by $E_j$ (respectively $E_j^0$) the comb formed by the maximal $\theta$-bands of $\Psi_{j,j+1}$ (respectively $\Psi_{j,j+1}^0$) crossing the $t$-spoke $\pazocal{Q}_j$ but not crossing $\pazocal{Q}_{j+1}$. Its handle $\pazocal{C}_j$ of height $h_j-h_{j+1}$ is contained in $\pazocal{Q}_j$; $\partial E_j$ (respectively $\partial E_j^0$) consists of the side of this handle, the paths $\textbf{z}_j$, and the path $\textbf{p}_{j,j+1}$ (respectively $\textbf{q}_{j,j+1}$).

\begin{lemma} \label{a-bands ending on theta-band}

\textit{(Lemma 7.30 of [26])} At most $2N$ $a$-bands starting on the path $\textbf{y}_j$ (or $\textbf{z}_j$) can end on the $(\theta,q)$-cells of the same $\theta$-band.

\end{lemma}

%
%
%

Recall the parameter $L_0$, assigned in Section 3.3, is assigned before $L$. Without loss of generality, suppose $h\defeq h_{L_0+1}\geq h_{L-L_0-3}$.

\begin{lemma} \label{h leq 1}

\textit{(Lemma 7.31 of [26])} If $h\leq L_0^2|V|_a$, then the number of trapezia $\Gamma_j$ satisfying $|\textbf{z}_j|_a\geq|V|_a/2c_6N$ for $j\in[L_0+1,r-1]$ or $j\in[r+1,L-L_0-3]$ is less than $L/5$.

\end{lemma}

\begin{lemma} \label{h leq 2}

\textit{(Lemma 7.32 of [26])} If $h\leq L_0^2|V|_a$, then the histories $H_1$ and $H_{L-3}$ have different first letters.

\end{lemma}

%
%
%
%

\begin{lemma} \label{h leq 3}

\textit{(Lemma 7.33 of [26])} If $h\leq L_0^2|V|_a$, then $|V|_a>\frac{LN}{4L_0}$.

\end{lemma}

%
%

\begin{lemma} \label{h >}

\textit{(Lemma 7.34 of [26])} The inequality $h>L_0^2|V|_a$ must be true.

\end{lemma}

Next, we define the positive integers $\omega$ and $\tau$ satisfying $1\leq\omega<\tau\leq L_0$ satisfying the property that $\omega,\dots,\tau$ is the maximal string of such consecutive integers such that for each $\omega\leq i\leq\tau$, $\Psi_{i,i+1}$ is not the distinguished clove.

Note that $\tau-\omega\geq L_0/2$, $1\leq\omega\leq L_0/2\leq\tau\leq L_0$, and $h_\tau\geq h$.

\begin{lemma} \label{big first h_i}

\textit{(Lemma 7.35 of [26])} For $\omega\leq i\leq\tau$, $h_i>\delta^{-1}$.

\end{lemma}

\begin{proof}


Note the difference presented here can be attributed to the possibility of the distinguished clove.

\end{proof}

\begin{lemma} \label{lambda-shafts in first h_i}

\textit{(Lemma 7.36 of [26])} For $\omega\leq i\leq\tau$, the spoke $\pazocal{Q}_i$ does not contain a $\lambda$-shaft of $\Pi$ of length at least $\delta h$.

\end{lemma}

%
%
%
%

\begin{lemma} \label{trapezia in first h_i}

\textit{(Lemma 7.37 of [26])} For $\omega\leq i\leq\tau-1$, $|\textbf{z}_i|_a>h_{i+1}/c_6$.

\end{lemma}

%
%
%
%
%
%

\begin{lemma} \label{first h_i vs h_{i+1}}

\textit{(Lemma 7.38 of [26])} For $\omega\leq i\leq\tau-1$, $h_{i+1}<(1-\frac{1}{20c_6N})h_i$.

\end{lemma}

\begin{lemma} \label{upper bound on z_i}

\textit{(Lemma 7.39 of [26])} For $\omega+1\leq i\leq \tau-1$, $|\textbf{z}_i|_a\leq 2Nh_i$.

\end{lemma}

%
%
%

For $\omega+1\leq i\leq\tau$, assume that a maximal $a$-band $\textbf{A}$ of $E_i^0$ starts on $\textbf{z}_i$ and ends on a side of a maximal $q$-band $\pazocal{C}$ of $E_i^0$. Then $\textbf{A}$, a part of $\pazocal{C}$, and a subpath $\textbf{z}$ of $\textbf{z}_i$ bound a comb $\nabla$.

\begin{lemma} \label{copy of comb}

\textit{(Lemma 7.29 of [26])} For $\omega+1\leq i\leq\tau$, let $\nabla$ be a comb as above. Then there is a copy of the comb $\nabla$ in the trapezium $\Gamma=\Gamma_{i-1}\setminus\Gamma_i'$.

\end{lemma}

%
%

The following is the analogue of Lemma 7.40 of [26].

\begin{lemma} \label{no one-step}

For $\omega+1\leq i\leq \tau-2$, let $H_i=H_{i+1}H'=H_{i+2}H''H'$ and $\pazocal{C}$ be the computation with history with history $H_i$ corresponding to the trapezium $\Gamma_{i-1}$. Then the subcomputation $\pazocal{D}$ of $\pazocal{C}$ with history $H''H'$ cannot have step history of length 1 so that there exists a sector $QQ'$ such that one of either $Q$ or $Q'$ has a letter inserted next to it to increase the length of this sector for each rule of $\pazocal{D}$.

\end{lemma}

\begin{proof}

Note that for all $\omega\leq j\leq \tau$, the base of $\Gamma_j$ is $t(l)B_4(l)t(l+1)$ for $2\leq l\leq L$ (where if $l=L$, then $l+1$ is replaced with 1). So, assuming the existence of such a subcomputation $\pazocal{D}$ in the statement, then the symmetry of how the rules operate on $B_3(l)$ and its mirror copy implies that they do the same thing to a copy of the $(Q')^{-1}Q^{-1}$ sector. As a result, we can assume without loss of generality that the rules of $\pazocal{D}$ write letters to the right of $Q$.

Let $\pazocal{Q}$ be the maximal $q$-band of $E_i^0$ that, in $\Delta$, is a subband of the $q$-spoke of $\Pi$ corresponding to a coordinate shift the state letter $Q$. Similarly, we define $\pazocal{Q}'$ as the maximal $q$-band for a copy of $Q'$, so that $\pazocal{Q}$ and $\pazocal{Q}'$ are neighbor $q$-bands. Let $\textbf{x}$ be the subpath of $\textbf{z}_i$ between $\pazocal{Q}$ and $\pazocal{Q}'$.

Since $\Gamma_i$ contains a copy $\Gamma_{i+1}'$ of the trapezium $\Gamma_{i+1}$, the bottom of the trapezium $\Gamma_i\setminus\Gamma_{i+1}'$ is a copy $\textbf{z}_{i+1}'$ of $\textbf{z}_{i+1}$, while the top is $\textbf{z}_i$. This trapezium has history $H''$, so that it inserts one letter per rule. As a result, $|\textbf{x}|_a\geq\|H''\|=h_{i+1}-h_{i+2}$. 

By Lemma \ref{first h_i vs h_{i+1}}, $h_{i+1}-h_{i+2}>\frac{1}{20c_6N}h_{i+1}$. Using $h_{i+1}\geq h$, Lemma \ref{h >}, and $L_0>200c_6N$, we get
$$|\textbf{x}|_a\geq\frac{h}{20c_6N}>\frac{L_0^2|V|_a}{20c_6N}>10L_0|V|_a$$
If an $a$-band starting on $\textbf{x}$ ended on a $(\theta,q)$-cell of $\pazocal{Q}$, then Lemma \ref{copy of comb} implies that there is a copy of this in the trapezium $\Gamma_{i-1}\setminus\Gamma_i'$. By Lemma \ref{trapezia are computations}, though, this would contradict the assumption that rules of $\pazocal{D}$ only write letters in the sector.

Now, consider the comb bounded by $\pazocal{Q}$, $\pazocal{Q}'$, $\textbf{x}$, and $\textbf{q}_{i,i+1}$. Set $s$ and $s'$ as the lengths of $\pazocal{Q}$ and $\pazocal{Q}'$, respectively, so that $s'\leq s$ by Lemma \ref{theta-bands in cloves}. So, there are $|\textbf{x}|_a+s$ maximal $a$-bands starting on $\textbf{x}$ and $\pazocal{Q}$ and ending on $\pazocal{Q}'$ or on $\textbf{q}_{i,i+1}$ by Lemma \ref{no subcombs}. Since only $s'\leq s$ $a$-bands can end on $\pazocal{Q}'$ by Lemma \ref{simplify rules}, at least $|\textbf{x}|_a+s-s'$ of them end on the segement of $\textbf{q}_{i,i+1}$ between $\pazocal{Q}$ and $\pazocal{Q}'$. By Lemma \ref{theta-bands in cloves}(2), the same segment contains $s-s'$ $\theta$-edges, meaning at least $|\textbf{x}|_a$ of them contribute $\delta$ to its length by Lemma \ref{lengths}. So, by Lemma \ref{path inequalities},
\begin{align*}
|\textbf{p}_{i,L-L_0-3}|\geq|\textbf{q}_{i,L-L_0-3}|&\geq h_i+h_{L-L_0-3}+LN/2+\delta\frac{h_{i+1}}{20c_6N} \\
&\geq h_i+h_{L-L_0-3}+LN/2+10\delta L_0|V|_a
\end{align*}
Also by Lemma \ref{path inequalities} and \ref{h >},
$$|\bar{\textbf{p}}_{i,L-L_0-3}|\leq h_i+h_{L-L_0-3}+3L_0N+3\delta L_0|V|_a\leq h_i+h_{L-L_0-3}+3L_0N+3\delta h/L_0$$
So,
\begin{align*}
|\textbf{p}_{i,L-L_0-3}|-|\bar{\textbf{p}}_{i,L-L_0-3}|&\geq (L/2-3L_0)N+\delta(h_{i+1}/20c_6N-3h/L_0) \\
&>\delta h_{i+1}(1/20c_6N-3/L_0) \\
&>\delta h_{i+1}/50c_6N
\end{align*}
by again taking $L_0>>c_6N$.

Noting that $\Psi'_{i,L-L_0-3}$ is the subdiagram of $\Delta$ arising by replacing the subpath $\textbf{p}_{i,L-L_0-3}$ of the contour with $\bar{\textbf{p}}_{i,L-L_0-3}$, letting $\textbf{s}$ be the complement of the $\textbf{p}_{i,L-L_0-3}$ in $\partial\Delta$, and applying Lemma \ref{lengths} gives:
\begin{align*}
|\partial\Psi'_{i,L-L_0-3}|&\leq|\textbf{s}|+|\bar{\textbf{p}}_{i,L-L_0-3}| \\
&<|\textbf{s}|+|\textbf{p}_{i,L-L_0-3}|-\delta h_{i+1}/50c_6N \\
&<|\partial\Delta|-\delta(h_{i+1}/50c_6N-1)
\end{align*}
Since $h_{i+1}>\delta^{-1}$ by Lemma \ref{big first h_i}, we take $\delta^{-1}>100c_6N$ so that $h_{i+1}/100c_6N>1$, so that
$$|\partial\Psi'_{i,L-L_0-3}|<|\partial\Delta|-\delta h_{i+1}/100c_6N$$
Lemma \ref{weakly minimal}(1),(2) also implies that $\Psi'_{i,L-L_0-3}$ is weakly minimal with $\sigma_\lambda((\Psi'_{i,L-L_0-3})^*)\leq\sigma_\lambda(\Delta^*)$. So, the inductive hypothesis gives:
\begin{align*}
\text{wt}(\Psi'_{i,L-L_0-3})&\leq N_2(|\Psi'_{i,L-L_0-3}|+\sigma_\lambda((\Psi'_{i,L-L_0-3})^*))^3 \\
&< N_2(|\partial\Delta|+\sigma_\lambda(\Delta^*)-\delta h_{i+1}/100c_6N)^3
\end{align*}

Next, noting that $|V|_a\leq h_i/L_0^2$ by Lemma \ref{h >}, $h_i>\delta^{-1}>100L_0N$ by Lemma \ref{big first h_i}, and $h_{L-L_0-3}\leq h_i$, we get
$$|\bar{\textbf{p}}_{i,L-L_0-3}|\leq 2h_i+3h_i/100+3\delta h_i/L_0\leq 2.1h_i$$
So, since $|\textbf{p}_{i,L-L_0-3}|\leq(1+\eps)|\bar{\textbf{p}}_{i,L-L_0-3}|$ by Lemma \ref{eps intro}, taking $\eps<<1$ gives 
$$|\Psi_{i,L-L_0-3}|\leq|\textbf{p}_{i,L-L_0-3}|+|\bar{\textbf{p}}_{i,L-L_0-3}|\leq4.21h_i$$
Since $\Psi_{i,L-L_0-3}$ contains no disks, Lemma \ref{M_a cubic} then gives
$$\text{wt}(\Psi_{i,L-L_0-3})\leq75N_1h_i^3$$
What's more, since $|\partial\Pi|<L|\bar{\textbf{p}}_{i,L-L_0-3}|$, we have
$$\text{wt}(\Pi)=c_7|\partial\Pi|^2\leq5c_7L^2h_i^2$$
Combining these two inequalities, we have
$$\text{wt}(\bar{\Delta}_{i,L-L_0-3})\leq75N_1h_i^3+5c_7L^2h_i^2<(75N_1+5c_7L^2)h_i^2$$
by Lemma \ref{big first h_i}. What's more, we have
$$\text{wt}(\Delta)\leq N_2\left(|\partial\Delta|+\sigma_\lambda(\Delta^*)-\frac{\delta h_{i+1}}{100c_6N}\right)^3+(75N_1+5c_7L^2)h_i^2$$
So, we have reached a contradiction if the following expression is at most 0:
\begin{align*}
3N_2(|\partial\Delta|+\sigma_\lambda(\Delta^*))\bigg(\frac{\delta h_{i+1}}{100c_6N}&\bigg)\bigg(\frac{\delta h_{i+1}}{100c_6N}-|\partial\Delta|-\sigma_\lambda(\Delta^*)\bigg)-N_2\bigg(\frac{\delta h_{i+1}}{100c_6N}\bigg)^3 +(75N_1+5c_7L^2)h_i^2
\end{align*}
As $h_{i+1}<|\partial\Delta|$ by Lemmas \ref{lengths} and \ref{G_a theta-annuli}, we can take
$$\frac{\delta h_{i+1}}{100c_6N}-|\partial\Delta|\leq-2|\partial\Delta|/3$$
it suffices to show the inequality:
$$2N_2|\partial\Delta|^2\frac{\delta h_{i+1}}{100c_6N}\geq(75N_1+5c_7L^2)h_i^2$$
By Lemma \ref{big first h_i}, $h_{i+1}>\delta^{-1}$, so that it suffices to show that 
$$N_2|\partial\Delta|^2\geq50c_6N(75N_1+5c_7L^2)h_i^2$$
By Lemma \ref{theta-bands in cloves}, each $\theta$-band of $\Psi$ starting on $\pazocal{Q}_i$ ends on $\textbf{p}$, so that it contributes at least one to the length of $|\partial\Delta|$. Hence, $h_i\leq|\partial\Delta|$, so that it suffices to show that
$$N_2\geq50c_6N(75N_1+5c_7L^2)$$
But this amounts to a parameter choices for $N_2$, so that we have our contradiction.

\end{proof}

Finally, we reach our desired contradiction, the analogue of Lemma 7.41 of [26].

\begin{lemma} \label{contradiction}

The counterexample diagram $\Delta$ cannot exist.

\end{lemma}

\begin{proof}

First, fix an integer $\eta$ dependant on $c_6$ and $N$ such that $(1-\frac{1}{20c_6N})^\eta<\frac{1}{6c_6N}$. Note that, although $\eta$ is not listed as one of the parameters of Section 2.4, since we choose $L_0$ after $c_6$ and $N$, we can take $L_0>>\eta$.

For $\omega\leq i\leq\tau-1$, Lemma \ref{first h_i vs h_{i+1}} gives us that $h_{i+1}<(1-\frac{1}{20c_6N})h_i$. So, if $\omega\leq i<j\leq \tau-1$ with $j-i-1\geq\eta$, then $h_j<(1-\frac{1}{20c_6N})^\eta h_{i+1}<\frac{1}{6c_6N}h_{i+1}$, i.e $h_{i+1}>6c_6N h_j$.

Lemma \ref{trapezia in first h_i} implies that $|\textbf{z}_i|_a\geq h_{i+1}/c_6>6Nh_j$; Lemma \ref{upper bound on z_i} then gives $2Nh_j\geq|\textbf{z}_j|_a$, so that $|\textbf{z}_i|_a>3|\textbf{z}_j|_a$.

Now, assuming $L_0>30\eta$, we obtain indices $\omega+1\leq j_1<j_2<\dots<j_{30}\leq\tau-2$ such that $j_i-j_{i-1}-1\geq\eta$, so that $|\textbf{z}_{j_{i-1}}|_a>3|\textbf{z}_{j_i}|_a$ and $h_{j_{i-1}}\geq6c_6Nh_{j_i}$.

Let $\pazocal{C}:W_0\to\dots\to W_t$ be the computation corresponding to the trapezium $\Gamma_{j_2}$ by Lemma \ref{trapezia are computations}. As it contains a copy of $\Gamma'_{j_2+1}$, which in turn contains a copy of $\Gamma_{j_2+2}$ and so on, there exist words $W(l)$ in $\pazocal{C}$ for $l=1,\dots,$ that are coordinate shifts of the labels of $\textbf{z}_{j_l}$. By the inequalities above, $|W(l+1)|_a>3|W(l)|_a$.

If for some $l$ the subcomputation $W(l+2)\to\dots\to W(l)$ is a one-step computation, then Lemma \ref{M one-step} implies that the subcomputation $W(l+1)\to\dots\to W(l)$ and there exists a sector for which a letter is inserted on the left or right increasing the sector's length. But since this subcomputation has length at least $\eta+1\geq2$, it follows that we can find a subcomputation contradicting Lemma \ref{no one-step}. 

So, the step history of the computation $\pazocal{C}$ must have length at least 10, so that Lemma \ref{long step history} limies that its step history contains a subword $(34)_i(4)_i(45)_i$, $(54)_i(4)_i(43)_i$, $(12)_i(2)_i(23)_i$, or $(32)_i(2)_i(21)_i$.

Suppose first that the subword is $(34)_i(4)_i(45)_i$ or $(54)_i(4)_i(43)_i$. Then we can factor $H_{j_2+1}$ as $H'H''H'''$ where $(H'')^{\pm1}$ is of the form $\chi(i-1,i)H_0\chi(i,i+1)$ for $i$ chosen such that this is a controlled history. Further, since $k\geq2$, Lemma \ref{M_3 standard base}$(b)$ tells us we can choose $i$ so that $\|H''\|\leq\|H'\|$. Finally, noting that $\pazocal{C}$ corresponds to a subtrapezium of $\Gamma_{j_2}\setminus\Gamma_\tau$, we can assume that $\|H'\|\geq\|H_\tau\|\geq h$.

Now, since $h_{j_1}>6c_6Nh_{j_2}>2h_{j_2}$, the history $H_{j_1+1}$ of $\Gamma_{j_1}$ has prefix $H'H''H'''_1$ where $\|H'''_1\|=\|H'\|\geq\|H''\|$. Set $\pazocal{C}$ as the subband of the spoke $\pazocal{Q}_{j_1+1}$ with this history. Then, for any factorization $\pazocal{C}=\pazocal{C}_1\pazocal{C}_2\pazocal{C}_3$, with $\|\pazocal{C}_1\|+\pazocal{C}_2\|\leq\|\pazocal{C}\|/3$, the history of $\pazocal{C}_2$ must contain $H''$. So, taking $\lambda<1/3$, $\pazocal{C}$ is a $\lambda$-shaft with length at least $h$, contradicting Lemma \ref{lambda-shafts in first h_i}.

Conversely, suppose that the subword is $(12)_i(2)_i(23)_i$ or $(32)_i(2)_i(21)_i$. Then factor $H_{j_2+1}$ as $H'H''H'''$ with $H''$ of the form $\zeta^{(i-1,i)}H_0\zeta^{(i,i+1)}$ for appropriate $i$. Lemma \ref{primitive computations} then implies that we can choose this factorization so that $\|H'\|\geq\|H''\|$, so that again we obtain a $\lambda$-shaft with length at least $h$.

Thus, we have reached the desired contradiction.

\end{proof}

\medskip


\section{Proof of Theorem \ref{main theorem}}

\begin{lemma} \label{embedding}

The group $B(2,n)$ embeds in the group $G(\textbf{M})$.

\end{lemma}

\begin{proof}

Consider the natural map $\varphi:\pazocal{A}\to G_a(\textbf{M})$ sending the elements of $\pazocal{A}$ to their copies in the tape alphabet of the `special' input sector. The theorem of von Dyck implies that this extends to a homomorphism $\varphi:B(2,n)\to G_a(\textbf{M})$.

Now suppose the word $w\in F(\pazocal{A})$ satisfies $\varphi(w)=1$. Then there exists a minimal diagram $\Delta$ over $G_a(\textbf{M})$ satisfying $\text{Lab}(\partial\Delta)\equiv w$. By Lemmas \ref{graph} and \ref{G_a theta-annuli}, every cell of $\Delta$ must be an $a$-cell. But then this is a diagram over $B(2,n)$, so that $w=1$ in $F(\pazocal{A})$.

So, $\varphi:B(2,n)\to G_a(\textbf{M})$ is an embedding. Lemma \ref{G isomorphic to G_a} then implies the statement.

\end{proof}

The next two lemmas are fundamental results relating to the Dehn functions of products of groups proved in [4]:

\begin{lemma} \label{product upper}

\textit{(Proposition 2.1 of [4])} If $G$ and $H$ are finitely presented groups, then: $$\delta_{G\times H}\preccurlyeq n^2+\delta_G+\delta_H$$

\end{lemma}

\begin{lemma} \label{product lower}

\textit{(Corollary 2.3 of [4])} If $G$ and $H$ are finitely presented groups, then $\delta_G\preccurlyeq\delta_{G\times H}$ and $\delta_H\preccurlyeq\delta_{G\times H}$.

\end{lemma}

Now, let $G_n=G(\textbf{M})\times H_3\Z$, where $H_3\Z$ is the discrete Heisenberg group, i.e the subgroup of $SL_3\Z$ given by matrices of the form $\begin{pmatrix}
1 & a & b \\
0 & 1 & c \\
0 & 0 & 1
\end{pmatrix}$. Lemma \ref{product lower} immediately implies that $G_n$ has at least cubic Dehn function, as $H_3\Z$ has cubic Dehn function [6].

So, as $G(\textbf{M})\xhookrightarrow{ \ } G_n$, to prove Theorem \ref{main theorem} it suffices to show that $G_n$ has at most cubic Dehn function. In fact, Lemma \ref{product upper} moreover implies that it suffices to show that $G(\textbf{M})$ has at most cubic Dehn function.

Let $w\in F(\pazocal{X})$ such that $w=1$ in $G(\textbf{M})$. By Lemma \ref{G isomorphic to G_a}, $w$ is also trivial over the group $G_a(\textbf{M})$, so that we can find a minimal diagram $\Delta_a$ over $G_a(\textbf{M})$ with $\text{Lab}(\partial\Delta_a)\equiv w$. By Lemma \ref{contradiction}, we have
$$\text{wt}(\Delta_a)\leq N_2(|w|+\sigma_\lambda(\Delta_a^*))^3$$
By Lemma \ref{G_a design}, $\sigma_\lambda(\Delta_a^*)\leq c|w|$. Further, by the modified definition of length, $|w|\leq\|w\|$. As a result, we can choose $N_3$ large enough so that
$$\text{wt}(\Delta_a)\leq N_3\|w\|^3$$
Now, by Lemmas \ref{disks are quadratic} and \ref{a-cells are quadratic}, from $\Delta_a$ we can construct the diagram $\Delta$ over the canonical presentation of $G(\textbf{M})$ by:

\begin{addmargin}[1em]{0em}

$\bullet$ excising any disk $\Pi$ and pasting in its place a diagram over $G(\textbf{M})$ with the same contour label and area at most $c_7|\partial\Pi|^2$

$\bullet$ excising any $a$-cell $\pi$ and pasting in its place a diagram  over $G(\textbf{M})$ with the same contour label and area at most $c_7\|\partial\pi\|^2$

\end{addmargin}

By the definition of $\text{wt}$, it follows that $\text{Area}(\Delta)\leq\text{wt}(\Delta_a)\leq N_3\|w\|^3$, and thus Theorem \ref{main theorem} is proved.

\bigskip


\section{References}

[1] S. I. Adian, \textit{The Burnside Problem and Identities in Groups}, Springer-Verlag, (1979).

[2] J.-C. Birget, A. Yu. Olshanskii, E. Rips, M. Sapir, \textit{Isoperimetric functions of groups and combinatorial complexity of the word problem}, Annals of Mathematics , 156 (2002), no. 2, 467–518.

[3] B. H. Bowditch, \textit{Notes on Gromov’s hyperbolicity criterion for path-metric spaces}, "Group theory from a geometrical viewpoint (Trieste, 1990)", (E Ghys, A Haefliger, A Verjovsky, editors), World Sci. Publ., River Edge, NJ (1991)

[4] S. G. Brick, \textit{Dehn functions of groups and products of groups}, Transactions of the American Mathematical Society. 335. 369-384, (1993).

[5] A. Darbinyan, \textit{Word and conjugacy problems in lacunary hyperbolic groups}, (2017).

[6] D. B. A. Epstein, J. W. Cannon, S. V. F. Levy, M. S. Paterson, W. P. Thurston, \textit{Word Processing in Groups}, Jones and Bartlett, Boston, (1992).

[7] E. Ghys, P. de la Harpe, \textit{Sur les Groupes Hyperboliques d'apr\`{e}s Mikhael Gromov}, Springer, 1990.

[8] E. S. Golod, I. R. Shafarevich, \textit{On the class field tower}, Izv. Akad. Nauk SSSR Ser. Mat., 28:2 (1964), 261–272 

[9] M. Gromov, \textit{Hyperbolic groups}, Essays in Group Theory (S.M.Gersten, ed.), M.S.R.I. Pub. 8, Springer, (1987), 75–263.

[10] M. Gromov, \textit{Asymptotic invariants of infinite groups}, in: Geometric Group Theory. Vol. 2 (G.A.Niblo and M.A.Roller, eds.), London Math. Soc. Lecture Notes Ser., 182 (1993), 1–295.

[11] S. V. Ivanov, \textit{On subgroups of free Burnside groups of large odd exponent}. Illinois J. Math. 47 (2003), no. 1-2, 299--304. 

[12] S. V. Ivanov, \textit{Embedding free Burnside groups in finitely presented groups}, Geometriae Dedicata, (2005), vol. 111, no. 1, pp. 87-105.

[13] R. C. Lyndon and P. E. Schupp, \textit{Combinatorial group theory}, Springer-Verlag, 1977.

[14] P. S. Novikov, S. I. Adian, \textit{Defining relations and the word problem for free periodic groups of odd order}, Izv. Akad. Nauk SSSR Ser. Mat., 32:4 (1968)

[15] A. Yu. Ol'shanskii, \textit{Groups of bounded period with subgroups of prime order}, Algebra and Logic 21 (1983), 369–418; translation of Algebra i Logika 21 (1982)

[16] A. Yu. Ol'shanskii, \textit{Hyperbolicity of groups with subquadratic isoperimetric inequality} Internat. J. Algebra Comput. 1 (1991), no. 3, 281–289.

[17] A. Yu. Ol'shanskii, \textit{Geometry of Defining Relations in Groups}, Springer Netherlands, (1991)

[18] A. Yu. Ol'shanskii, \textit{On subgroup distortion in finitely presented groups} Mat. Sb., 188:11 (1997), 51–98; Sb. Math., 188:11 (1997), 1617–1664 

[19] A. Yu. Ol'shanskii, \textit{Polynomially-bounded Dehn functions of groups}, Journal of Combinatorial Algebra, 2. (2018) 311-433

[20] A. Yu. Ol'shanskii, M. V. Sapir, \textit{Embeddings of relatively free groups into finitely presented groups}, (2000)

[21] A. Yu. Ol'shanskii, M. V. Sapir, \textit{Length and area functions in groups and quasiisometric Higman embeddings}, Intern. J. Algebra and Comput., 11 (2001), no. 2, 137–170.

[22] A. Yu. Ol'shanskii, M. V. Sapir, \textit{The Conjugacy Problem and Higman Embeddings}. Memoirs of the American Mathematical Society. 170. (2003). 

[23] A. Yu. Ol’shanskii, M. V. Sapir, \textit{Non-Amenable Finitely Presented Torsion-by-Cyclic Groups}, Publ. math., Inst. Hautes Étud. Sci. (2003)

[24] A. Yu. Ol’shanskii, M. V. Sapir, \textit{Groups with Small Dehn functions and Bipartite Chord Diagrams} GAFA, Geom. funct. anal. 16 (2006), 1324

[25] A. Yu. Ol'shanskii, M. V. Sapir, \textit{Groups with undecidable word problem and almost quadratic Dehn function}, Journal of Topology. 5. (2012) 785-886. 10.1112/jtopol/jts020

[26] A. Yu. Ol'shanskii, M. V. Sapir, \textit{Conjugacy problem in groups with quadratic Dehn function}, (2018). 

[27] M. V. Sapir, J. C. Birget, E. Rips, \textit{Isoperimetric and Isodiametric Functions of Groups}, Annals of Mathematics, 156(2), second series, (1998), 345-466 

[28] M. V. Sapir, \textit{Combinatorial algebra: Syntax and Semantics. With contributions by Victor S. Guba and Mikhail V. Volkov}, Springer Monographs in Mathematics, Springer, Cham, (2014).

[29] V. L. Shirvanyan, \textit{Embedding the group $B(\infty,n)$ in the group $B(2,n)$}, Izv. Akad. Nauk SSR Ser. Mat. 40 (1976), 190–208.

[30] E. van Kampen, \textit{On Some Lemmas in the Theory of Groups}, American Journal of Mathematics
Vol. 55, No. 1 (1933), pp. 268-273.

\end{document}